\documentclass[11pt]{article}
\usepackage[left=1in, right=1in, top=1in]{geometry}
\usepackage{xargs}
\usepackage[numbers]{natbib}
\usepackage{fancyhdr}
\usepackage{setspace}
\usepackage{lastpage}
\usepackage{upgreek}
\usepackage[american]{babel}
\usepackage[utf8]{inputenc}
\usepackage[T1]{fontenc}
\usepackage{amsmath,mathtools,amsthm,amsfonts,amssymb}
\usepackage{dsfont,bbm}
\usepackage{graphicx}
\usepackage{subcaption}
\usepackage{placeins}
\usepackage{booktabs}
\usepackage{nicefrac}
\usepackage{microtype}
\usepackage{xcolor}
\usepackage{longtable}
\usepackage{comment}
\usepackage{enumitem}
\usepackage{multirow}
\usepackage{bm}
\usepackage{mathrsfs}
\usepackage{xurl}

\setcitestyle{number}

\usepackage[colorlinks=true,breaklinks=true,bookmarks=true,urlcolor=blue,
  citecolor=blue,linkcolor=blue,bookmarksopen=false,draft=false]{hyperref}
\usepackage{aliascnt}
\usepackage[nameinlink,capitalize,noabbrev]{cleveref}

\newcommand{\abs}[1]{\left\vert #1 \right\vert}
\newcommandx{\norm}[2][2=]{\Vert#1 \Vert_{{#2}}}
\newcommand{\supnorm}[1]{\norm{ #1 }[\infty]}
\newcommand{\lzeronorm}[1]{\Vert #1 \Vert_{0}}
\newcommand{\lr}[1]{\left( #1 \right)}
\newcommand{\lrcb}[1]{\left\{ #1 \right\}}
\newcommand{\indi}[1]{\mathbbm{1}_{#1}}
\newcommand{\indiacc}[1]{\mathbbm{1}_{\{#1\}}}
\newcommand{\ceil}[1]{\lceil #1 \rceil}
\newcommand{\kbeta}{k_\beta}
\newcommand{\supp}{\text{supp}}
\DeclareMathOperator*{\argmin}{arg\,min}

\def\eqsp{\,}
\def\rmd{\mathrm{d}}
\def\rme{\mathrm{e}}

\newcommand{\E}{\mathbb{E}}
\newcommand{\PP}{\mathbb{P}}

\newcommand{\PVar}{\mathsf{Var}}
\newcommand{\R}{\mathbb{R}}
\newcommand{\X}{\mathcal{X}}
\newcommand{\Y}{\mathcal{Y}}
\def\N{\mathbb{N}}
\def\XC{\mathcal{X}}
\def\Xset{\XC}
\def\HC{\mathcal{H}}
\def\FC{\mathcal{F}}
\def\SC{\mathcal{S}}
\def\WC{\mathcal{W}}
\newcommandx{\DC}[1][1=]{P_{#1}}
\newcommandx{\QC}[1][1=]{Q_{#1}}
\newcommandx{\Lone}[2][2=P_X]{\|#1\|_{L^1(#2)}}
\newcommandx{\Ltwo}[2][2=P_X]{\|#1\|_{L^2(#2)}}
\newcommandx{\Lp}[2][1=\Lambda]{\|#2\|_{L^p(#1)}}

\usepackage{aliascnt}
\usepackage{cleveref}
\makeatletter
\@ifundefined{theorem}{%
  \newtheorem{theorem}{Theorem}
  \newaliascnt{lemma}{theorem}
  \newtheorem{lemma}[lemma]{Lemma}
  \aliascntresetthe{lemma}
  \newaliascnt{proposition}{theorem}
  \newtheorem{proposition}[proposition]{Proposition}
  \aliascntresetthe{proposition}
  \newaliascnt{corollary}{theorem}
  \newtheorem{corollary}[corollary]{Corollary}
  \aliascntresetthe{corollary}
  \newaliascnt{definition}{theorem}
  
  \aliascntresetthe{definition}
  \newaliascnt{remark}{theorem}
  \newtheorem{remark}[remark]{Remark}
  \aliascntresetthe{remark}
}{}
\makeatother
\crefname{theorem}{Theorem}{Theorems}
\crefname{lemma}{Lemma}{Lemmas}
\crefname{proposition}{Proposition}{Propositions}
\crefname{corollary}{Corollary}{Corollaries}
\crefname{definition}{Definition}{Definitions}
\crefname{remark}{Remark}{Remarks}

\newtheorem{assum}{\textbf{A}\hspace{-2pt}}
\crefname{assum}{\textbf{A}\hspace{-2pt}}{\textbf{A}\hspace{-2pt}}
\Crefname{assum}{\textbf{A}\hspace{-2pt}}{\textbf{A}\hspace{-2pt}}
\newtheorem{assumB}{\textbf{B}\hspace{-2pt}}
\crefname{assumB}{\textbf{B}\hspace{-2pt}}{\textbf{B}\hspace{-2pt}}
\Crefname{assumB}{\textbf{B}\hspace{-2pt}}{\textbf{B}\hspace{-2pt}}
\newtheorem{assumL}{\textbf{L}\hspace{-2pt}}
\crefname{assumL}{\textbf{L}}{\textbf{L}}
\Crefname{assumL}{\textbf{L}\hspace{-2pt}}{\textbf{L}\hspace{-2pt}}

\def\low{\mu_{\operatorname{low}}}
\def\up{\mu_{\operatorname{up}}}
\newcommand{\ahi}{1-\alpha/2}
\newcommand{\alo}{\alpha/2}
\newcommand{\quant}[2]{\operatorname{Q}(#1;#2)}
\newcommand{\TM}{T_M}
\newcommandx{\fhat}[2][1=n]{\hat{f}_{#1,#2}}
\newcommand{\fstar}[1]{f^\star_{#1}}
\newcommand{\Fhat}{\widehat F}
\newcommand{\Sstar}{S^\star}
\newcommand{\Shat}{\widehat S}

\newcommandx{\CCstar}[1][1=\alpha]{\mathcal{C}_{#1}^{\star}}
\newcommandx{\CChat}[1][1=\alpha]{\hat{\mathcal{C}}_{#1}}
\newcommand{\Qhat}[1]{\widehat{Q}_{{#1}}}
\newcommand{\betam}{\beta_m}
\newcommandx{\Dtrain}[1][1=n]{\mathcal{D}_{#1}^{\operatorname{tr}}}
\newcommandx{\Dcal}[1][1=m]{\mathcal{D}_{#1}^{\operatorname{cal}}}
\newcommand{\Eunif}{\mathcal{E}_{\mathrm{unif}}}

\newcommand{\pinball}{\rho_\tau}
\newcommandx{\Risk}[2][1=\tau]{R_{#1}(#2)}
\newcommandx{\EmpRisk}[2][1=\tau]{R_{n,#1}(#2)}
\newcommandx{\approxerror}[1][1=\tau]{\mathcal{A}_{#1}}
\def\param{\theta}
\def\Param{\Theta}
\newcommandx{\hparam}[2][1=n,2=\tau]{{\widehat{\param}_{#1,#2}}}
\newcommand{\func}[1]{f^{{#1}}}
\newcommandx{\bestfinclass}[1][1=\tau]{\tilde{f}_{#1}}
\newcommandx{\bestparam}[1][1=\tau]{\bar{\param}_{#1}}
\newcommand{\relu}{\sigma}
\newcommand{\bgamma}{{\boldsymbol{\gamma}}}
\def\bp{\mathbf{p}}
\def\lipconst{L}
\newcommand{\NNclass}{\mathcal{F}_{L,\WC,\SC,M,1}}
\newcommand{\NNclassSH}{\FC(L,\bp,s,M)}
\newcommand{\NNclassSHn}[3]{\FC(#1, #2, #3, M)}
\newcommand{\metricinfty}{\Vert\cdot\Vert_{\infty}}
\newcommand{\logcover}[2]{\log \mathcal{N}\!\left(#1, \; #2, \; \metricinfty \right)}
\newcommand{\cover}[2]{\mathcal{N}\!\left(#1, \; #2, \; \metricinfty \right)}
\newcommand{\Capp}{C_{\mathrm{app}}}
\newcommand{\fone}{f^{(1)}}
\newcommand{\ftwo}{f^{(2)}}
\newcommand{\Wone}{W^{(1)}}
\newcommand{\Wtwo}{W^{(2)}}
\newcommand{\vone}{v^{(1)}}
\newcommand{\vtwo}{v^{(2)}}
\def\eps{\varepsilon}

\newcommand{\wsh}{w}
\newcommand{\wmax}{w_{\max}}
\newcommand{\csdiv}{\bigl(1+\chi^2(\QC[X]\Vert\DC[X])\bigr)}
\newcommand{\Fhatw}{\widehat{F}^{\,w}_{m}}
\newcommand{\Qhatw}{\widehat{Q}^{\,w}_{m}}
\newcommandx{\CChatw}[1][1=\alpha]{\widehat{\mathcal{C}}^{\,w}_{#1}}
\newcommand{\etaw}[1]{\eta^{w}_{m}(#1)}
\newcommand{\FQS}{F^{Q}_{\Shat}}
\newcommand{\FQSstar}{F^{Q}_{\Sstar}}

\makeatletter
\def\rep@title{Restatement}
\newtheorem*{rep@theorem}{\rep@title}
\newcommand{\newreptheorem}[2]{%
  \newenvironment{rep#1}[2][]{%
    \def\rep@title{\texorpdfstring{#2~\ref{##2}%
      \if\relax\detokenize{##1}\relax\else\space(##1)\fi}{#2}}%
    \begin{rep@theorem}}
  {\end{rep@theorem}}}
\makeatother

\newreptheorem{theorem}{Theorem}
\newreptheorem{proposition}{Proposition}

\newcommand{\appendixcontextname}{appendix}

\hypersetup{
  pdftitle={Conformalized Quantile Regression and Minimax Limits of Fixed-Score Calibration under Known Covariate Shift},
  pdfauthor={Rustam Isaev, Anton Conrad, Denis Belomestny, Eric Moulines, Sergey Samsonov},
}

\title{Conformalized Quantile Regression and Minimax Limits of\\
Fixed-Score Calibration under Known Covariate Shift}

\author{%
  Rustam Isaev\textsuperscript{1,2}\thanks{Corresponding author: \texttt{risaev@hse.ru}}
  \and Anton Conrad\textsuperscript{3}
  \and Denis Belomestny\textsuperscript{1,4}
  \and Eric Moulines\textsuperscript{3,5}
  \and Sergey Samsonov\textsuperscript{1}}
\date{}

\begin{document}

\emergencystretch=3em

\maketitle

\begingroup
\footnotesize\noindent
\textsuperscript{1}Faculty of Computer Science, HSE University,
Moscow, Russia\\
\textsuperscript{2}Faculty of Computational Mathematics and Cybernetics,
Lomonosov Moscow State University, Moscow, Russia\\
\textsuperscript{3}Laboratoire de Recherche d'EPITA,
Le Kremlin-Bic\^etre, France\\
\textsuperscript{4}Faculty of Mathematics,
University of Duisburg-Essen, Essen, Germany\\
\textsuperscript{5}Computing and Mathematical Sciences Division,
Mohamed bin Zayed University of Artificial Intelligence (MBZUAI),
Abu Dhabi, United Arab Emirates
\endgroup

\begin{abstract}
In this paper, we study nonasymptotic $L^p$ error bounds for interval length and conditional coverage in split conformalized quantile regression (CQR).
Our bounds rely on local regularity conditions and accuracy guarantees for the estimated quantiles.
We further instantiate our bounds for quantile regression with sparse ReLU neural networks.
We also consider covariate shift, where the calibration and test covariates have different distributions, and derive nonasymptotic bounds for this setting.
We obtain matching minimax upper and lower bounds in expectation for two constructed fixed-score calibration benchmarks under known covariate shift.
The bounds match for every $p\in[1,\infty]$ in the scalar problem and for finite $p$ in the $K$-threshold problem; for the latter, a high-probability minimax lower bound holds for every $p\in[1,\infty]$.

\end{abstract}


\smallskip
\noindent\textbf{MSC2020 subject classifications:} 62G08, 62G15, 62C20.

\section{Introduction}
\label{sec:intro}

Conformal prediction converts a fitted score into a prediction set with
finite-sample marginal coverage under exchangeability
\cite{vovk2005algorithmic,shafer2008tutorial,lei2018distributionfree}.
Unless stated otherwise, training, calibration, and test pairs follow
$\DC[XY]\coloneqq\DC[X]\otimes\DC[Y\mid X]$.
For a target miscoverage level $\alpha\in(0,1)$ and an independent test
observation $(X,Y)\sim\DC[XY]$, a conformal rule $\CChat$ constructed from
training and calibration samples satisfies
\[
 \PP\{Y\in\CChat(X)\}\ge 1-\alpha.
\]
This guarantee averages over the test covariate and over the data used to construct the set. Exact distribution-free conditional validity, that is,
\[
\PP\{Y\in\CChat(x)\mid X=x\} \geq 1-\alpha
\]
for almost every $x$, is unavailable for nontrivial rules without further assumptions, see \cite{vovk2012conditional,barber2021limits}. We therefore ask how, after the training and calibration samples have been fixed, the resulting random conditional coverage profile and interval length compare with an oracle benchmark.

We answer this question for split conformalized quantile regression (CQR)
\cite{papadopoulos2002inductive,romano2019conformalized}.
Split CQR fits the two conditional quantile endpoints on the training fold and
selects a common scalar score threshold on the independent calibration fold.
For $\tau\in(0,1)$, let $\fstar{\tau}(x)$ denote the conditional $\tau$-quantile
of $Y$ given $X=x$. The equal-tailed oracle interval is
\[
 \CCstar(x)=\bigl[\fstar{\alo}(x),\fstar{\ahi}(x)\bigr],
 \qquad
 \PP\{Y\in\CCstar(x)\mid X=x\}=1-\alpha,
\]
under the local mass conditions imposed below.
For the realized split CQR set $\CChat$ constructed from both folds, we control,
for every $p\in[1,\infty]$,
\[
 \Lp[{\DC[X]}]{|\CChat|-|\CCstar|},
 \qquad
 \Lp[{\DC[X]}]{\operatorname{Cov}_{\CChat}-(1-\alpha)},
 \qquad
 \operatorname{Cov}_{\CChat}(x)=\PP\{Y\in\CChat(x)\mid X=x\}.
\]
We prove finite sample bounds for both quantities that hold with high
probability jointly over the training and calibration folds and separate
endpoint estimation, score calibration, and rank correction. Their only
learner-specific input is the $L^p(\DC[X])$ error of the two fitted quantile
endpoints.

Covariate shift changes only the covariate law: training and calibration pairs
follow $\DC[X]\otimes\DC[Y\mid X]$, whereas the test pair follows
$\QC[X]\otimes\DC[Y\mid X]$. When
$\wsh=\rmd\QC[X]/\rmd\DC[X]$ is known, weighted split conformal prediction
retains finite sample target marginal coverage
\cite{tibshirani2019conformal}. Under the same local mass and endpoint
conditions and $\wsh\le\wmax$, our bounds additionally control the realized
target $L^p(\QC[X])$ conditional coverage profile and the length deviation
from $\CCstar$.

Conditional on the training fold, the fitted CQR score is fixed. We therefore
isolate calibration in a separate minimax problem: the source and target
marginals and a reference score are fixed, the conditional response kernel
varies, and only the threshold rule uses the $m$ source calibration
observations. The reference score is the oracle CQR score under one baseline
kernel and remains fixed across the candidate kernels.

On an explicit two-carrier construction with point-mass target
$\QC[X]=\delta_{x_0}$, only the threshold value at $x_0$ affects the target
risk. Hence every covariate-dependent threshold rule is equivalent, for this
experiment, to its scalar value at $x_0$. Let $\mathcal R_{m,p}$ denote the
infimum over these measurable scalar outputs of the worst case, over
admissible conditional response kernels, expected $L^p(\QC[X])$ coverage
profile error. For every $p\in[1,\infty]$ and $m\ge m_\star$,
\[
 c_\alpha\sqrt{\frac{\csdiv}{m}}
 \le \mathcal R_{m,p}
 \le \frac32\sqrt{\frac{\csdiv}{m}},
\]
so the exact weighted rule is minimax-rate optimal on this construction. A
$K$-atomic Fano construction with $K\ge23$ and $m>m_\star^{(K)}$ gives a
lower bound of order $\sqrt{\csdiv K/m}$ with probability at least
$1-2e^{-K/32}$; atomwise split calibration gives a matching moment upper
bound for finite $p$. The
single-atom construction has effective calibration size $m/\csdiv$, while
each target atom in the $K$-atomic construction receives on average
$m/(\csdiv K)$ calibration observations. The single-atom rate reproduces the
$\sqrt{\csdiv/m}$ calibration scaling of the CQR upper bounds. The $K$-atomic
rate instead quantifies the additional cost of estimating $K$ atom-specific
threshold values within the separate fixed-score benchmark.

Our contributions are:
\begin{itemize}
\item finite-sample high-probability $L^p(\DC[X])$ bounds for equal-tailed
oracle-length deviation and the realized conditional-coverage profile of CQR,
with an explicit sparse-ReLU instantiation of the endpoint condition
(\Cref{prop:target,th:global_cov_bound,theo:requ_rates,cor:relu_rates});
\item corresponding $L^p(\QC[X])$ guarantees under known bounded covariate
shift, with calibration rate $\sqrt{\csdiv/m}$
(\Cref{prop:cs_length,th:cs_cond_cov,cor:cs_relu});
\item under the stated sample-size conditions, matching expected bounds for
two constructed fixed-score calibration benchmarks: the single-carrier
experiment (\Cref{prop:cs_lecam_lower,prop:cs_carrier_upper}) and, for $K\ge23$
and every finite $p$, the $K$-atomic experiment, with rate
$\sqrt{\csdiv K/m}$ and a same-order lower bound holding with probability at
least $1-2e^{-K/32}$ for every $p\in[1,\infty]$
(\Cref{thm:cs_fano_lower,prop:cs_cellwise_upper}).
\end{itemize}

The numerical study evaluates the fixed-score carrier rules and illustrates
learned weighted CQR under a continuous shift; detailed carrier calculations
and figures are reported in \Cref{sec:numerical_carrier_details}.

\section{Related Work}
\label{sec:related_work}

\paragraph{CQR efficiency and conditional profiles}
Split conformal prediction gives finite-sample marginal coverage
\cite{papadopoulos2002inductive,vovk2005algorithmic,lei2018distributionfree},
whereas exact distribution-free conditional validity is generally impossible
\cite{vovk2012conditional,barber2021limits}. Romano et
al.~\cite{romano2019conformalized} introduced CQR, and Sesia and
Cand\`es~\cite{sesia2020comparison} established asymptotic efficiency under
consistent quantile estimation. Kivaranovic et
al.~\cite{kivaranovic2020adaptive} adapt the interval while retaining marginal
validity. Rossellini et al.~\cite{rossellini2024uncertainty} adapt the
conformal correction to estimated endpoint uncertainty; their finite-sample
result is marginal, and their conditional comparisons are empirical. Yao et
al.~\cite{yao2025non} bound expected CQR length error for a correctly specified
linear quantile model trained by SGD. Other efficiency results compare with a
shortest symmetric-residual oracle \cite{lebars2025volume} or estimate more of
the conditional law to target shortest or highest-density sets
\cite{chernozhukov2021distributional,izbicki2022cdsplit}.

Our bounds use the equal-tailed CQR oracle and control, with high probability
after both data folds are fixed, its length deviation together with the full
$L^p$ conditional-coverage profile. This conditioning convention is
important: the nonasymptotic decomposition of Min et
al.~\cite{min2026unified} averages conditional coverage over the procedure's
randomness. Localized calibration and score transformations instead pursue
pointwise or neighborhood guarantees
\cite{plassier2025rectifying,conrad2026localized}; their target and the
bias--variance tradeoff created by localization differ from the global
additive CQR correction studied here.

\paragraph{Covariate shift}
Tibshirani et al.~\cite{tibshirani2019conformal} proved finite-sample target
marginal validity for weighted conformal prediction with a known likelihood
ratio; Lei and Cand\`es~\cite{lei2021conformal} applied the construction to
counterfactual CQR. Pournaderi and
Xiang~\cite{pournaderi2026trainingconditional} control target marginal
miscoverage after the source sample is fixed, whereas we control the realized
target conditional-coverage profile and oracle length. In the unshifted
setting, Bian and Barber~\cite{bian2023trainingconditional} distinguish
training-conditional guarantees from ordinary marginal validity. PAC
prediction sets under shift can also be obtained by rejection sampling with
known or interval-estimated weights \cite{park2022pac}.

When the likelihood ratio is unknown, asymptotic target validity can be
obtained through nuisance estimation or doubly robust calibration
\cite{qiu2023prediction,yang2024doubly}. Direct covariate-dependent threshold
learning and weight clipping provide further alternatives
\cite{joshi2025lrqr,wang2026weightclipping}. Our finite-sample analysis assumes
a known uniformly bounded likelihood ratio.

\paragraph{Fixed-score calibration limits}
Fixed-score calibration is studied for covariate-dependent thresholds
\cite{areces2024two,duchi2025sampleconditional}, prescribed finite-dimensional
shift classes \cite{gibbs2025conditional}, and overlapping or fractional groups
\cite{kandinsky2025}. Related results give conditional-calibration oracle
inequalities for set-valued maps \cite{bao2026shapeadaptive} and characterize
continuous split calibration through transported beta laws
\cite{ramos2026transportedbeta}. \Cref{sec:calibration_lower_bounds} compares
our benchmarks with the threshold calibration results. Analogous
$\chi^2$-controlled rates for mean estimation from biased sources appear in
\cite{harding2026biased}.

\paragraph{Quantile endpoint estimation}
Sparse-ReLU quantile rates are developed by Madrid Padilla et
al.~\cite{madridPadilla2022quantile}, building on the approximation
architecture of Schmidt-Hieber~\cite{schmidthieber_2020} and its later
regression correction \cite{schmidthieber_vu_2024_correction}. In the
isotropic H\"older setting, inserting the resulting $L^2(\DC[X])$
endpoint-error bound into our general CQR result yields the $(\log n)^{3/2}$
factor in the final rate. Feng et
al.~\cite{feng2024deep} analyze expected target $L^2$ error for neural quantile
regression under known or estimated covariate shift; our ReLU result instead
supplies a high-probability source endpoint condition and then propagates it
through conformal calibration.

\section{Split CQR}
\label{sec:non-asymptotic-CQR}

We follow the split-CQR construction of Romano et al.~\cite{romano2019conformalized}.
Let $\Xset$ be an arbitrary measurable space, let $Y$ be real-valued, and let $x\mapsto\DC[Y\mid X=x]$ be a probability kernel.
Set
\[
\DC[XY]=\DC[X]\otimes\DC[Y\mid X].
\]
For any nondecreasing, right-continuous function $H\colon\R\to[0,\infty)$ with $\lim_{t\to-\infty}H(t)=0$, and any $\tau>0$, define its generalized inverse by
\[
\quant{\tau}{H}\coloneqq\inf\{t\in\R:H(t)\ge\tau\},
\qquad \inf\varnothing\coloneqq+\infty.
\]
Write $\overline\R=\R\cup\{-\infty,+\infty\}$ with its order-Borel
$\sigma$-field, and extend every probability CDF $F$ by $F(-\infty)=0$ and
$F(+\infty)=1$.
Weighted empirical cumulative functions retain their realized total mass; an unattained level has generalized inverse $+\infty$.
For $\tau\in(0,1)$, the conditional $\tau$-quantile is
\begin{equation}
\label{eq:quantile_regr}
\fstar{\tau}(x) = \inf\{t\in\R\colon \PP(Y\le t\mid X=x)\ge \tau\}\eqsp.
\end{equation}
Fix $\alpha\in(0,1)$. The benchmark for interval length is the equal-tailed
conditional-quantile interval
\begin{equation}
\label{eq:oracle_interval}
\CCstar(x)\coloneqq
\bigl[\fstar{\alo}(x),\fstar{\ahi}(x)\bigr].
\end{equation}

\begin{assum}
\label{assum:mass_regular}
There are constants $0<\low\le\up<\infty$ and a local mass radius $r_0>0$, and one measurable set
$\mathcal X_0$ with $\DC[X](\mathcal X_0)=1$ such that, for every
$x\in\mathcal X_0$,
\begin{equation}
\label{eq:upper_mass_bound}
\DC[Y\mid X=x]([u,v])\le\up(v-u)
\quad\text{for all finite }u\le v,
\end{equation}
and, for $\tau\in\{\alo,\ahi\}$ and $0\le h\le r_0$,
\[
\DC[Y\mid X=x]([\fstar{\tau}(x)-h,\fstar{\tau}(x)])\ge\low h,
\qquad
\DC[Y\mid X=x]([\fstar{\tau}(x),\fstar{\tau}(x)+h])\ge\low h.
\]
\end{assum}
The upper bound in \Cref{assum:mass_regular} makes the conditional CDFs
$\up$-Lipschitz and atomless; hence, for every $x\in\mathcal X_0$,
$\DC[Y\mid X=x](({-\infty},\fstar{\alo}(x)])=\alo$ and
$\DC[Y\mid X=x](({-\infty},\fstar{\ahi}(x)])=\ahi$.
Thus endpoint error controls score-CDF and conditional-coverage error. The
local lower bounds provide two-sided linear
growth at the oracle quantiles, which is used to convert score-CDF error into
threshold error. Conditions of this form appear in
\cite[Definition~2.1]{steinwart2011pinball} and
\cite[Assumption~2]{madridPadilla2022quantile}. Our condition is weaker than the
continuous two-sided density bounds of \cite[Assumption~3.3]{yao2025non}: it
requires lower growth only near the two oracle quantiles and allows unbounded
response support and discontinuous conditional densities.

Let $\Dtrain\sim\DC[XY]^{\otimes n}$ and
$\Dcal\sim\DC[XY]^{\otimes m}$ be independent.
The training sample produces finite raw maps $(\Dtrain,x)\mapsto\widetilde f_{n,\tau}(x)$, jointly measurable for $\tau\in\{\alo,\ahi\}$.

\begingroup
\renewcommand{\theassum}{\arabic{assum}($p$)}
\begin{assum}
\label{assum:HPD-excess-risk}
Fix $p\in[1,\infty]$. There is
$\epsilon_p\colon\N\times(0,1)\to\R_+$ such that, for every $n\in\N$,
$\delta\in(0,1)$ and $\tau\in\{\alo,\ahi\}$,
\[
\PP_{\Dtrain}\!\left(
\Lp[{\DC[X]}]{\widetilde f_{n,\tau}-\fstar{\tau}}
\le\epsilon_p(n,\delta)
\right)\ge1-\delta.
\]
\end{assum}
\endgroup
\Cref{assum:HPD-excess-risk} is the only condition imposed on the endpoint
learner. Its $L^p(\DC[X])$ guarantee controls the spatial error terms and, since
$p\geq 1$, yields the $L^1(\DC[X])$ bound needed to compare the estimated and
oracle score distributions. Its high-probability form allows endpoint
estimation and calibration errors to be combined in a single guarantee.
Accordingly, the results below apply to any learner satisfying this integrated
error bound.

Define the fitted endpoints by pointwise sorting,
\begin{equation}
\label{eq:raw_endpoint_sorting}
\fhat{\alo}=\widetilde f_{n,\alo}\wedge\widetilde f_{n,\ahi},
\qquad
\fhat{\ahi}=\widetilde f_{n,\alo}\vee\widetilde f_{n,\ahi}.
\end{equation}
Since the oracle endpoints are ordered, \Cref{lem:sorting_contraction_mass}
gives, pointwise,
\[
|\fhat{\alo}-\fstar{\alo}|+|\fhat{\ahi}-\fstar{\ahi}|
\le
|\widetilde f_{n,\alo}-\fstar{\alo}|
+|\widetilde f_{n,\ahi}-\fstar{\ahi}|.
\]
Thus \Cref{assum:HPD-excess-risk} controls the combined endpoint error after
sorting; all scores and prediction intervals below use $\fhat{\alo}$ and
$\fhat{\ahi}$.

\paragraph{Calibration rule}
Following Romano et al.~\cite{romano2019conformalized}, we calibrate the fitted CQR score by a common scalar empirical quantile and use the resulting score sublevel set as the prediction interval.
By \Cref{assum:mass_regular}, the oracle interval
$\CCstar(x)$ has conditional coverage
$1-\alpha$ for $\DC[X]$-almost every $x$. Define the oracle and fitted scores
\[
\Sstar(x,y) \coloneqq \max\{\fstar{\alo}(x)-y,\, y-\fstar{\ahi}(x)\},\qquad
\Shat(x,y) \coloneqq \max\{\fhat{\alo}(x)-y,\, y-\fhat{\ahi}(x)\}.
\]
For fixed $x$, $\Shat(x,y)$ is the smallest scalar $t$ such that
$y\in[\fhat{\alo}(x)-t,\fhat{\ahi}(x)+t]$; calibration therefore selects one
common expansion or contraction across covariates.
Write $\Dcal=\{(X_i,Y_i)\}_{i=1}^m$ and $\Shat_i=\Shat(X_i,Y_i)$. Define the
empirical score CDF by
\[
\Fhat^{(\Shat)}_m(t)
\coloneqq \frac{1}{m}\sum_{i=1}^m\indiacc{\Shat_i\le t}.
\]
Define the split-conformal level and calibrated threshold by
\begin{equation}
\label{eq:beta_m_def}
\betam \coloneqq \frac{\lceil (m+1)(1-\alpha)\rceil}{m},
\qquad
\Qhat{\betam} \coloneqq \quant{\betam}{\Fhat^{(\Shat)}_m}.
\end{equation}
The index $\lceil(m+1)(1-\alpha)\rceil$ is the standard finite sample split conformal rank correction.
The split CQR interval is
\begin{equation}
\label{eq:split-conformal-eq-main}
\CChat(x)\coloneqq\{y\in\R:\Shat(x,y)\le\Qhat{\betam}\}.
\end{equation}
If $\Qhat{\betam}\in\R$, this set equals
$[\fhat{\alo}(x)-\Qhat{\betam},\fhat{\ahi}(x)+\Qhat{\betam}]$, with
$[a,b]=\varnothing$ when $a>b$.  If $\Qhat{\betam}=+\infty$, it equals
$\R$.  Thus $|\CChat(x)|$ always denotes the Lebesgue length, namely
the positive part of the endpoint difference.  Conditional on $\Dtrain$, the
calibration scores and an independent test score are exchangeable.
Consequently, for $\PP_{\Dtrain}$-almost every training sample,
\[
\PP(Y\in\CChat(X)\mid\Dtrain)\ge1-\alpha,
\]
where the probability is taken over the random calibration sample $\Dcal$
and the independent test point $(X,Y)$ \cite[Theorem~1]{romano2019conformalized}.

For $\delta\in(0,1)$, set
\begin{equation}
\label{eq:no_shift_scales}
\begin{aligned}
s_m&\coloneqq\betam-(1-\alpha),
&\eta_m(\delta)&\coloneqq\sqrt{\frac{\log(4/\delta)}{2m}},\\
r_-&\coloneqq r_0\wedge\frac{1-\alpha}{2\up},&r_+&\coloneqq r_0.
\end{aligned}
\end{equation}
The exact rank slack satisfies
\begin{equation}
\label{eq:rank_slack_exact}
\frac{1-\alpha}{m}\le s_m<\frac{2-\alpha}{m}.
\end{equation}
\paragraph{Finite-sample guarantees}
\begin{proposition}
\label{prop:target}
Let $\alpha\in(0,1)$ and $p\in[1,\infty]$, and assume
\Cref{assum:mass_regular} and \mbox{\Cref{assum:HPD-excess-risk}}.
Fix
$\delta\in(0,1)$ and $n,m$ such that
\begin{equation}
\label{eq:no_shift_localization}
2\up\epsilon_p(n,\delta/4)+\eta_m(\delta)\le2\low r_-,
\qquad
2\up\epsilon_p(n,\delta/4)+\eta_m(\delta)+s_m\le2\low r_+.
\end{equation}
Then there is an event $\mathcal A_{n,m}(\delta)$, measurable with respect to
$(\Dtrain,\Dcal)$, with probability at least $1-\delta$.  On this event,
$\Qhat{\betam}$ is finite and
\begin{equation}
\label{eq:length_mismatch_bound}
\Lp[{\DC[X]}]{|\CChat(\cdot)|-|\CCstar(\cdot)|}
\le2\left(1+\frac{\up}{\low}\right)\epsilon_p(n,\delta/4)
+\frac{\eta_m(\delta)+s_m}{\low}.
\end{equation}
\end{proposition}
The construction of $\mathcal A_{n,m}(\delta)$ and the complete proof are
given in \Cref{app:efficiency_proof}.

Once the data are fixed, a measurable prediction rule $C$ has conditional coverage profile
\[
\operatorname{Cov}_{C}(x)
\coloneqq \PP\{Y\in C(x)\mid X=x\}
=\DC[Y\mid X=x]\bigl(C(x)\bigr).
\]
The event in \Cref{prop:target} also controls this profile in $L^p(\DC[X])$.

\begin{theorem}
\label{th:global_cov_bound}
Let $\alpha\in(0,1)$ and $p\in[1,\infty]$, and assume
\Cref{assum:mass_regular} and \mbox{\Cref{assum:HPD-excess-risk}}.
Fix
$\delta\in(0,1)$ and $n,m$ satisfying \eqref{eq:no_shift_localization}. On the
event $\mathcal A_{n,m}(\delta)$ defined in \Cref{prop:target},
\begin{equation}
\label{eq:cond_cov}
\Lp[{\DC[X]}]{\operatorname{Cov}_{\CChat}(\cdot)-(1-\alpha)}
\le2\up\left(1+\frac{\up}{\low}\right)\epsilon_p(n,\delta/4)
+\frac{\up}{\low}\bigl(\eta_m(\delta)+s_m\bigr).
\end{equation}
\end{theorem}
The proof is given in \Cref{sec:global_cov_bound_proof}.

For $p=1$, nesting of scalar score sublevel sets gives a direct bound that
does not require localization of the calibrated threshold.

\begin{proposition}
\label{prop:global_cov_l1_direct}
Let $\alpha\in(0,1)$ and assume \Cref{assum:mass_regular} and
\hyperref[assum:HPD-excess-risk]{\textbf{A2(1)}}. For every
$\delta\in(0,1)$ and
$n,m\ge1$, there is an event of probability at least $1-\delta$ on which
\begin{equation}
\label{eq:cond_cov_l1_direct}
\Lone{\operatorname{Cov}_{\CChat}-(1-\alpha)}[{\DC[X]}]
\le
4\up\epsilon_1(n,\delta/4)+\eta_m(\delta)+s_m.
\end{equation}
\end{proposition}
A complete proof is given in \Cref{sec:global_cov_bound_proof}.

The three bounds above answer slightly different questions.  The marginal
split-conformal guarantee is an average over a fresh calibration sample and a
fresh test pair, conditional only on the training fold.  By contrast,
\Cref{th:global_cov_bound,prop:global_cov_l1_direct} fix the realized training
and calibration data and measure how far the resulting conditional-coverage
function is from the constant $1-\alpha$.  This distinction matters because
positive and negative conditional errors may cancel in the marginal average.
The $L^p$ norm retains their spatial magnitude, increasingly emphasizing
regions of poor conditional coverage as $p$ grows; the case $p=\infty$
controls the essential worst case.  Moreover, the same event
$\mathcal A_{n,m}(\delta)$ controls both length and coverage, so the two
criteria do not require separate favorable calibration realizations.

Conditions \eqref{eq:no_shift_localization} keep the calibrated threshold
within the region where \Cref{assum:mass_regular} provides the required lower
CDF growth; \Cref{prop:global_cov_l1_direct} avoids this localization for the
direct $L^1$ coverage bound.

\paragraph{Discussion}
For correctly specified linear CQR fitted by SGD, Yao et al.~\cite[Theorem~3.2]{yao2025non} obtain an expected absolute length deviation of order $n^{-1/2}+(\alpha^2n)^{-1}+m^{-1/2}+\exp(-\alpha^2m)$.
They assume continuous conditional densities bounded above and below on a common bounded support; \Cref{assum:mass_regular} uses a global upper mass bound and local lower mass bounds near the two oracle quantiles.
Our length bound holds with probability at least $1-\delta$ for every
$p\in[1,\infty]$ and accepts any endpoint learner satisfying
\Cref{assum:HPD-excess-risk}; \Cref{theo:requ_rates} verifies
\hyperref[assum:HPD-excess-risk]{\textbf{A2(2)}} for sparse-ReLU pinball ERM.
\Cref{th:global_cov_bound} also controls the realized conditional coverage profile.
Le Bars and Humbert~\cite{lebars2025volume} bound excess volume relative to an oracle for symmetric intervals with a common radius; our length loss measures deviations in both directions from the equal-tailed CQR oracle.
Min et al.~\cite{min2026unified} average conditional coverage over the procedure's randomness, while our profile is evaluated after the training and calibration samples are fixed.

\section{Known Covariate Shift}
\label{sec:cs_main}

In \Cref{sec:non-asymptotic-CQR}, training, calibration, and test pairs share the law $\DC[X]\otimes\DC[Y\mid X]$.
We now study how the preceding guarantees change under covariate shift, where training and calibration retain this source law while the test pair follows $\QC[X]\otimes\DC[Y\mid X]$.
The conditional response law, and hence the conditional quantiles and the oracle interval $\CCstar$, remain unchanged.
We analyze the weighted conformal CQR rule under this shift.
For this rule, we retain finite-sample target marginal validity and derive high-probability bounds, evaluated after the training and calibration samples have been fixed, for
\begin{equation}
\label{eq:cs_target_main}
\Lp[{\QC[X]}]{|\CChatw(\cdot)|-|\CCstar(\cdot)|},
\qquad
\Lp[{\QC[X]}]{\operatorname{Cov}_{\CChatw}(\cdot)-(1-\alpha)}.
\end{equation}

\begin{assumL}
\label{assum:cov_shift}
$\QC[X]\ll\DC[X]$, and the likelihood ratio
$\wsh=\rmd\QC[X]/\rmd\DC[X]$ and a deterministic envelope
$\wmax<\infty$ satisfying $0\le\wsh\le\wmax$ are known.
\end{assumL}

The absolute continuity condition and knowledge of $\wsh$ are inherited from the weighted conformal construction above.

Since $\E_{\DC[X]}[\wsh]=1$, the chi-square divergence satisfies
\[
\chi^2(\QC[X]\Vert\DC[X])
\coloneqq \int(\wsh-1)^2\,\rmd\DC[X]
=\E_{\DC[X]}[\wsh^2]-1.
\]
Together with $\wsh^2\le\wmax\wsh$, this identity gives
\begin{equation}
\label{eq:cs_weight_second_moment_main}
\csdiv
=\E_{\DC[X]}[\wsh^2]
\le\wmax.
\end{equation}
For every measurable $g$ and $p\in[1,\infty]$,
\begin{equation}
\label{eq:cs_change_measure}
\Lp[{\QC[X]}]{g}\le\wmax^{1/p}\Lp[{\DC[X]}]{g},
\qquad \wmax^{1/\infty}\coloneqq1.
\end{equation}

\paragraph{Weighted calibration}
Write $\Shat_j=\Shat(X_j,Y_j)$ for the fitted calibration scores from
\Cref{sec:non-asymptotic-CQR}, and define the raw weighted empirical function
\begin{equation}
\label{eq:cs_raw_main}
\Fhatw(t)
\coloneqq\frac1m\sum_{j=1}^m
\wsh(X_j)\indiacc{\Shat_j\le t},
\qquad
\Fhatw(+\infty)=\frac1m\sum_{j=1}^m\wsh(X_j).
\end{equation}
Conditional on $\Dtrain$, the fitted score is fixed and the calibration pairs
follow the source law. The Radon--Nikodym identity therefore changes the
covariate marginal from $\DC[X]$ to $\QC[X]$ and gives, for every $t\in\R$,
\[
\E[\Fhatw(t)\mid\Dtrain]
=\E_{\DC[XY]}\!\left[
\wsh(X)\indi{\{\Shat(X,Y)\le t\}}\middle|\Dtrain
\right]
=\FQS(t)
\coloneqq(\QC[X]\otimes\DC[Y\mid X])\{\Shat\le t\}.
\]
Its total mass $\Fhatw(+\infty)$ is random. Since
\[
\QC[X]\{\wsh=0\}
=\int_{\{\wsh=0\}}\wsh\,\rmd\DC[X]
=0,
\]
$\{\wsh>0\}$ is $\QC[X]$-full. Following the weighted split conformal construction of
Tibshirani et al.~\cite{tibshirani2019conformal}, for
$x\in\{\wsh>0\}$ define the probability measure
\begin{equation}
\label{eq:cs_exact_weighted_measure_main}
\sum_{j=1}^m
\frac{\wsh(X_j)}{m\Fhatw(+\infty)+\wsh(x)}\,
\delta_{\Shat_j}
+
\frac{\wsh(x)}{m\Fhatw(+\infty)+\wsh(x)}\,
\delta_{+\infty}.
\end{equation}
Let $\Qhatw(x)$ be its $(1-\alpha)$-quantile.
The corresponding weighted CQR interval is
\begin{equation}
\label{eq:cs_interval_main}
\CChatw(x)\coloneqq
\{y\in\R:\Shat(x,y)\le\Qhatw(x)\}.
\end{equation}
Equivalently, for $x\in\{\wsh>0\}$,
\Cref{lem:cs_exact_threshold_measurable} gives
\begin{equation}
\label{eq:cs_exact_raw_level_main}
\Qhatw(x)
=
\quant{(1-\alpha)\bigl(\Fhatw(+\infty)+\wsh(x)/m\bigr)}{\Fhatw}.
\end{equation}
By the weighted split conformal validity result of Tibshirani et al.~\cite[Corollary~1]{tibshirani2019conformal}, for almost every training realization,
\begin{equation}
\label{eq:cs_exact_marginal_main}
\PP\!\left(
Y_\star\in\CChatw(X_\star)\,\middle|\,\Dtrain
\right)\ge1-\alpha,
\end{equation}
where the probability averages over the source calibration sample and the
independent pair
$(X_\star,Y_\star)\sim\QC[X]\otimes\DC[Y\mid X]$; only the training fold is
conditioned upon. The results below fix both folds and control the realized
$L^p(\QC[X])$ profile with probability at least $1-\delta$.

\paragraph{Target guarantees}
For $\delta\in(0,1)$, set
\begin{equation}
\label{eq:cs_weighted_radius_main}
\etaw{\delta}
\coloneqq
9\sqrt{\frac{\csdiv\log(4/\delta)}{m}}
+\frac{4\wmax\log(4/\delta)}{3m}.
\end{equation}

\begin{proposition}
\label{prop:cs_length}
Let $\alpha\in(0,1)$, $p\in[1,\infty]$, and assume
\Cref{assum:mass_regular,assum:cov_shift} and
\Cref{assum:HPD-excess-risk}.
Fix $\delta\in(0,1)$ and $n,m\ge1$ such that
\begin{equation}
\label{eq:cs_localization_main}
\begin{aligned}
2\up\wmax^{1/p}\epsilon_p(n,\delta/4)
+(2-\alpha)\etaw{\delta}
&<2\low r_-,\\
2\up\wmax^{1/p}\epsilon_p(n,\delta/4)
+(2-\alpha)\etaw{\delta}
+\frac{(1-\alpha)\wmax}{m}
&\le2\low r_+.
\end{aligned}
\end{equation}
Then there is an event $\mathcal A^w_{n,m}(\delta)$, measurable with respect
to $(\Dtrain,\Dcal)$, with probability at least $1-\delta$.
On this event,
$\Qhatw(x)$ is finite for $\QC[X]$-almost every $x$ and
\begin{equation}
\label{eq:cs_length_main}
\begin{aligned}
\Lp[{\QC[X]}]{|\CChatw(\cdot)|-|\CCstar(\cdot)|}
&\le2\left(1+\frac{\up}{\low}\right)
\wmax^{1/p}\epsilon_p(n,\delta/4)
\\
&\quad+\frac{
(2-\alpha)\etaw{\delta}
+(1-\alpha)\Lp[{\QC[X]}]{\wsh}/m
}{\low}.
\end{aligned}
\end{equation}
\end{proposition}
The event and signed threshold bounds are established in
\Cref{lem:cs_calibration_bound}; the proof is given in
\Cref{subsec:cs_length}.

\begin{theorem}
\label{th:cs_cond_cov}
Let $\alpha\in(0,1)$ and $p\in[1,\infty]$, and assume
\Cref{assum:mass_regular,assum:cov_shift} and
\Cref{assum:HPD-excess-risk}. Fix $\delta\in(0,1)$ and $n,m$
satisfying
\eqref{eq:cs_localization_main}. On the event
$\mathcal A^w_{n,m}(\delta)$ of \Cref{prop:cs_length},
\begin{equation}
\label{eq:cs_cov_main}
\begin{aligned}
\Lp[{\QC[X]}]{\operatorname{Cov}_{\CChatw}(\cdot)-(1-\alpha)}
&\le2\up\left(1+\frac{\up}{\low}\right)
\wmax^{1/p}\epsilon_p(n,\delta/4)
\\
&\quad+\frac{\up}{\low}
\left(
(2-\alpha)\etaw{\delta}
+\frac{1-\alpha}{m}\Lp[{\QC[X]}]{\wsh}
\right).
\end{aligned}
\end{equation}
\end{theorem}
A complete proof is given in \Cref{subsec:cs_theorem}.

\paragraph{Calibration deviation}
By \Cref{lem:cs_weighted_expectation,lem:cs_weighted_ep}, conditionally on
$\Dtrain$,
\begin{equation}
\label{eq:cs_weighted_radius_purpose}
\PP\!\left(
\sup_{t\in\R}|\Fhatw(t)-\FQS(t)|\le\etaw{\delta}
\,\middle|\,\Dtrain
\right)
\ge1-\frac{\delta}{2}.
\end{equation}
The same event gives $|\Fhatw(+\infty)-1|\le\etaw{\delta}$.
The deviation bound applies to the unbiased raw process $\Fhatw$; the
conformal measure in \eqref{eq:cs_exact_weighted_measure_main} is normalized.

The $\csdiv$ term follows from the $L^2$ maximal bound for weighted score
prefixes; the $\wmax$ term comes from the envelope in Bousquet's inequality.
The leading constant $9$ combines symmetrization, the martingale maximal
inequality, and Bousquet's fluctuation bound.
Bounding the score-CDF and total-mass fluctuations separately gives
$(2-\alpha)$, although the same supremum controls both.
These constants are not optimized. For $\wsh\equiv1$, this weighted bound is
weaker than the unshifted bound based on $\eta_m(\delta)$ in
\eqref{eq:no_shift_scales}.

\paragraph{Discussion}
For fixed $\delta$, the endpoint, calibration, and test-atom terms in \eqref{eq:cs_length_main} and \eqref{eq:cs_cov_main} have orders $\wmax^{1/p}\epsilon_p(n,\delta/4)$, $\sqrt{\csdiv/m}+\wmax/m$, and $\|\wsh\|_{L^p(\QC[X])}/m$, respectively.
Pournaderi and Xiang~\cite[Theorem~1 and Corollary~1]{pournaderi2026trainingconditional} control the one-sided target marginal miscoverage of the same rule conditional on the complete source sample, whereas our bounds control the realized $L^p(\QC[X])$ profile and equal-tailed oracle length on an event measurable with respect to both source folds.
Under $\wsh\le\wmax$, their bounded-ratio excess has order $\sqrt{\wmax/m}+\wmax\sqrt{\log(4/\delta)/m}$ with exponential confidence, while their second-moment alternative has polynomial dependence on $1/\delta$.
Their score learner is unrestricted; our profile and length bounds additionally require the endpoint event and mass regularity.
Joshi et al.~\cite{joshi2025lrqr} learn a covariate-dependent threshold from
labeled source observations and unlabeled source and target covariates without
directly estimating $\wsh$. Their result gives target marginal coverage up to
estimation and likelihood-ratio projection errors, whereas our known-ratio
bounds control the realized $L^p(\QC[X])$ profile and equal-tailed oracle
length.

The three shift quantities in the bounds have distinct roles.
The factor $\wmax^{1/p}$ transfers an integrated endpoint error from the
source covariate law to the target law and disappears when $p=\infty$.
The second moment $\csdiv=\E_{\DC[X]}[\wsh^2]$ determines the leading
stochastic calibration fluctuation, just as a variance determines the scale
of an importance-weighted average. Finally,
$\|\wsh\|_{L^p(\QC[X])}/m$ comes from the additional mass assigned to
$+\infty$ for the test point in \eqref{eq:cs_exact_weighted_measure_main}.
The leading calibration term is not determined by $\wmax$ alone.
When $\wsh\equiv1$, $\csdiv=1$ and the effective calibration size is $m$.
The fixed score benchmarks in \Cref{sec:calibration_lower_bounds} show that
the dependence $\sqrt{\csdiv/m}$ can arise from calibration alone.

\section{Minimax Limits for Fixed-Score Calibration under Covariate Shift}
\label{sec:calibration_lower_bounds}

Split CQR uses independent samples for score construction and calibration.
Conditional on the training sample, the fitted score $\Shat$ is fixed, while
the calibration sample determines a threshold function through
\[
\CChatw(x)=\{y\in\R:\Shat(x,y)\leq\Qhatw(x)\}.
\]
This conditional representation motivates studying calibration with a fixed
score. For fixed failure probability, the leading calibration term in
\Cref{th:cs_cond_cov} has order $\sqrt{\csdiv/m}$; the benchmarks below ask
whether the same dependence can arise from calibration alone.

Each construction fixes source and target covariate marginals, a baseline
conditional kernel $P^0_{Y\mid X}$, and its equal-tailed oracle CQR score
$\Sstar$. The supremum varies $P_{Y\mid X}$ over all conditional kernels
satisfying \Cref{assum:mass_regular} on a measurable set of full
$\DC[X]$-measure, while the marginals and $\Sstar$ remain fixed. Each infimum
ranges over calibration rules that are measurable functions of the $m$ source
observations and take values in $\overline{\R}$ for scalar thresholds or in
$\overline{\R}^{K}$ for threshold vectors; the conformal value $+\infty$ is
admissible. Thus only the calibration rule depends on the source observations.
Both constructions use finite measurable subsets of $\Xset$ as carriers.

Without covariate shift, Areces et al.~\cite{areces2024two} obtain an
expected minimax lower bound of order $\sqrt{d/m}$ for covariate-dependent
fixed-score thresholds. Their loss is the supremum of a normalized weighted
coverage discrepancy over a class of $\{-1,1\}$-valued functions of the
covariates with VC dimension $d$.
Their upper theorem uses a finite-dimensional linear witness space with a
uniformly bounded orthonormal basis, so it is not a matching upper bound for
every binary VC class.
Duchi~\cite{duchi2025sampleconditional} derives high-probability
sample-conditional guarantees for threshold functions estimated by quantile
regression.
We study two calibration problems with a fixed score. The first has one target
atom, so only one threshold value affects the target risk. The second has $K$
target atoms and uses a separate threshold at each atom. In our construction,
each target atom receives $m/(\csdiv K)$ source calibration observations on
average, which yields the rate $\sqrt{\csdiv K/m}$.
For $p=1$ and $K=d$, its dimension dependence agrees with the atomic
lower-bound geometry of Areces et al.

Formally, a calibration rule maps each calibration sample $\Dcal$ to a
measurable threshold function
$\widehat q_{\Dcal}\colon\Xset\to\overline{\R}$, with
$(\Dcal,x)\mapsto\widehat q_{\Dcal}(x)$ jointly measurable. For a realized
calibration sample, define
\[
C_{\widehat q_{\Dcal}}^{\star}(x)
\coloneqq
\bigl\{y\in\R:\Sstar(x,y)\leq\widehat q_{\Dcal}(x)\bigr\}.
\]
For a candidate conditional kernel $P_{Y\mid X}$, its realized
conditional coverage profile is
\begin{equation}
\label{eq:fixed_score_coverage_profile}
\operatorname{Cov}^{P}_{\{\Sstar\leq\widehat q_{\Dcal}\}}(x)
\coloneqq
P_{Y\mid X=x}\!\left(C_{\widehat q_{\Dcal}}^{\star}(x)\right)
=
P_{Y\mid X=x}\!\left(
\Sstar(x,Y)\leq\widehat q_{\Dcal}(x)
\right).
\end{equation}
The calibration sample is held fixed in this conditional probability. The
profile remains random through $\widehat q_{\Dcal}$. Henceforth, we suppress
the subscript $\Dcal$ and write $\widehat q$, $C_{\widehat q}^{\star}$, and
$\operatorname{Cov}^{P}_{\{\Sstar\leq\widehat q\}}$.

The exact weighted CQR threshold $\Qhatw(x)$ is covariate dependent through
the test-point weight $\wsh(x)$ in
\eqref{eq:cs_exact_weighted_measure_main}--\eqref{eq:cs_exact_raw_level_main}.
The calibration rules considered here may likewise depend on $x$. In the
single-carrier experiment,
$\QC[X]=\delta_{x_0}$, so an arbitrary threshold function enters the target
risk only through the scalar $\widehat q_{\Dcal}(x_0)$; conversely, any scalar
rule admits a measurable constant extension. In the $K$-atomic experiment,
$\QC[X]$ is uniform on $x_1,\ldots,x_K$, so the rule is identified with the
vector of its values at these atoms. Values away from the target support do
not enter the $L^p(\QC[X])$ loss.

\begin{theorem}
\label{prop:cs_lecam_lower}
Fix $\alpha\in(0,1)$ and $\kappa>0$, and assume that $\Xset$ contains two
distinct points whose singleton sets are measurable. Then there exist source
and target marginals $\DC[X],\QC[X]$, a reference conditional kernel
$P^0_{Y\mid X}$, and a reference score $\Sstar$, all independent of $m$, with
the following properties. The marginals satisfy \Cref{assum:cov_shift} and
$\chi^2(\QC[X]\Vert\DC[X])=\kappa$.
The kernel $P^0_{Y\mid X}$ satisfies \Cref{assum:mass_regular} on a set of
full $\DC[X]$-measure. The score $\Sstar$ is the equal-tailed oracle CQR score
under $P^0_{Y\mid X}$ and is held fixed throughout the minimax problem below.
For $m\ge1$ and $p\in[1,\infty]$, define
\begin{equation}
\label{eq:cs_scalar_risk}
\mathcal R_{m,p}\coloneqq
\inf_{\widehat q}
\sup_{P_{Y\mid X}}
\E_{(\DC[X]\otimes P_{Y\mid X})^{\otimes m}}
\!\left[
\Lp[{\QC[X]}]{
\operatorname{Cov}^{P}_{\{\Sstar\le\widehat q\}}(\cdot)-(1-\alpha)}
\right].
\end{equation}
Set
\begin{equation}
\label{eq:cs_lecam_threshold}
m_\star\coloneqq
\left\lceil
\frac{4(\log 2)(1+\kappa)\,
\alpha(1-\alpha)}{\min\{\alpha,1-\alpha\}^2}
\right\rceil.
\end{equation}
Then for every $m\ge m_\star$ and $p\in[1,\infty]$,
\[
\mathcal R_{m,p}
\ge c_\alpha\sqrt{\frac{1+\kappa}{m}},
\qquad
c_\alpha\coloneqq\frac18\sqrt{(\log 2)\alpha(1-\alpha)}.
\]
\end{theorem}
A complete proof is given in \Cref{subsec:cs_lecam}.

\paragraph{Lower-bound construction}
Choose distinct $x_0,x_D\in\Xset$ and set
\[
\QC[X]=\delta_{x_0},
\qquad
\DC[X]=\frac1{1+\kappa}\delta_{x_0}
+\frac{\kappa}{1+\kappa}\delta_{x_D}.
\]
Then $\csdiv=1+\kappa$, the likelihood ratio satisfies
$\wsh(x_0)=\csdiv$ and $\wsh(x_D)=0$, and
\[
\lVert\wsh\rVert_{L^\infty(\DC[X])}
=\int_{\Xset}\wsh(x)^2\,\rmd\DC[X](x)
=\csdiv.
\]
Thus a source calibration draw has $X=x_0$ with probability $1/\csdiv$.
Since $\QC[X]=\delta_{x_0}$, for every candidate
kernel $P_{Y\mid X}$, every realized scalar threshold $\widehat q$, and every
$p\in[1,\infty]$,
\[
\Lp[{\QC[X]}]{
\operatorname{Cov}^{P}_{\{\Sstar\leq\widehat q\}}(\cdot)-(1-\alpha)}
=\bigl|
\operatorname{Cov}^{P}_{\{\Sstar\leq\widehat q\}}(x_0)-(1-\alpha)
\bigr|.
\]
These identities belong to the construction; \Cref{assum:cov_shift} itself
only gives $\csdiv\le\wmax$.

Let $P^0_{Y\mid X}$ be uniform on $[-1,1]$ at every covariate value, and fix
the score $\Sstar(x,y)=|y|-(1-\alpha)$. This is the equal-tailed oracle CQR
score for $P^0_{Y\mid X}$; the same score is retained under the alternative
$P^1_{Y\mid X}$. Define the two score bands
\[
B_0=[-(1-\alpha),0],
\qquad
B_1=(0,\alpha],
\]
which have respective masses $1-\alpha$ and $\alpha$ under
$P^0_{Y\mid X}$. Let $P^1_{Y\mid X}$ agree with $P^0_{Y\mid X}$ away from
$x_0$ and, at $x_0$, transfer conditional mass $\xi$ from $B_0$ to $B_1$.

Clipping a realized scalar threshold $\widehat q$ to $[0,\alpha]$ cannot
increase its absolute coverage error under either hypothesis. Write the
clipped threshold as $\widetilde q=\alpha a$, where $a\in[0,1]$ is the
fraction of $B_1$ admitted by the threshold ray. For $j\in\{0,1\}$, define
\[
D_j(a)
\coloneqq
\operatorname{Cov}^{P^j}_{\{\Sstar\leq\alpha a\}}(x_0)-(1-\alpha).
\]
Then
\[
D_0(a)=\alpha a,
\qquad
D_1(a)=(\alpha+\xi)a-\xi,
\]
whose respective zeros are $0$ and $\xi/(\alpha+\xi)$. Thus no scalar
threshold attains exact target coverage under both hypotheses. Since the two
source experiments differ only when $X=x_0$, their per-observation divergence
is
\[
\chi^2\!\left(
\DC[X]\otimes P^1_{Y\mid X}
\,\middle\Vert\,
\DC[X]\otimes P^0_{Y\mid X}
\right)
=\frac{\xi^2}{\csdiv\alpha(1-\alpha)}.
\]
For
$\xi^2=(\log 2)\csdiv\alpha(1-\alpha)/m$, the total variation distance
between the two $m$-sample laws is at most $1/2$, and Le Cam's lemma gives
the lower bound $\xi/8$ in \Cref{prop:cs_lecam_lower}. The density construction
and Le Cam reduction are given in \Cref{subsec:cs_lecam}.

For an upper bound on the same construction, apply the weighted conformal
construction \eqref{eq:cs_exact_weighted_measure_main} to the calibration
scores $\Sstar(X_i,Y_i)$ under the same source and target marginals and fixed
score, and let $\Qhatw(x)$ denote its $(1-\alpha)$ quantile. Since
$\QC[X]=\delta_{x_0}$, only $\Qhatw(x_0)$ affects the $L^p(\QC[X])$ profile
loss. \Cref{prop:cs_carrier_upper} bounds this loss uniformly over the
candidate conditional kernels, with
$\Dcal\sim(\DC[X]\otimes P_{Y\mid X})^{\otimes m}$.

\begin{proposition}
\label{prop:cs_carrier_upper}
Under the assumptions of \Cref{prop:cs_lecam_lower}, consider the source and
target marginals and fixed reference score used in its two-point construction,
and let $x_0$ denote the unique target atom. Thus
$\QC[X]=\delta_{x_0}$ and $\DC[X](\{x_0\})=\csdiv^{-1}$. Then, for every
$m\ge1$ and $p\in[1,\infty]$,
\begin{equation}
\label{eq:cs_carrier_finite_upper}
\begin{aligned}
&\sup_{P_{Y\mid X}}
\E_{\Dcal}\!\left[
\Lp[{\QC[X]}]{
\operatorname{Cov}^{P}_{\{\Sstar\le\Qhatw(x_0)\}}(\cdot)-(1-\alpha)}
\right]\\
&\qquad\le
\frac32
\sqrt{
\frac{\csdiv}{m+1}
\left[1-\left(1-\csdiv^{-1}\right)^{m+1}\right]}\\
&\qquad\le \frac32\sqrt{\frac{\csdiv}{m+1}}.
\end{aligned}
\end{equation}
\end{proposition}

The scalar rule $\widehat q=\Qhatw(x_0)$ is admissible in the definition of $\mathcal R_{m,p}$.
Combining \Cref{prop:cs_lecam_lower,prop:cs_carrier_upper} therefore gives, for every $m\ge m_\star$ and $p\in[1,\infty]$,
\begin{equation}
\label{eq:cs_carrier_rank_upper}
c_\alpha\sqrt{\frac{\csdiv}{m}}
\le \mathcal R_{m,p}
\le \frac32\sqrt{\frac{\csdiv}{m+1}}
\le \frac32\sqrt{\frac{\csdiv}{m}}.
\end{equation}
Thus, on this fixed-score construction, the weighted conformal rule of
Tibshirani et al.~\cite{tibshirani2019conformal} attains the expected minimax
rate $\sqrt{\csdiv/m}$ up to $\alpha$-dependent constants. The effective
calibration size is $m/\csdiv$; the same shift factor appears in the leading
calibration term of \Cref{th:cs_cond_cov}.

For a target uniform on $K$ atoms, the calibration rule must estimate $K$
threshold values from the same source sample. The theorem below gives a
minimax lower bound whose probability tends to one exponentially fast in
$K$. For finite $p$,
\Cref{prop:cs_cellwise_upper} gives the corresponding moment upper bound.

\begin{theorem}
\label{thm:cs_fano_lower}
Fix $\alpha\in(0,1)$, an integer $K\ge23$, and $\kappa>0$, and assume that
$\Xset$ contains at least $2K$ distinct points whose singleton sets are
measurable. Then there exist source and target marginals $\DC[X],\QC[X]$,
distinct points $x_1,\ldots,x_K$, a reference conditional kernel
$P^0_{Y\mid X}$, and a reference score $\Sstar$, all independent of $m$,
with the following properties. The marginals satisfy
\Cref{assum:cov_shift},
\[
\chi^2(\QC[X]\Vert\DC[X])=\kappa,
\qquad
\QC[X]=\frac1K\sum_{k=1}^K\delta_{x_k}.
\]
$P^0_{Y\mid X}$ satisfies \Cref{assum:mass_regular} on a set of full
$\DC[X]$-measure. The score $\Sstar$ is the equal-tailed oracle CQR score
under $P^0_{Y\mid X}$ and is held fixed throughout the minimax problem below.
For a threshold vector $\widehat q_{1:K}$, define its canonical measurable
extension by
\[
\widehat q(x)
=
\begin{cases}
\widehat q_k,&x=x_k\text{ for some }k\in\{1,\ldots,K\},\\
0,&x\notin\{x_1,\ldots,x_K\}.
\end{cases}
\]
The value chosen off the target support is immaterial to the loss below.
Set
\begin{equation}
\label{eq:cs_fano_threshold}
m_\star^{(K)}\coloneqq
\frac{(1+\kappa)\,\alpha(1-\alpha)K}
{4\min\{\alpha,1-\alpha\}^2}.
\end{equation}
Then for every integer $m>m_\star^{(K)}$ and every $p\in[1,\infty]$,
\[
\begin{aligned}
&\inf_{\widehat q_{1:K}}
\sup_{P_{Y\mid X}}
\PP_{(\DC[X]\otimes P_{Y\mid X})^{\otimes m}}\!\Biggl(
\Lp[{\QC[X]}]{\operatorname{Cov}^{P}_{\{\Sstar\le\widehat q\}}(\cdot)-(1-\alpha)}\\
&\hspace{11em}\ge\frac1{64}
\sqrt{\frac{(1+\kappa)\,
\alpha(1-\alpha)K}{m}}
\Biggr)
\ge1-2e^{-K/32}.
\end{aligned}
\]
\end{theorem}

A complete proof is given in \Cref{subsec:cs_lev}.

\paragraph{Geometry of the hypercube}
The construction pairs each target atom $x_k$ with an auxiliary atom
$x_{D,k}$ carrying source mass but no target mass. Select distinct points
$x_1,\ldots,x_K$ and $x_{D,1},\ldots,x_{D,K}$, and support both marginals on
these $2K$ points. For $k=1,\ldots,K$, set
\[
\begin{gathered}
\QC[X](\{x_k\})=\frac1K,
\qquad
\QC[X](\{x_{D,k}\})=0,\\
\DC[X](\{x_k\})=\frac1{(1+\kappa)K},
\qquad
\DC[X](\{x_{D,k}\})=\frac{\kappa}{(1+\kappa)K}.
\end{gathered}
\]
Then $\csdiv=1+\kappa$, $\wsh(x_k)=\csdiv$ and
$\wsh(x_{D,k})=0$, and
\[
\lVert\wsh\rVert_{L^\infty(\DC[X])}
=\int_{\Xset}\wsh(x)^2\,\rmd\DC[X](x)
=\sum_{k=1}^K\frac{\csdiv^2}{\csdiv K}
=\csdiv,
\]
and a source draw reaches a given target atom with probability
$1/(\csdiv K)$. The $x_{D,k}$ absorb the remaining source mass without
carrying information about the alternative.

Unlike the scalar construction, the target mass is now spread over $K$
atoms. For a threshold vector $\widehat q_{1:K}$, write
\[
e_k(\widehat q_k)
\coloneqq\operatorname{Cov}^{P}_{\{\Sstar\le\widehat q\}}(x_k)-(1-\alpha).
\]
Then the target marginal error and the realized profile loss are, respectively,
\[
\begin{aligned}
\left|\E_{\QC[X]}[\operatorname{Cov}^{P}_{\{\Sstar\le\widehat q\}}(X)]-(1-\alpha)\right|
&=\left|\frac1K\sum_{k=1}^K e_k(\widehat q_k)\right|,\\
\Lp[{\QC[X]}]{\operatorname{Cov}^{P}_{\{\Sstar\le\widehat q\}}-(1-\alpha)}
&=\begin{cases}
\displaystyle
\left(\frac1K\sum_{k=1}^K|e_k(\widehat q_k)|^p\right)^{1/p},
&p<\infty,\\[1ex]
\displaystyle\max_{1\le k\le K}|e_k(\widehat q_k)|,
&p=\infty.
\end{cases}
\end{aligned}
\]
The marginal error may vanish by cancellation, whereas the profile loss
retains the deviations at individual atoms. The Fano argument therefore
encodes the alternatives by $K$-dimensional binary vectors and compares the
corresponding threshold vectors coordinate by coordinate.

Use the same uniform reference conditional at every carrier point, reuse the
score bands $B_0,B_1$ from the scalar construction, and write $D$ for the
coarsened label of the corresponding dump point $x_{D,k}$. A vertex
$\tau\in\{0,1\}^K$ transfers mass $\xi$ from $B_0$ to $B_1$ at target atom
$x_k$ when $\tau_k=1$. After a source observation is coarsened to a label in
$\{1,\ldots,K\}\times\{B_0,B_1,D\}$, let the resulting source law be
$\overline{\DC}^{\,\tau}$. Its masses are
\[
\begin{gathered}
\overline{\DC}^{\,\tau}(k,B_0)
=\frac{1-\alpha-\xi\tau_k}{\csdiv K},
\qquad
\overline{\DC}^{\,\tau}(k,B_1)
=\frac{\alpha+\xi\tau_k}{\csdiv K},
\\
\overline{\DC}^{\,\tau}(k,D)=\frac{\csdiv-1}{\csdiv K}.
\end{gathered}
\]
The threshold fraction yielding conditional coverage $1-\alpha$ at $x_k$ is
\[
a_k^\star(\tau)=\frac{\xi}{\alpha+\xi}\tau_k.
\]
For $\tau_k=0$ and $\tau_k=1$, this fraction equals $0$ and
$\xi/(\alpha+\xi)$, respectively. The proof selects exponentially many binary
vectors such that every pair differs in at least $K/4$ coordinates, and sets
\[
\xi=\frac14
\sqrt{\frac{\csdiv\alpha(1-\alpha)K}{m}}.
\]
For this choice of $\xi$, the $\chi^2$ divergence between each $m$-sample
alternative and the $m$-sample reference law is at most $e^{K/16}-1$.
Target $L^1$ coverage loss below $\xi/16$ identifies the corresponding binary
vector. Fano's inequality bounds the average probability of correct
identification over the selected vectors by $2e^{-K/32}$; hence the worst-case
probability of loss at least $\xi/16$ is at least $1-2e^{-K/32}$. The density
and $\chi^2$ calculations and the identification argument are given in
\Cref{subsec:cs_lev}.

\paragraph{Split conformal calibration at each target atom}
For the upper bound, apply the standard finite-sample split conformal rank
correction separately to the calibration observations at each target atom
$x_k$.
For $k=1,\ldots,K$, let
$N_k=\sum_{i=1}^m\indiacc{X_i=x_k}$ be the number of source calibration
observations at $x_k$, and set
$j_k=\lceil(N_k+1)(1-\alpha)\rceil$.
Define
\begin{equation}
\label{eq:cs_cellwise_rule}
\widehat q_k
\coloneqq
\inf\left\{
t\in\R:
\sum_{i=1}^m
\indiacc{X_i=x_k}\indiacc{\Sstar(X_i,Y_i)\le t}
\ge j_k
\right\},
\qquad k=1,\ldots,K.
\end{equation}
Under covariate shift, the source and target laws share
$P_{Y\mid X=x_k}$, so $\Sstar(x_k,Y)$ has the same conditional distribution
under both laws. If $j_k=N_k+1$, then $\widehat q_k=+\infty$ by the generalized
inverse convention.
On this construction, $N_k\sim\operatorname{Bin}(m,1/(\csdiv K))$,
so the effective calibration size at each target atom is $m/(\csdiv K)$.

\begin{proposition}
\label{prop:cs_cellwise_upper}
Under the assumptions of \Cref{thm:cs_fano_lower}, consider the source and
target marginals, target atoms $x_1,\ldots,x_K$, and fixed reference score
used in its $K$-atomic construction. In particular,
\[
\QC[X]=\frac1K\sum_{k=1}^K\delta_{x_k},
\qquad
\DC[X](\{x_k\})=\frac1{\csdiv K},
\quad k=1,\ldots,K.
\]
For every $p\in[1,\infty)$, there is a
constant $C_p<\infty$, depending only on $p$, such that the rule
\eqref{eq:cs_cellwise_rule}, with $\widehat q$ defined by the canonical
extension in \Cref{thm:cs_fano_lower}, satisfies, for every $m\ge1$,
\begin{equation}
\label{eq:cs_cellwise_moment_upper}
\sup_{P_{Y\mid X}}
\left\{
\E_{\Dcal}\!\left[
\left(
\Lp[{\QC[X]}]{
\operatorname{Cov}^{P}_{\{\Sstar\le\widehat q\}}(\cdot)-(1-\alpha)}
\right)^p
\right]
\right\}^{1/p}
\le C_p\sqrt{\frac{\csdiv K}{m}}.
\end{equation}
\end{proposition}

A complete proof is given in \Cref{subsec:cs_cellwise_upper}.

By Markov's inequality, for every $\delta\in(0,1)$,
\begin{equation}
\label{eq:cs_cellwise_probability_upper}
\sup_{P_{Y\mid X}}
\PP_{\Dcal}\!\left(
\Lp[{\QC[X]}]{
\operatorname{Cov}^{P}_{\{\Sstar\le\widehat q\}}(\cdot)-(1-\alpha)}
\ge C_p\delta^{-1/p}\sqrt{\frac{\csdiv K}{m}}
\right)
\le\delta.
\end{equation}
Integrating the lower tail bound in \Cref{thm:cs_fano_lower} and applying
Jensen's inequality to \Cref{prop:cs_cellwise_upper} gives, for every
$K\ge23$, $p\in[1,\infty)$, and $m$ satisfying the condition of
\Cref{thm:cs_fano_lower},
\begin{equation}
\label{eq:cs_cellwise_expected_sandwich}
\begin{aligned}
&\left(1-2e^{-K/32}\right)
\frac{\sqrt{\alpha(1-\alpha)}}{64}
\sqrt{\frac{\csdiv K}{m}}\\
&\qquad\le
\inf_{\widehat q_{1:K}}
\sup_{P_{Y\mid X}}
\E_{\Dcal}\!\left[
\Lp[{\QC[X]}]{
\operatorname{Cov}^{P}_{\{\Sstar\le\widehat q\}}(\cdot)-(1-\alpha)}
\right]
\le C_p\sqrt{\frac{\csdiv K}{m}}.
\end{aligned}
\end{equation}
Thus the expected minimax rate on this construction is
$\sqrt{\csdiv K/m}$ for every fixed $p<\infty$. For every fixed
$\delta<1-2e^{-K/32}$, \eqref{eq:cs_cellwise_probability_upper} shows that,
uniformly over the candidate kernels, the atomwise rule has loss at most
\[
C_p\delta^{-1/p}\sqrt{\frac{\csdiv K}{m}}
\]
with probability at least $1-\delta$. Conversely, for every calibration rule,
\Cref{thm:cs_fano_lower} gives the worst-case lower bound
\[
\frac{\sqrt{\alpha(1-\alpha)}}{64}
\sqrt{\frac{\csdiv K}{m}}
\]
with probability at least $1-2e^{-K/32}>\delta$. The moment bound does not
cover $p=\infty$. Its high probability constant grows as $\delta^{-1/p}$ and
is therefore not independent of $K$ when $\delta$ decreases exponentially in
$K$.

For expected loss, the rates $\sqrt{\csdiv/m}$ and
$\sqrt{\csdiv K/m}$ are the inverse square roots of the mean source counts
$m/\csdiv$ and $m/(\csdiv K)$ at the target atoms. In both lower bounds,
$\Sstar$ is fixed and the infimum ranges over every measurable scalar or
vector threshold rule. The weighted conformal rule attains the one-atom rate
for $p\in[1,\infty]$, while the atomwise split conformal rule attains the
$K$-atom rate for every fixed $p<\infty$.

\section{Numerical Experiments}
\label{sec:numerical_experiments}

Experiments E1--E2 evaluate the fixed-score carrier rules from
\Cref{sec:calibration_lower_bounds}; Experiment E3 illustrates the learned CQR
bounds of \Cref{sec:non-asymptotic-CQR,sec:cs_main} under a continuous shift.
The complete carrier identities, parameters, and plots are given in
\Cref{sec:numerical_carrier_details}.

\subsection{Fixed-Score Carrier Experiments}

E1 uses the two-point construction of \Cref{prop:cs_lecam_lower} at
$\alpha=0.1$ and $\kappa\in\{1,3,9,27\}$. The exact weighted-rule risk is a
Binomial--Beta mixture, so no response simulation is needed. Against the
effective size $m/\csdiv$, the four curves nearly collapse. At small effective
sizes, the conformal rank can exceed the number of carrier observations; the
threshold is then $+\infty$ and the absolute coverage error is $\alpha$,
producing the visible plateau. At large effective sizes, the scaled risk
converges to the Gaussian constant
$\sqrt{2\alpha(1-\alpha)/\pi}$. The comparison with
\Cref{prop:cs_lecam_lower,prop:cs_carrier_upper} also shows the gap created by
unoptimized nonasymptotic constants; see \Cref{fig:e1_scalar}.

E2 applies split calibration separately at
$K\in\{23,64,256\}$ target atoms for $\kappa\in\{1,3,9\}$. Exact beta moments
give $\{\E L_p^p\}^{1/p}$, while Monte Carlo estimates from $10^4$
replications give $\E L_p$. The curves illustrate the scaling
$\sqrt{\csdiv K/m}$ for finite $p$. Their separation for $p>1$ is the
finite-$K$ Jensen gap, rather than a discrepancy between theory and
simulation. The $p=\infty$ panel illustrates the carrier rule's
$\sqrt{\log K}$ behavior after the effective-size limit; see
\Cref{fig:e2_cellwise}.

\subsection{Learned CQR under Covariate Shift}
\label{subsec:numerical_cqr}

At $\alpha=0.1$, E3 takes $\DC[X]=\operatorname{Unif}[0,1]$ and
\[
 \wsh(x)=\frac{a_\kappa e^{a_\kappa x}}{e^{a_\kappa}-1},
 \qquad
 \frac{a_\kappa}{2}\coth\!\left(\frac{a_\kappa}{2}\right)=1+\kappa,
 \qquad \kappa\in\{1,3\},
\]
so $\csdiv=1+\kappa$. The response model is
\[
 Y\mid X=x\sim\mathcal N\!\left(\sin(2\pi x),
 \left\{\tfrac12+\tfrac14\cos(2\pi x)\right\}^2\right).
\]
It satisfies \Cref{assum:mass_regular}. Raw endpoints are fitted either by a
one-hidden-layer tanh network of width $20$ or by misspecified affine quantile
regression, and then sorted as in \eqref{eq:raw_endpoint_sorting}. Thus E3
illustrates the generic learned-score results.

We compare the exact weighted rule $\CChatw$, unweighted split CQR, and the
population-normalized diagnostic
\[
 \{y:\Shat(x,y)\le \quant{\betam}{\Fhatw}\}.
\]
The last rule replaces the random normalization and test-point correction in
\eqref{eq:cs_exact_raw_level_main} by $\betam$ and therefore has no exact
target validity guarantee. Because the conditional law is Gaussian, all
coverage profiles and target integrals are evaluated deterministically; only
the source training and calibration folds are simulated.

\begin{figure}[!htbp]
  \centering
  \includegraphics[width=\textwidth]{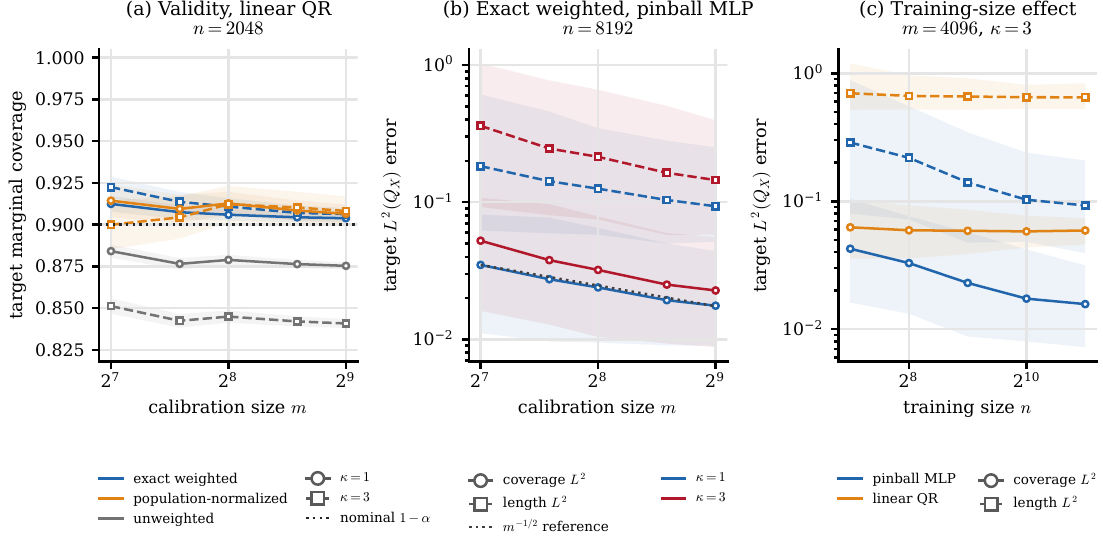}
  \caption{Learned CQR experiment (E3) under exponential covariate shift at
  $\alpha=0.1$, from $R=200$ replications. Panel~(a) reports means with
  pointwise $95\%$ intervals; panels~(b)--(c) report medians with pointwise
  empirical $[0.025,0.975]$-quantile bands.}
  \label{fig:e3_cqr}
\end{figure}

Panel~(a) uses affine endpoints with $n=2048$. The exact weighted rule is
consistent with target marginal validity, while the unweighted rule
undercovers under the positive tilt. Panel~(b), using the network at
$n=8192$, shows length and coverage-profile errors decaying at the
calibration scale $m^{-1/2}$, with a larger prefactor for $\kappa=3$.
Panel~(c) fixes $(m,\kappa)=(4096,3)$: network errors decrease with the
training size, whereas the affine fit reaches its misspecification floor,
illustrating the endpoint term in
\Cref{prop:cs_length,th:cs_cond_cov}.

\FloatBarrier

\section{A Sparse-ReLU Instantiation of the Endpoint Condition}
\label{sec:relu_specialization}

Neural quantile regression rates are established by Madrid Padilla et
al.~\cite{madridPadilla2022quantile} and Shen et al.~\cite{shen2021deep}.
Here we verify the $p=2$ instance of
\Cref{assum:HPD-excess-risk} for one sparse-ReLU
pinball estimator and substitute the resulting endpoint radius into the
generic CQR bounds. The verification uses the bounded support and two-sided
conditional-density bounds required by the sparse-ReLU approximation and the
pinball-risk argument.

\begin{assum}
\label{assum:boundedness}
$\Xset=[0,1]^d$ and $\Y=[-M,M]$ for some $M>0$.
\end{assum}
\begin{assum}
\label{assum:density}
For every $x\in\Xset$, the conditional law $\DC[Y\mid X=x]$ has support
$\Y$ and a Lebesgue density satisfying
$0<\low\le p_{Y\mid X}(y\mid x)\le\up<\infty$ for every $y\in\Y$.
\end{assum}
For every $\alpha\in(0,1)$, \Cref{assum:boundedness,assum:density} imply
\Cref{assum:mass_regular} with $r_0=\alpha/(2\up)$; see
\Cref{subsec:relu_endpoint_condition}.

\begin{assumB}
\label{assum:smoothness}
For each $\tau\in\{\alo,\ahi\}$, the conditional quantile function
$\fstar{\tau}$ belongs to the isotropic H\"older ball
$\HC^\beta(\Xset,\lipconst_\beta)$ for some $\beta>0$.
\end{assumB}

\paragraph{Estimator and endpoint rate}
Let $\mathcal F_n^{\mathrm{SH}}$ denote the sparse ReLU class in
\eqref{eq:NNSHn}, with outputs clipped to $[-M,M]$.
The pinball loss, H\"older-ball convention, architecture, and training threshold
$n_{\mathrm{SH}}$ are specified in \Cref{sec:nn-classes}.

\begin{theorem}
\label{theo:requ_rates}
Fix $\alpha\in(0,1)$ and $\tau\in\{\alo,\ahi\}$, and assume
\Cref{assum:boundedness,assum:density,assum:smoothness}.
Let $\widetilde f_{n,\tau}$ be a measurable empirical
pinball risk minimizer over $\mathcal F_n^{\mathrm{SH}}$.
There are constants
$C_{\mathrm{rate}},C_{\mathrm{conf}}>0$, depending only on
$(\alpha,\low,\up,\beta,d,M,\lipconst_\beta)$, such that, for every
$\delta\in(0,1)$ and integer $n\ge n_{\mathrm{SH}}$, with probability at
least $1-\delta$ over $\Dtrain$,
\[
\Ltwo{\widetilde f_{n,\tau}-\fstar{\tau}}
\le C_{\mathrm{rate}}n^{-\beta/(2\beta+d)}(\log n)^{3/2}
+C_{\mathrm{conf}}\sqrt{\frac{\log(1/\delta)}{n}}.
\]
\end{theorem}

The proof combines the sparse ReLU approximation of
Schmidt-Hieber~\cite[Theorem~5]{schmidthieber_2020}, metric entropy, a
Bernstein oracle inequality, and the quadratic risk comparison.
A complete proof is given in \Cref{sec:theo_requ_rates}.

Together with the deterministic $2M$ bound, \Cref{theo:requ_rates} verifies
\Cref{assum:HPD-excess-risk} at $p=2$; the radius $\epsilon_2$ is given in
\eqref{eq:relu_epsilon_main}.

\begin{corollary}
\label{cor:relu_rates}
Fix $\alpha\in(0,1)$. Under
\Cref{assum:boundedness,assum:density,assum:smoothness}, construct the endpoints
by~\eqref{eq:raw_endpoint_sorting}. For every $\delta\in(0,1)$, fix
$n\ge n_{\mathrm{SH}}$ and $m\ge1$ satisfying
\eqref{eq:no_shift_localization} with $p=2$ and $\epsilon_2$ from
\eqref{eq:relu_epsilon_main}. Then, with probability at least $1-\delta$,
\begin{equation}
\label{eq:relu_rate}
\max\!\left\{
\Ltwo{|\CChat(\cdot)|-|\CCstar(\cdot)|},
\Ltwo{\operatorname{Cov}_{\CChat}(\cdot)-(1-\alpha)}
\right\}
\ \lesssim\
\begin{aligned}
&n^{-\frac{\beta}{2\beta+d}}(\log n)^{3/2}
+\sqrt{\tfrac{\log(4/\delta)}{n}}\\
&\qquad+\frac1m+\sqrt{\tfrac{\log(4/\delta)}{m}},
\end{aligned}
\end{equation}
where $\lesssim$ hides only
$(\alpha,\low,\up,\beta,d,M,\lipconst_\beta)$-dependent constants.
\end{corollary}
The proof combines \Cref{prop:target,th:global_cov_bound} with
\Cref{theo:requ_rates}; see \Cref{app:length_mismatch_l2}.

\begin{corollary}
\label{cor:cs_relu}
Fix $\alpha\in(0,1)$ and assume
\Cref{assum:boundedness,assum:density,assum:cov_shift,assum:smoothness},
and use the rearranged sparse-ReLU endpoints
of~\eqref{eq:raw_endpoint_sorting}. For every $\delta\in(0,1)$, fix
$n\ge n_{\mathrm{SH}}$ and $m\ge1$ satisfying
\eqref{eq:cs_localization_main} with $p=2$ and $\epsilon_2$ from
\eqref{eq:relu_epsilon_main}. Then, with probability at least $1-\delta$,
\begin{equation}
\label{eq:cs_relu_rate}
\begin{aligned}
\max\!\left\{
\begin{gathered}
\Ltwo{|\CChatw(\cdot)|-|\CCstar(\cdot)|}[{\QC[X]}],\\
\Ltwo{\operatorname{Cov}_{\CChatw}(\cdot)-(1-\alpha)}[{\QC[X]}]
\end{gathered}
\right\}
&\lesssim
\wmax^{1/2}\left\{
n^{-\beta/(2\beta+d)}(\log n)^{3/2}
+\sqrt{\frac{\log(4/\delta)}n}\right\}\\
&\quad +(2-\alpha)\etaw{\delta}
+\frac{1-\alpha}{m}\sqrt{\wmax\,\csdiv}.
\end{aligned}
\end{equation}
The multiplicative constant depends only on
$(\alpha,\low,\up,\beta,d,M,\lipconst_\beta)$, not on
$(n,m,\delta,\wsh,\wmax)$.
The last term follows from
$\Ltwo{\wsh}[{\QC[X]}]\le\sqrt{\wmax\,\csdiv}$ and is the correction due to
the test atom at $+\infty$.
\end{corollary}
The proof is given in \Cref{subsec:cs_corollary}.

\section{Conclusion}
\label{sec:conclusion}
We obtained finite-sample high-probability CQR bounds that separate quantile
endpoint estimation from score calibration in oracle-length and realized
conditional-coverage-profile error. Under known bounded covariate shift,
endpoint transfer contributes $\wmax^{1/p}$, while the leading calibration
term is governed by $\csdiv=\E_{\DC[X]}[\wsh^2]$.

We complemented these guarantees with fixed-score calibration benchmarks on
constructed source--target marginals. The single-carrier benchmark has
expected minimax rate $\sqrt{\csdiv/m}$ for $m\ge m_\star$. For the constructed
$K$-atomic benchmark with $K\ge23$ and $m>m_\star^{(K)}$, expected upper and
lower bounds match at rate $\sqrt{\csdiv K/m}$ for every finite $p$. For every
fixed finite $p$ and $\delta<1-2e^{-K/32}$, the atomwise rule has loss at most
$C_p\delta^{-1/p}\sqrt{\csdiv K/m}$ with probability at least $1-\delta$,
whereas, for every calibration rule, the worst-case probability of loss at
least $\sqrt{\alpha(1-\alpha)\csdiv K/m}/64$ is at least $1-2e^{-K/32}$.

The numerical experiments illustrate both effective-calibration-size scalings
and the transition from calibration error to endpoint error for a
learned CQR score. All shifted results assume a known likelihood ratio. The
fixed-score benchmarks quantify calibration on the constructed problems,
while the CQR bounds propagate learner-specific endpoint error. We leave a
joint minimax analysis of score learning and calibration open.

\section*{Acknowledgments}
The work of Eric Moulines and Anton Conrad was supported by the European
Union (ERC-2022-SYGOCEAN-101071601). Views and opinions expressed are however
those of the author(s) only and do not necessarily reflect those of the
European Union or the European Research Council Executive Agency. Neither
the European Union nor the granting authority can be held responsible for them.

The work of Rustam Isaev, Sergey Samsonov, and Denis Belomestny was supported
by research project HSE-BR-2025-019 implemented as part of the Basic
Research Program at HSE University. This research was supported in part
through computational resources of HPC facilities at HSE
University~\cite{kostenetskiy2021hpc}.

The authors used OpenAI Codex to assist with language editing, consistency
checks, literature searches, \LaTeX{} formatting, and debugging and editing
the code for the numerical experiments. All AI-assisted suggestions were
reviewed and verified by the authors. The authors assume responsibility for
all content.

\clearpage
\appendix
\crefalias{section}{appendix}
\crefalias{subsection}{subappendix}
\crefalias{subsubsection}{subsubappendix}

\section{Preliminaries on Quantiles and Score Distributions}
\subsection{Notation and conventions}
\label{subsec:notation_conventions}

The \appendixcontextname\ uses the notation of
\Cref{sec:non-asymptotic-CQR,sec:cs_main}. The following decorations
distinguish conditional, marginal, empirical, source, and target score CDFs.

\paragraph{Distribution functions on the score line}
We write $F$ for an exact population CDF and $\Fhat$ for a random empirical
CDF based on a finite sample. A superscript $(\,\cdot\mid X)$ denotes a
conditional law given $X$; superscripts $\DC$, $\QC$, and $\nu$ denote source,
target, and generic marginalization over $X$, while $\wsh$ marks
likelihood-ratio weighting.
Subscripts $\Sstar$ and $\Shat$ identify the oracle and fitted scores, while
the subscript $m$ in $\Fhat^{(\Shat)}_{m}$ and $\Fhatw$ is the calibration
sample size. Thus, $F(t)$ is a marginal CDF and $F(t\mid x)$ is a conditional
CDF evaluated at $X=x$.
We write $\overline\R=\R\cup\{-\infty,+\infty\}$ with its order-Borel
$\sigma$-field. Whenever a probability CDF is evaluated at an extended-real
quantile, it is extended by $F(-\infty)=0$ and $F(+\infty)=1$. Weighted
empirical cumulative functions retain
their realized total mass; if the requested level is unattained, their
generalized inverse equals $+\infty$.

For a fixed training realization, the conditional fitted-score CDF is
\[
F^{(\Shat\mid X)}(t\mid x)
\coloneqq\PP\bigl(\Shat(X,Y)\le t\mid X=x\bigr),
\]
whereas the marginal CDF of a score $S$ under $\nu$ is
\[
F^{\nu}_{S}(t)=\int_{\Xset}F^{(S\mid X)}(t\mid x)\,\rmd\nu(x).
\]

\paragraph{Calibration radii}
The no-shift proof uses $\eta_m(\delta)$ from
\eqref{eq:no_shift_scales}; the known-shift proof uses $\etaw{\delta}$ from
\eqref{eq:cs_weighted_radius_main}.

\subsection{Primitive-mass calibration lemmas}
\label{subsec:primitive_mass_lemmas}

The generic no-shift and known-shift proofs use the next five lemmas; they
require neither bounded support nor bounded fitted endpoints.

\begin{lemma}
\label{lem:sorting_contraction_mass}
Fix $\alpha\in(0,1)$.
Let $g_{\alo},g_{\ahi},u,v\colon\Xset\to\R$ be finite measurable functions with
$u\le v$, and put
\[
a=g_{\alo}\wedge g_{\ahi},
\qquad b=g_{\alo}\vee g_{\ahi},
\qquad D=|a-u|+|b-v|.
\]
Then $a,b$ are finite measurable, $a\le b$, and
\begin{equation}
\label{eq:sort_contraction_mass}
D\le |g_{\alo}-u|+|g_{\ahi}-v|.
\end{equation}
Consequently, for every probability measure $\nu$, $p\in[1,\infty]$ and
$\varepsilon\ge0$, the bounds
$\|g_{\alo}-u\|_{L^p(\nu)}\le\varepsilon$ and
$\|g_{\ahi}-v\|_{L^p(\nu)}\le\varepsilon$ imply
\[
\|D\|_{L^p(\nu)}\le2\varepsilon,
\qquad
\|D\|_{L^1(\nu)}\le2\varepsilon.
\]
\end{lemma}

\begin{proof}
Measurability follows because minimum and maximum are continuous maps on
$\R^2$.  If $g_{\alo}\le g_{\ahi}$, \eqref{eq:sort_contraction_mass} is an
equality.  If $g_{\ahi}<g_{\alo}$, the ordered matching inequality
\[
|g_{\ahi}-u|+|g_{\alo}-v|
\le |g_{\alo}-u|+|g_{\ahi}-v|,
\qquad u\le v,
\]
follows by checking the positions of $g_{\alo},g_{\ahi}$ relative to $[u,v]$.
Minkowski's inequality gives the $L^p$ bound, also for $p=\infty$; on a
probability space $\|D\|_{L^1}\le\|D\|_{L^p}$.
\end{proof}

\begin{lemma}
\label{lem:score_cdf_perturbation_mass}
Fix $\alpha\in(0,1)$ and assume \Cref{assum:mass_regular}.
Let $\nu$ be a probability measure on $\Xset$ such that
$\nu\ll\DC[X]$.
Let $a\le b$ and $a'\le b'$ be finite measurable functions, and define
\[
F_{a,b}^{\nu}(t)=\int \DC[Y\mid X=x]([a(x)-t,b(x)+t])\,\rmd\nu(x),
\]
with $[l,r]=\varnothing$ for $l>r$, and analogously $F_{a',b'}^{\nu}$.
Then
\begin{equation}
\label{eq:score_mass_upper}
\sup_{t\in\R}|F_{a,b}^{\nu}(t)-F_{a',b'}^{\nu}(t)|
\le\up\int\bigl(|a-a'|+|b-b'|\bigr)\,\rmd\nu.
\end{equation}
Moreover, both score CDFs are $2\up$-Lipschitz: for $s<t$,
\begin{equation}
\label{eq:score_mass_lipschitz}
0\le F_{a,b}^{\nu}(t)-F_{a,b}^{\nu}(s)\le2\up(t-s),
\end{equation}
and likewise for $F_{a',b'}^{\nu}$; in particular, both are continuous.
\end{lemma}

\begin{proof}
For every finite $a\le b$ and $t\in\R$,
\begin{equation}
\label{eq:score_sublevel_interval}
\{y:\max(a-y,y-b)\le t\}=[a-t,b+t],
\end{equation}
where the right side is empty when $a-t>b+t$.  For possibly empty intervals,
\[
[l,r]\mathbin\triangle[l',r']
\subseteq[l\wedge l',l\vee l']\cup[r\wedge r',r\vee r'].
\]
The two covering intervals have lengths $|l-l'|$ and $|r-r'|$.
Applying the mass upper bound with
$l=a(x)-t$, $r=b(x)+t$, $l'=a'(x)-t$, $r'=b'(x)+t$ and integrating over
$x$ proves \eqref{eq:score_mass_upper}.

For fixed $x$ and $s<t$,
\[
\{y:s<\max(a(x)-y,y-b(x))\le t\}
\subseteq[a(x)-t,a(x)-s]\cup[b(x)+s,b(x)+t].
\]
Each interval on the right has length $t-s$. The mass upper bound and
integration over $x$ prove \eqref{eq:score_mass_lipschitz}; the argument for
$a',b'$ is identical.
\end{proof}

\begin{lemma}
\label{lem:oracle_score_growth}
Fix $\alpha\in(0,1)$ and assume \Cref{assum:mass_regular}.
Put
$a^\star=\fstar{\alo}$, $b^\star=\fstar{\ahi}$ and
\[
F^{(\Sstar\mid X)}(t\mid x)
=\DC[Y\mid X=x]([a^\star(x)-t,b^\star(x)+t]),
\qquad
F^{(\Sstar)}(t)=\int F^{(\Sstar\mid X)}(t\mid x)\,\rmd\DC[X](x).
\]
Then, for every $x\in\mathcal X_0$,
\[
F^{(\Sstar\mid X)}(0\mid x)=1-\alpha,
\qquad
b^\star(x)-a^\star(x)\ge\frac{1-\alpha}{\up}.
\]
Moreover, with $r_-=r_0\wedge(1-\alpha)/(2\up)$ and $r_+=r_0$,
\begin{align}
F^{(\Sstar)}(0)-F^{(\Sstar)}(-u)&\ge2\low u &&(0\le u\le r_-),
\label{eq:oracle_left_mass_growth}\\
F^{(\Sstar)}(v)-F^{(\Sstar)}(0)&\ge2\low v &&(0\le v\le r_+).
\label{eq:oracle_right_mass_growth}
\end{align}
\end{lemma}

\begin{proof}
The upper mass inequality at $l=r$ gives $\DC[Y\mid X=x](\{l\})=0$.  Hence the two
equal-tail identities yield
$\DC[Y\mid X=x]([a^\star(x),b^\star(x)])=\ahi-\alo=1-\alpha$, and the upper mass bound
gives the width inequality.  If $u\le r_-$, then
$a^\star(x)+u\le b^\star(x)-u$; removing the two inner boundary intervals
from $[a^\star(x),b^\star(x)]$ loses at least $2\low u$.  If $v\le r_+$,
adding the two outer boundary intervals gains at least $2\low v$.
Atomlessness removes all endpoint overlaps.  Integration over the common
full set $\mathcal X_0$ proves the marginal assertions.
\end{proof}

\begin{lemma}
\label{lem:sharp_asymmetric_inversion_mass}
Let $\alpha\in(0,1)$, $d\ge0$, $\rho,r_-,r_+>0$, and let $F$ be a CDF with
$F(0)=1-\alpha$ such that
\[
F(0)-F(-u)\ge\rho u\quad(0\le u\le r_-),
\qquad
F(v)-F(0)\ge\rho v\quad(0\le v\le r_+).
\]
Let $\hat H\colon\R\to[0,\infty)$ be nondecreasing and right-continuous,
$\lim_{t\to-\infty}\hat H(t)=0$, and
$\sup_t|\hat H(t)-F(t)|\le d$. For $s\in\R$, suppose
$\beta=1-\alpha+s>0$. If
\[
(d-s)_+<\rho r_-,
\qquad (d+s)_+\le\rho r_+,
\]
then $\{t:\hat H(t)\ge\beta\}$ is nonempty and bounded below.  Thus its
ordinary infimum is finite and
\begin{equation}
\label{eq:sharp_mass_inversion}
-\frac{(d-s)_+}{\rho}\le\quant{\beta}{\hat H}
\le\frac{(d+s)_+}{\rho}.
\end{equation}
\end{lemma}

\begin{proof}
Put $t_+=(d+s)_+/\rho$. If $d+s\ge0$, right growth and uniform
approximation give
\[
\hat H(t_+)\ge F(t_+)-d
\ge1-\alpha+\rho t_+-d=\beta,
\]
whereas $d+s<0$ gives $t_+=0$ and
$\hat H(0)\ge1-\alpha-d>\beta$. Thus the superlevel set is nonempty.

Let $a=(d-s)_+/\rho$. If $-r_-\le t<-a$, left growth gives
$\hat H(t)\le1-\alpha-\rho|t|+d<\beta$. If $t<-r_-$, monotonicity and the
strict left-window condition give
\[
\hat H(t)\le\hat H(-r_-)
\le1-\alpha-\rho r_-+d<\beta.
\]
Hence every superlevel point is at least $-a$, proving finiteness and
\eqref{eq:sharp_mass_inversion}.
\end{proof}

The strict left condition is necessary for discontinuous empirical
cumulative functions. Take $1-\alpha=1/2$, $\rho=1/2$, $r_-=r_+=1$,
$d=1/2$, $s=0$, let $F$ be the uniform CDF on $[-1,1]$, and set
\[
\hat H(t)=
\begin{cases}
0,&t<-2,\\
1/2,&-2\le t<1,\\
1,&t\ge1.
\end{cases}
\]
Then $\sup_t|\hat H(t)-F(t)|=1/2$ and the left condition holds with
equality, but $\quant{1/2}{\hat H}=-2<-1$.

\begin{lemma}
\label{lem:coverage_decomposition_mass}
Fix $\alpha\in(0,1)$ and assume \Cref{assum:mass_regular}.
Let $\nu$ be a probability measure on $\Xset$ such that
$\nu\ll\DC[X]$.
Let $a\le b$ be finite measurable functions, $q\in\R$, and
$\widehat C(x)=[a(x)-q,b(x)+q]$, with the empty-interval convention.  Put
$D(x)=|a(x)-\fstar{\alo}(x)|+|b(x)-\fstar{\ahi}(x)|$.  Then, for every
$p\in[1,\infty]$,
\begin{equation}
\label{eq:mass_coverage_decomposition}
\left\|\operatorname{Cov}_{\widehat C}-(1-\alpha)\right\|_{L^p(\nu)}
\le2\up|q|+\up\|D\|_{L^p(\nu)}.
\end{equation}
\end{lemma}

\begin{proof}
For $x\in\mathcal X_0$, compare
\[
[a(x)-q,b(x)+q]
\quad\text{and}\quad
[\fstar{\alo}(x),\fstar{\ahi}(x)].
\]
Their symmetric difference is covered by the two endpoint-movement
intervals, whose total length is at most $D(x)+2|q|$, also when the first
interval is empty.  Therefore
\[
|\operatorname{Cov}_{\widehat C}(x)-(1-\alpha)|
\le\up D(x)+2\up|q|.
\]
Taking the $L^p(\nu)$ norm proves \eqref{eq:mass_coverage_decomposition}.
\end{proof}

\section{Proofs for Global CQR}
\subsection{Proof of~Proposition~\ref{prop:target}}
\label{app:efficiency_proof}
This section gives the complete training and calibration events, DKW
allocation, signed-threshold localization, and length calculation.
We restate \Cref{prop:target} in complete form to display the two
localization inequalities used in the inversion step.
\begin{repproposition}{prop:target}
Let $\alpha\in(0,1)$, $p\in[1,\infty]$, and assume
\Cref{assum:mass_regular} and \mbox{\Cref{assum:HPD-excess-risk}}.
Fix
$\delta\in(0,1)$ and $n,m$ satisfying
\[
2\up\epsilon_p(n,\delta/4)+\eta_m(\delta)\le2\low r_-,
\qquad
2\up\epsilon_p(n,\delta/4)+\eta_m(\delta)+s_m\le2\low r_+.
\]
With probability at least $1-\delta$ over $(\Dtrain,\Dcal)$,
$\Qhat{\betam}$ is finite,
\[
-\frac{2\up\epsilon_p(n,\delta/4)+\eta_m(\delta)}{2\low}
\le\Qhat{\betam}\le
\frac{2\up\epsilon_p(n,\delta/4)+\eta_m(\delta)+s_m}{2\low},
\]
and
\[
\Lp[{\DC[X]}]{|\CChat(\cdot)|-|\CCstar(\cdot)|}
\le2\left(1+\frac{\up}{\low}\right)\epsilon_p(n,\delta/4)
+\frac{\eta_m(\delta)+s_m}{\low}.
\]
\end{repproposition}

\begin{proof}
\begin{equation}
\label{eq:primitive_Dn}
D_n(x)=|\fhat{\alo}(x)-\fstar{\alo}(x)|
+|\fhat{\ahi}(x)-\fstar{\ahi}(x)|.
\end{equation}
For $\varepsilon=\epsilon_p(n,\delta/4)$, let
\begin{equation}
\label{eq:definition-EC-global}
\mathcal E_{\rm tr}
=\bigcap_{\tau\in\{\alo,\ahi\}}
\left\{\Lp[{\DC[X]}]{\widetilde f_{n,\tau}-\fstar{\tau}}
\le\varepsilon\right\}.
\end{equation}
By \Cref{assum:HPD-excess-risk} and a union bound,
$\PP(\mathcal E_{\rm tr})\ge1-\delta/2$.  On this event,
\Cref{lem:sorting_contraction_mass} gives, also for $p=\infty$,
\begin{equation}
\label{eq:Delta_L2_bound}
\Lp[{\DC[X]}]{D_n}\le2\varepsilon,
\qquad
\Lone{D_n}[{\DC[X]}]\le2\varepsilon.
\end{equation}

Conditional on $\Dtrain$, the calibration scores are i.i.d.\ real random
variables with CDF
\[
F^{(\Shat)}(t)=\PP(\Shat(X,Y)\le t\mid\Dtrain).
\]
For the two right-continuous CDFs in the following supremum, the supremum
over $\R$ equals that over $\mathbb Q$ and is therefore measurable.
The two-sided Dvoretzky--Kiefer--Wolfowitz inequality with Massart's sharp
constant~\cite{dvoretzky1956asymptotic,massart1990tight}, applied conditionally
on $\Dtrain$, requires no continuity of this CDF and gives
\begin{equation}
\label{eq:dkw_checked}
\PP\!\left(
\left.\sup_{t\in\R}|\Fhat_m^{(\Shat)}(t)-F^{(\Shat)}(t)|
>\eta_m(\delta)\,\right|\Dtrain
\right)
\le2e^{-2m\eta_m(\delta)^2}=\frac\delta2.
\end{equation}
Define
\[
\mathcal E_{\rm DKW}
=\left\{\sup_{t\in\R}|\Fhat_m^{(\Shat)}(t)-F^{(\Shat)}(t)|
\le\eta_m(\delta)\right\}.
\]
The tower property and another union bound give
\begin{equation}
\label{eq:primitive_common_event}
\mathcal A_{n,m}(\delta)
=\mathcal E_{\rm tr}\cap\mathcal E_{\rm DKW},
\qquad
\PP(\mathcal A_{n,m}(\delta))\ge1-\delta.
\end{equation}

Let $F^{(\Sstar)}(t)=\PP(\Sstar(X,Y)\le t)$.  On
$\mathcal E_{\rm tr}$, \Cref{lem:score_cdf_perturbation_mass} with
$\nu=\DC[X]$ gives
\[
\sup_t|F^{(\Shat)}(t)-F^{(\Sstar)}(t)|
\le\up\Lone{D_n}[{\DC[X]}]
\le2\up\varepsilon.
\]
Consequently, on $\mathcal A_{n,m}(\delta)$,
\begin{equation}
\label{eq:primitive_uniform_error}
\sup_t|\Fhat_m^{(\Shat)}(t)-F^{(\Sstar)}(t)|
\le2\up\varepsilon+\eta_m(\delta).
\end{equation}
By \Cref{lem:oracle_score_growth}, $F^{(\Sstar)}(0)=1-\alpha$ and its left and right
growth constants are $2\low$ on $r_-$ and $r_+$, respectively.  Moreover,
\eqref{eq:rank_slack_exact} gives $s_m>0$ and
$\betam=1-\alpha+s_m$.  Since $F^{(\Sstar)}(r_+)\le1$, right growth at $r_+$ gives
$2\low r_+\le\alpha$; hence the right localization condition and
$2\up\varepsilon+\eta_m(\delta)\ge0$ imply
$s_m\le\alpha$ and $\betam\le1$.

Apply \Cref{lem:sharp_asymmetric_inversion_mass} to
\eqref{eq:primitive_uniform_error} with $\rho=2\low$,
$d=2\up\varepsilon+\eta_m(\delta)$, $s=s_m$, and level $\betam>0$. Since
$s_m>0$, the first localization inequality implies the strict left-window
condition $(d-s_m)_+<2\low r_-$: either $d\le s_m$ and $(d-s_m)_+=0$, or
$(d-s_m)_+=d-s_m<d\le2\low r_-$. The second localization inequality is the
right-window condition $(d+s_m)_+\le2\low r_+$. The lemma gives a nonempty,
bounded-below empirical superlevel set, hence a finite ordinary infimum,
and, since $(d-s_m)_+\le d$,
\begin{equation}
\label{eq:E3}
-\frac{2\up\varepsilon+\eta_m(\delta)}{2\low}
\le\Qhat{\betam}\le
\frac{2\up\varepsilon+\eta_m(\delta)+s_m}{2\low}.
\end{equation}

It remains to control the genuine length.  Since
$\fhat{\alo}\le\fhat{\ahi}$,
\[
|\CChat(x)|
=\bigl(\fhat{\ahi}(x)-\fhat{\alo}(x)+2\Qhat{\betam}\bigr)_+,
\qquad
|\CCstar(x)|=\fstar{\ahi}(x)-\fstar{\alo}(x).
\]
The positive-part map is $1$-Lipschitz, so, for every $x$,
\[
\bigl||\CChat(x)|-|\CCstar(x)|\bigr|
\le D_n(x)+2|\Qhat{\betam}|.
\]
This includes negative thresholds and empty learned intervals.  Therefore,
on $\mathcal A_{n,m}(\delta)$, taking the $\Lp[{\DC[X]}]{\cdot}$ norm and using
\eqref{eq:Delta_L2_bound} together with
$|\Qhat{\betam}|\le
(2\up\varepsilon+\eta_m(\delta)+s_m)/(2\low)$ from \eqref{eq:E3} gives
\[
\Lp[{\DC[X]}]{|\CChat(\cdot)|-|\CCstar(\cdot)|}
\le2\left(1+\frac{\up}{\low}\right)\varepsilon
+\frac{\eta_m(\delta)+s_m}{\low},
\]
which proves \Cref{prop:target}.
\end{proof}

\subsection{Proofs of~Theorem~\ref{th:global_cov_bound} and
Proposition~\ref{prop:global_cov_l1_direct}}
\label{sec:global_cov_bound_proof}
Section~3 of the article gives the coverage mechanism. Here we retain the
complete common-event bookkeeping, endpoint/threshold calculation, and direct
$L^1$ argument.

\begin{reptheorem}{th:global_cov_bound}
Let $\alpha\in(0,1)$, $p\in[1,\infty]$, and assume
\Cref{assum:mass_regular} and \mbox{\Cref{assum:HPD-excess-risk}}.
Fix
$\delta\in(0,1)$ and $n,m$ satisfying
\[
2\up\epsilon_p(n,\delta/4)+\eta_m(\delta)\le2\low r_-,
\qquad
2\up\epsilon_p(n,\delta/4)+\eta_m(\delta)+s_m\le2\low r_+.
\]
On the event $\mathcal A_{n,m}(\delta)$ of
\eqref{eq:primitive_common_event}, whose probability is at least
$1-\delta$,
\begin{equation}
\label{eq:cond_cov_complete}
\Lp[{\DC[X]}]{\operatorname{Cov}_{\CChat}(\cdot)-(1-\alpha)}
\le2\up\left(1+\frac{\up}{\low}\right)\epsilon_p(n,\delta/4)
+\frac{\up}{\low}\bigl(\eta_m(\delta)+s_m\bigr).
\end{equation}
\end{reptheorem}

\begin{proof}
Work on $\mathcal A_{n,m}(\delta)$, where \eqref{eq:E3} ensures that
$\Qhat{\betam}\in\R$.
On $\mathcal X_0$, \Cref{lem:oracle_score_growth} gives
$\DC[Y\mid X=x]([\fstar{\alo}(x),\fstar{\ahi}(x)])=1-\alpha$.
Apply \Cref{lem:coverage_decomposition_mass} with $\nu=\DC[X]$,
$a=\fhat{\alo}$, $b=\fhat{\ahi}$, $q=\Qhat{\betam}$ and $D=D_n$:
\begin{equation}
\label{eq:B2_decomp}
\Lp[{\DC[X]}]{\operatorname{Cov}_{\CChat}(\cdot)-(1-\alpha)}
\le2\up|\Qhat{\betam}|+\up\Lp[{\DC[X]}]{D_n}.
\end{equation}
Equations \eqref{eq:Delta_L2_bound} and \eqref{eq:E3} bound, respectively,
the endpoint and threshold terms in \eqref{eq:B2_decomp}:
\[
\up\Lp[{\DC[X]}]{D_n}\le2\up\epsilon_p(n,\delta/4),
\qquad
2\up|\Qhat{\betam}|
\le\frac{\up}{\low}\bigl(
2\up\epsilon_p(n,\delta/4)+\eta_m(\delta)+s_m\bigr).
\]
Substitution in \eqref{eq:B2_decomp} gives
\[
\Lp[{\DC[X]}]{\operatorname{Cov}_{\CChat}(\cdot)-(1-\alpha)}
\le2\up\left(1+\frac{\up}{\low}\right)\epsilon_p(n,\delta/4)
+\frac{\up}{\low}\bigl(\eta_m(\delta)+s_m\bigr),
\]
which is~\eqref{eq:cond_cov_complete}.
\end{proof}

\begin{repproposition}{prop:global_cov_l1_direct}
Let $\alpha\in(0,1)$ and assume \Cref{assum:mass_regular} and the $p=1$ instance of
\Cref{assum:HPD-excess-risk}. For every $\delta\in(0,1)$ and $n,m\ge1$, let
$\mathcal A_{n,m}(\delta)$ be the event in
\eqref{eq:primitive_common_event}, with $\mathcal E_{\rm tr}$ defined by
\eqref{eq:definition-EC-global} using $\epsilon_1(n,\delta/4)$. Then
$\PP(\mathcal A_{n,m}(\delta))\ge1-\delta$, and, on this event,
\begin{equation}
\label{eq:cond_cov_l1_direct_complete}
\Lone{\operatorname{Cov}_{\CChat}-(1-\alpha)}[{\DC[X]}]
\le4\up\epsilon_1(n,\delta/4)+\eta_m(\delta)+s_m.
\end{equation}
\end{repproposition}

\begin{proof}
Work on $\mathcal A_{n,m}(\delta)$.
Conditional on $\Dtrain$, let
\[
F^{(\Shat)}(t)=\PP\bigl(\Shat(X,Y)\le t\mid\Dtrain\bigr).
\]
The CDF $F^{(\Shat)}$ is continuous by
\Cref{lem:score_cdf_perturbation_mass}.

If $\betam\le1$, then $\Qhat{\betam}\in\R$, and the definition of the
generalized inverse together with the DKW bound gives
\[
\betam-\eta_m(\delta)
\le F^{(\Shat)}(\Qhat{\betam})
\le\betam+\eta_m(\delta).
\]
If $\betam>1$, then $\Qhat{\betam}=+\infty$ and
$|F^{(\Shat)}(\Qhat{\betam})-(1-\alpha)|=\alpha<s_m$.
Since $\betam=1-\alpha+s_m$, both cases give
\begin{equation}
\label{eq:direct_l1_calibration}
|F^{(\Shat)}(\Qhat{\betam})-(1-\alpha)|
\le\eta_m(\delta)+s_m.
\end{equation}

Set
$\widehat C_0(x)=\{y:\Shat(x,y)\le0\}
=[\fhat{\alo}(x),\fhat{\ahi}(x)]$.
For every $x$, the sets $\CChat(x)$ and $\widehat C_0(x)$ are nested; hence
\[
\Lone{\operatorname{Cov}_{\CChat}
-\operatorname{Cov}_{\widehat C_0}}[{\DC[X]}]
=\bigl|F^{(\Shat)}(\Qhat{\betam})-F^{(\Shat)}(0)\bigr|.
\]
The upper interval-mass condition makes the conditional laws atomless, so the
equal-tailed oracle has coverage $1-\alpha$. Applying
\Cref{lem:coverage_decomposition_mass} at threshold zero gives
\[
\Lone{\operatorname{Cov}_{\widehat C_0}-(1-\alpha)}[{\DC[X]}]
\le\up\Lone{D_n}[{\DC[X]}].
\]
Moreover, \eqref{eq:Delta_L2_bound} gives
$\Lone{D_n}[{\DC[X]}]\le2\epsilon_1(n,\delta/4)$. Therefore,
\begin{align*}
\Lone{\operatorname{Cov}_{\CChat}-(1-\alpha)}[{\DC[X]}]
&\le \bigl|F^{(\Shat)}(\Qhat{\betam})-(1-\alpha)\bigr|
 +\bigl|F^{(\Shat)}(0)-(1-\alpha)\bigr|\\
&\qquad
 +\Lone{\operatorname{Cov}_{\widehat C_0}-(1-\alpha)}[{\DC[X]}]\\
&\le\eta_m(\delta)+s_m+2\up\Lone{D_n}[{\DC[X]}]\\
&\le\eta_m(\delta)+s_m+4\up\epsilon_1(n,\delta/4).
\end{align*}
This is \eqref{eq:cond_cov_l1_direct_complete}.
\end{proof}

\section{Proofs for Known Covariate Shift}
\label{sec:covariate_shift}

Section~4 of the article gives the known-shift mechanism and decisive
displays. Here we verify the complete event and constant calculations for the
weighted empirical process, target length, and target profile bounds. The
separate fixed-score calibration benchmarks are proved in the next section.

\subsection{Setup and procedure}
\label{subsec:cs_setup}

The generic shifted proof assumes
\Cref{assum:mass_regular,assum:cov_shift} and
\Cref{assum:HPD-excess-risk}, and uses
the sorted endpoints in~\eqref{eq:raw_endpoint_sorting}. The target
functionals are those in~\eqref{eq:cs_target_main}; the realized profile is
evaluated under the target marginal.

\subsection{Known-Shift Calibration Lemmas}

Section~4 of the article records the unbiasedness, ordered-prefix mechanism,
and final radius; this subsection verifies the exact constants.

\paragraph{Weighted calibration deviation}
\label{subsec:cs_weighted_unif}

Fix a training realization.  Joint measurability and finiteness of the sorted
endpoints make
$(x,y)\mapsto(\wsh(x),\Shat(x,y))$ a measurable map into $\R^2$.
The calibration images are therefore i.i.d., and
\[
\FQS(t)\coloneqq(\QC[X]\otimes\DC[Y\mid X])(\Shat\le t)
=\E_{\DC[XY]}\!\bigl[\wsh\,\indi{(-\infty,t]}(\Shat)\bigm|\Dtrain\bigr].
\]
The equality follows by first integrating the indicator against
$\DC[Y\mid X=x]$ and then using
$\rmd\QC[X]=\wsh\,\rmd\DC[X]$; the normalization is deterministic and no
random denominator appears.

\begin{lemma}
\label{lem:cs_weighted_expectation}
Fix $m\ge1$ and assume \Cref{assum:cov_shift}.
Conditionally on $\Dtrain$, set
\[
Z_{m}^+\coloneqq\sup_{t\in\R}\{\Fhatw(t)-\FQS(t)\},
\qquad
Z_{m}^-\coloneqq\sup_{t\in\R}\{\FQS(t)-\Fhatw(t)\}.
\]
Then, with the universal constant $C_{\rm HL}=4$,
\begin{equation}
\label{eq:cs_expectation}
\E[Z_{m}^+\mid\Dtrain]\vee\E[Z_{m}^-\mid\Dtrain]
\le C_{\rm HL}\sqrt{\frac{\csdiv}{m}}.
\end{equation}
\end{lemma}

\begin{proof}
Write $g_t(x,y)=\wsh(x)\indi{\Shat(x,y)\le t}$.  Because the empirical and
population weighted CDFs are right-continuous, each one-sided supremum over
$\R$ equals the corresponding supremum over $\mathbb Q$. Hence the indexing
class may be taken countable and the suprema are measurable. The one-sided
consequence of the symmetrization lemma
\cite[Lemma~2.3.1, p.~108]{vandervaartwellner1996} gives, for either sign,
\[
\E[Z_{m}^\pm\mid\Dtrain]
\le2\E_{Z,\varepsilon}\sup_{t\in\R}
\frac1{m}\sum_{i=1}^{m}\varepsilon_i g_t(Z_i),
\]
where the $\varepsilon_i$ are independent Rademacher signs.  The same
symmetrization bound applies to
$Z_{m}^-$ after replacing every Rademacher sign by its negative; the joint
sign law is unchanged.  Conditional on the sample, order the finite scores
$\Shat_i$ increasingly, with any fixed ordering inside ties.  Every set
$\{i:\Shat_i\le t\}$ ends at a tie-block boundary and is therefore one of
these ordered prefixes.  Hence
\[
\sup_t\sum_i\varepsilon_i\wsh(X_i)\indi{\Shat_i\le t}
\le
\max_{0\le k\le m}
\left|\sum_{j\le k}\varepsilon_{\pi(j)}\wsh(X_{\pi(j)})\right|.
\]
The partial sums form a martingale in the Rademacher signs.  Doob's $L^2$
maximal inequality
\cite[Chapter~VII, Theorem~3.4, p.~317]{doob1953stochastic} and
Cauchy--Schwarz give conditional expectation at most
$2(\sum_i\wsh(X_i)^2)^{1/2}$.  Thus Jensen's inequality yields
\[
\E[Z_{m}^\pm\mid\Dtrain]
\le\frac4{m}\E\sqrt{\sum_{i=1}^{m}\wsh(X_i)^2}
\le4\sqrt{\frac{\csdiv}{m}}.
\]
Ties only remove possible prefix endpoints and cannot increase the maximum.
\end{proof}

\paragraph{Empirical-process refinement of the weighted deviation}
\label{subsec:cs_weighted_ep}

For $g_t=\wsh\indi{\Shat\le t}$,
\begin{equation}
\label{eq:cs_class_variance}
\sup_{t\in\R}\PVar_{\DC[XY]}(g_t)
\le\sup_t\E_{\DC[XY]}[g_t^2]
\le \csdiv\le\wmax.
\end{equation}
This variance upper bound, together with envelope $\wmax$, is the input to the
concentration inequality below.

\begin{lemma}
\label{lem:cs_weighted_ep}
Fix $m\ge1$ and $\delta\in(0,1)$, and assume \Cref{assum:cov_shift}.
Conditionally on $\Dtrain$, define
\begin{equation}
\label{eq:cs_unif_checked}
\Eunif^w\coloneqq
\left\{\sup_{t\in\R}\bigl|\Fhatw(t)-\FQS(t)\bigr|
\le\etaw{\delta}\right\}.
\end{equation}
Then
\[
\PP(\Eunif^w\mid\Dtrain)\ge1-\delta/2.
\]
\end{lemma}

\begin{proof}
The radius $\etaw{\delta}$ is defined in
\eqref{eq:cs_weighted_radius_main}.
The empirical and population weighted CDFs are right-continuous, so their
difference has the same supremum over $\mathbb Q$ as over $\R$.  This gives the
countable class required by \cite[Theorem~2.3]{bousquet2002bennett} and makes
the supremum measurable.

For $\mu_t=\E g_t$, apply that theorem separately to the centered classes
$g_t-\mu_t$ and $\mu_t-g_t$, $t\in\mathbb Q$.  Both have centered envelope
$\wmax$: indeed $g_t-\mu_t\le\wmax$, while
$\mu_t-g_t\le\mu_t\le\E\wsh=1\le\wmax$; in fact
$\abs{g_t-\mu_t}\le\wmax$.  Their variances are bounded above by
$\csdiv$ by~\eqref{eq:cs_class_variance}. In the notation of
\cite[Theorem~2.3]{bousquet2002bennett}, the two applications have
$Z=(m/\wmax)Z_{m}^{\pm}$,
$\sigma^2\le \csdiv/\wmax^2$, and
$v=m\sigma^2+2\E Z$.  Rescaling its unnormalized unit-envelope bound
gives, for either sign and every $u>0$,
\[
Z_{m}^\pm\le\E Z_{m}^\pm
+\sqrt{\frac{2u}{m}\bigl(\csdiv+2\wmax\E Z_{m}^\pm\bigr)}
+\frac{\wmax u}{3m}
\]
with failure probability at most $\rme^{-u}$.  Since
\[
\sqrt{\frac{4u\wmax\E Z_{m}^\pm}{m}}
\le\E Z_{m}^\pm+\frac{u\wmax}{m},
\]
\Cref{lem:cs_weighted_expectation} yields
\[
Z_{m}^\pm\le
2C_{\rm HL}\sqrt{\frac{\csdiv}{m}}
+\sqrt{\frac{2u\csdiv}{m}}+\frac{4u\wmax}{3m}.
\]
Take $u=u_\delta=\log(4/\delta)$. Since $u_\delta\ge\log 4$ and
$8/\sqrt{\log 4}+\sqrt{2}<9$, the right-hand side above is at most
$\etaw{\delta}$. Each sign has failure probability $\delta/4$, so a union
bound gives
$\PP(\Eunif^w\mid\Dtrain)\ge1-\delta/2$.
\end{proof}

\paragraph{Calibration threshold bound}
\label{subsec:cs_calibration}

Let
\[
\FQSstar(t)\coloneqq
\int_{\Xset}\DC[Y\mid X=x]([\fstar{\alo}(x)-t,\fstar{\ahi}(x)+t])\,
\rmd\QC[X](x).
\]
Since $\QC[X]\ll\DC[X]$, the common set $\mathcal X_0$ of
\Cref{assum:mass_regular} is also $\QC[X]$-full.  The pointwise argument of
\Cref{lem:oracle_score_growth}, now integrated under $\QC[X]$, gives
\begin{equation}
\label{eq:cs_oracle_mass_growth}
\begin{aligned}
\FQSstar(0)&=1-\alpha,\\
\FQSstar(0)-\FQSstar(-u)&\ge2\low u
&& (0\le u\le r_-),\\
\FQSstar(v)-\FQSstar(0)&\ge2\low v
&& (0\le v\le r_+).
\end{aligned}
\end{equation}
Thus the same asymmetric inversion lemma applies under the target marginal; no
density or bounded-support specialization is used.

\begin{lemma}
\label{lem:cs_exact_threshold_measurable}
Fix $\alpha\in(0,1)$ and $m\ge1$, and assume
\Cref{assum:cov_shift}.
The set $\{\wsh>0\}$ is $\QC[X]$-full.
For every fixed calibration realization and every $x$ in this set,
$m\Fhatw(+\infty)+\wsh(x)>0$, the measure in
\eqref{eq:cs_exact_weighted_measure_main} is a probability measure, its
$(1-\alpha)$-quantile satisfies the raw-level representation
\eqref{eq:cs_exact_raw_level_main}, and
$x\mapsto\Qhatw(x)$ is measurable as an $\overline\R$-valued map. In
particular, if $X_\star\sim\QC[X]$, then
\[
\wsh(X_\star)>0,
\qquad
m\Fhatw(+\infty)+\wsh(X_\star)>0
\quad\text{almost surely}.
\]
For every $a\in\R$,
\begin{equation}
\label{eq:cs_exact_threshold_sublevel}
\{x:\wsh(x)>0,\ \Qhatw(x)\le a\}
=\Bigl\{x:\wsh(x)>0,\;
(1-\alpha)\bigl(\Fhatw(+\infty)+\wsh(x)/m\bigr)
\le\Fhatw(a)\Bigr\}.
\end{equation}
\end{lemma}

\begin{proof}
The Radon--Nikodym identity gives
\[
\QC[X]\{\wsh=0\}
=\int_{\{\wsh=0\}}\wsh(x)\,\rmd\DC[X](x)
=0.
\]
Hence $\{\wsh>0\}$ is $\QC[X]$-full and $\wsh(X_\star)>0$ almost surely.
Moreover, $\wsh(X_j)\ge0$ for every $j$, so
$\Fhatw(+\infty)\ge0$. Thus, for $x\in\{\wsh>0\}$,
\[
m\Fhatw(+\infty)+\wsh(x)\ge\wsh(x)>0,
\]
and the coefficients in \eqref{eq:cs_exact_weighted_measure_main} are
nonnegative and sum to
\[
\frac{\sum_{j=1}^m\wsh(X_j)+\wsh(x)}
{m\Fhatw(+\infty)+\wsh(x)}=1.
\]
This also proves the asserted positivity at $X_\star$; no positivity of the
individual calibration weights is required.

For $x\in\{\wsh>0\}$ and $t\in\R$, the measure in
\eqref{eq:cs_exact_weighted_measure_main} assigns to $(-\infty,t]$ the mass
\[
\frac{m\Fhatw(t)}{m\Fhatw(+\infty)+\wsh(x)},
\]
which is at least $1-\alpha$ exactly when
\[
\Fhatw(t)\ge(1-\alpha)\bigl(\Fhatw(+\infty)+\wsh(x)/m\bigr).
\]
The $(1-\alpha)$-quantile of this measure is the infimum of such $t$ when
one exists; otherwise the requested level is carried by the atom at
$+\infty$ and the quantile equals $+\infty$, which matches the convention
that an unattained level has generalized inverse $+\infty$. This proves
\eqref{eq:cs_exact_raw_level_main}.

On $\{\wsh>0\}$, \eqref{eq:cs_exact_raw_level_main} and right-continuity of
the finite weighted step function give $\Qhatw(x)\le a$ exactly when its
requested raw level does not exceed $\Fhatw(a)$. The right-hand side of
\eqref{eq:cs_exact_threshold_sublevel} is measurable in the trace
$\sigma$-field on $\{\wsh>0\}$; finite sublevels generate the order-Borel
$\sigma$-field of $\overline\R$.
\end{proof}

\begin{lemma}
\label{lem:cs_calibration_bound}
Let $\alpha\in(0,1)$, $p\in[1,\infty]$, and assume
\Cref{assum:mass_regular,assum:cov_shift} and
\Cref{assum:HPD-excess-risk}. Fix $\delta\in(0,1)$ and $n,m\ge1$ satisfying
\eqref{eq:cs_localization_main}. There is an event
$\mathcal A^w_{n,m}(\delta)$, measurable with respect to
$(\Dtrain,\Dcal)$ and of probability at least $1-\delta$, on which
$\Fhatw(+\infty)>0$ and, for $\QC[X]$-almost every $x$,
\begin{equation}
\label{eq:cs_E3}
\begin{split}
\Qhatw(x)
&\ge
-\frac{
2\up\wmax^{1/p}\epsilon_p(n,\delta/4)
+(2-\alpha)\etaw{\delta}
}{2\low},\\
\Qhatw(x)
&\le
\frac{
2\up\wmax^{1/p}\epsilon_p(n,\delta/4)
+(2-\alpha)\etaw{\delta}
+(1-\alpha)\wsh(x)/m
}{2\low}.
\end{split}
\end{equation}
In particular, $\Qhatw(x)$ is finite on a target-full set and
\begin{equation}
\label{eq:cs_threshold_lp}
\Lp[{\QC[X]}]{\Qhatw}
\le
\frac{
2\up\wmax^{1/p}\epsilon_p(n,\delta/4)
+(2-\alpha)\etaw{\delta}
+(1-\alpha)\Lp[{\QC[X]}]{\wsh}/m
}{2\low}.
\end{equation}
\end{lemma}

\begin{proof}
Let $D_n$ and $\mathcal E_{\rm tr}$ be as in
\eqref{eq:primitive_Dn}--\eqref{eq:definition-EC-global}, and put
$\varepsilon=\epsilon_p(n,\delta/4)$. The endpoint assumption and sorting
contraction give $\PP(\mathcal E_{\rm tr})\ge1-\delta/2$ and, on this event,
\[
\Lp[{\DC[X]}]{D_n}\le2\varepsilon,
\qquad
\Lone{D_n}[{\DC[X]}]\le2\varepsilon.
\]
The change of measure and monotonicity of probability-space norms yield
\begin{equation}
\label{eq:cs_target_endpoint_transfer}
\Lp[{\QC[X]}]{D_n}
\le2\wmax^{1/p}\varepsilon,
\qquad
\Lone{D_n}[{\QC[X]}]
\le2\wmax^{1/p}\varepsilon.
\end{equation}
Consequently,
\begin{equation}
\label{eq:cs_cdf_control}
\sup_{t\in\R}|\FQS(t)-\FQSstar(t)|
\le2\up\wmax^{1/p}\varepsilon,
\end{equation}
and, on $\Eunif^w$,
\begin{equation}
\label{eq:cs_primitive_uniform_error}
\sup_{t\in\R}|\Fhatw(t)-\FQSstar(t)|
\le2\up\wmax^{1/p}\varepsilon+\etaw{\delta}.
\end{equation}
Set
\[
\mathcal A^w_{n,m}(\delta)
\coloneqq\mathcal E_{\rm tr}\cap\Eunif^w.
\]
The conditional calibration bound, the tower property, and a union bound give
$\PP(\mathcal A^w_{n,m}(\delta))\ge1-\delta$.

On this event,
$|\Fhatw(+\infty)-1|\le\etaw{\delta}$. The first condition in
\eqref{eq:cs_localization_main}, together with
$2\low r_-\le1-\alpha$, implies $\etaw{\delta}<1$; hence
$\Fhatw(+\infty)>0$. For fixed $x$, write
\[
\begin{aligned}
d&=2\up\wmax^{1/p}\varepsilon+\etaw{\delta},\\
s(x)&=(1-\alpha)
\bigl(\Fhatw(+\infty)-1+\wsh(x)/m\bigr).
\end{aligned}
\]
Then $1-\alpha+s(x)=(1-\alpha)
(\Fhatw(+\infty)+\wsh(x)/m)>0$, and
\begin{align*}
(d-s(x))_+
&\le2\up\wmax^{1/p}\varepsilon+(2-\alpha)\etaw{\delta},\\
(d+s(x))_+
&\le2\up\wmax^{1/p}\varepsilon+(2-\alpha)\etaw{\delta}
+(1-\alpha)\wsh(x)/m.
\end{align*}
The bound $\wsh(x)\le\wmax$ and
\eqref{eq:cs_localization_main} verify the strict left and weak right windows.
Apply \Cref{lem:sharp_asymmetric_inversion_mass} to
\eqref{eq:cs_oracle_mass_growth} and
\eqref{eq:cs_primitive_uniform_error} with $\rho=2\low$, then use
\eqref{eq:cs_exact_raw_level_main}. This proves \eqref{eq:cs_E3} and
finiteness. Since the right numerator dominates the absolute value of both
bounds, Minkowski's inequality gives \eqref{eq:cs_threshold_lp}.
\end{proof}

For finite $p$,
\begin{equation}
\label{eq:cs_weight_moment}
\Lp[{\QC[X]}]{\wsh}^{p}
=\E_{\DC[X]}[\wsh^{p+1}]
\le\wmax^{p-1}\E_{\DC[X]}[\wsh^2]
=\wmax^{p-1}\csdiv.
\end{equation}
For $p=\infty$, $\|\wsh\|_{L^\infty(\QC[X])}\le\wmax$.

\subsection{Proof of~Proposition~\ref{prop:cs_length}}
\label{subsec:cs_length}

\begin{repproposition}{prop:cs_length}
Let $\alpha\in(0,1)$ and $p\in[1,\infty]$, and assume
\Cref{assum:mass_regular,assum:cov_shift} and
\Cref{assum:HPD-excess-risk}.
Fix $\delta\in(0,1)$ and $n,m\ge1$ satisfying
\eqref{eq:cs_localization_main}.
On the event $\mathcal A^w_{n,m}(\delta)$ of
\Cref{lem:cs_calibration_bound}, the pointwise threshold bound
\eqref{eq:cs_E3} holds and
\begin{equation}
\label{eq:cs_length_complete}
\begin{aligned}
\Lp[{\QC[X]}]{\,|\CChatw(\cdot)|-|\CCstar(\cdot)|\,}
&\le2\left(1+\frac{\up}{\low}\right)
\wmax^{1/p}\epsilon_p(n,\delta/4)\\
&\quad+
\frac{
(2-\alpha)\etaw{\delta}
+(1-\alpha)\Lp[{\QC[X]}]{\wsh}/m
}{\low}.
\end{aligned}
\end{equation}
\end{repproposition}

\begin{proof}
For every $x$ at which $\Qhatw(x)$ is finite,
\[
|\CChatw(x)|
=\bigl(\fhat{\ahi}(x)-\fhat{\alo}(x)+2\Qhatw(x)\bigr)_+.
\]
The positive-part map is $1$-Lipschitz, including negative thresholds and
empty fitted intervals, so
\begin{equation}
\label{eq:cs_len_triangle}
\Lp[{\QC[X]}]{\,|\CChatw(\cdot)|-|\CCstar(\cdot)|\,}
\le\Lp[{\QC[X]}]{D_n}+2\Lp[{\QC[X]}]{\Qhatw}.
\end{equation}
Substitute \eqref{eq:cs_target_endpoint_transfer} and
\eqref{eq:cs_threshold_lp} into \eqref{eq:cs_len_triangle} and collect terms.
\end{proof}

\subsection{Proof of~Theorem~\ref{th:cs_cond_cov}}
\label{subsec:cs_theorem}

\begin{reptheorem}{th:cs_cond_cov}
Let $\alpha\in(0,1)$ and $p\in[1,\infty]$, and assume
\Cref{assum:mass_regular,assum:cov_shift} and
\Cref{assum:HPD-excess-risk}.
Fix $\delta\in(0,1)$ and $n,m\ge1$ satisfying
\eqref{eq:cs_localization_main}.
On the event $\mathcal A^w_{n,m}(\delta)$ of
\Cref{prop:cs_length},
\begin{equation}
\label{eq:cs_cond_cov_complete}
\begin{aligned}
\Lp[{\QC[X]}]{\operatorname{Cov}_{\CChatw}(\cdot)-(1-\alpha)}
&\le2\up\left(1+\frac{\up}{\low}\right)
\wmax^{1/p}\epsilon_p(n,\delta/4)\\
&\quad+
\frac{\up}{\low}
\left(
(2-\alpha)\etaw{\delta}
+\frac{1-\alpha}{m}\Lp[{\QC[X]}]{\wsh}
\right).
\end{aligned}
\end{equation}
\end{reptheorem}

\begin{proof}
On the target-full set $\mathcal X_0$, the symmetric difference between
$\CChatw(x)$ and $\CCstar(x)$ is covered by endpoint-movement intervals of
total length at most $D_n(x)+2|\Qhatw(x)|$. The upper conditional interval-mass
bound therefore gives
\begin{equation}
\label{eq:cs_B2_decomp}
\Lp[{\QC[X]}]{\operatorname{Cov}_{\CChatw}(\cdot)-(1-\alpha)}
\le\up\Lp[{\QC[X]}]{D_n}
+2\up\Lp[{\QC[X]}]{\Qhatw}.
\end{equation}
Substitute \eqref{eq:cs_target_endpoint_transfer} and
\eqref{eq:cs_threshold_lp}.
\end{proof}

\begin{remark}
\label{rem:cs_target_l1_refinement}
Under the $p=2$ endpoint assumption, Cauchy--Schwarz gives
\[
\Lone{g}[{\QC[X]}]
\le \sqrt{\csdiv}\,\Ltwo{g}[{\DC[X]}].
\]
Thus the $p=1$ results remain valid after replacing every endpoint factor
$\wmax\epsilon_1(n,\delta/4)$ by
$\sqrt{\csdiv}\,\epsilon_2(n,\delta/4)$, including in the localization
conditions. Since $\Lone{\wsh}[{\QC[X]}]=\csdiv$, the resulting length and
profile bounds contain, respectively,
\[
2\left(1+\frac{\up}{\low}\right)
\sqrt{\csdiv}\,\epsilon_2(n,\delta/4)
+\frac{(2-\alpha)\etaw{\delta}+(1-\alpha)\csdiv/m}{\low}
\]
and
\[
\begin{aligned}
2\up\left(1+\frac{\up}{\low}\right)
\sqrt{\csdiv}\,\epsilon_2(n,\delta/4)
&+\frac{\up}{\low}(2-\alpha)\etaw{\delta}\\
&+\frac{\up(1-\alpha)\csdiv}{\low m}.
\end{aligned}
\]
This refinement leaves the weighted calibration radius and its envelope
assumption unchanged.
\end{remark}

\paragraph{Exact target marginal validity}
Conditional on $\Dtrain$, the fitted score is fixed and measurable.
Corollary~1 of Tibshirani et al.~\cite{tibshirani2019conformal}, together
with the split-conformal extension described in their Section~2.2 and
supplementary material, applied to the source calibration sample and an
independent target test point, gives
\begin{equation}
\label{eq:cs_exact_marginal_validity}
\PP\!\left(
Y_\star\in\CChatw(X_\star)\,\middle|\,\Dtrain
\right)\ge1-\alpha.
\end{equation}
Equation \eqref{eq:cs_exact_marginal_validity} averages over both calibration and test
data; it does not condition on the realized calibration sample. Pournaderi
and Xiang~\cite[Theorem~1 and Corollary~1]{pournaderi2026trainingconditional}
control the upper tail of the corresponding target marginal miscoverage after
the calibration sample is also fixed. By contrast,
\Cref{th:cs_cond_cov} controls the two-sided spatial
$L^p(\QC[X])$ deviation of the complete realized profile. Its $p=1$ instance
also bounds the absolute realized marginal deviation through
\[
\left|
\int_{\Xset}\operatorname{Cov}_{\CChatw}(x)\,\rmd\QC[X](x)
-(1-\alpha)
\right|
\le
\Lone{\operatorname{Cov}_{\CChatw}(\cdot)-(1-\alpha)}[{\QC[X]}].
\]
\section{Proofs for Fixed-Score Calibration Limits}
\label{sec:fixed_score_limits_proofs}

\subsection{Scalar fixed-score calibration: Le Cam lower bound and achievability}
\label{subsec:cs_lecam}

Every supremum over $P_{Y\mid X}$ ranges over
conditional kernels for which \Cref{assum:mass_regular} holds on a measurable
set of full $\DC[X]$-measure.
All infima below range over scalar or vector thresholds measurable in the $m$
source calibration observations. We use the symbols $c_\alpha$, $m_\star$,
and $\mathcal R_{m,p}$ from \Cref{sec:calibration_lower_bounds}.
For a realized threshold function $\widehat q$, the notation
$\operatorname{Cov}^{P}_{\{\Sstar\leq\widehat q\}}$ always refers to the
conditional-coverage profile defined in
\eqref{eq:fixed_score_coverage_profile}.

\begin{reptheorem}{prop:cs_lecam_lower}
Fix $\alpha\in(0,1)$ and $\kappa>0$, and assume that $\Xset$ contains two
distinct points whose singleton sets are measurable. Then there exist source
and target marginals $\DC[X],\QC[X]$, a reference conditional kernel
$P^0_{Y\mid X}$, and a reference score $\Sstar$, all independent of $m$, with
the following properties. The marginals satisfy \Cref{assum:cov_shift} and
\[
\chi^2(\QC[X]\Vert\DC[X])=\kappa.
\]
The kernel $P^0_{Y\mid X}$ satisfies \Cref{assum:mass_regular} on a set of
full $\DC[X]$-measure. The score $\Sstar$ is the equal-tailed oracle CQR score
under $P^0_{Y\mid X}$ and is held fixed throughout the minimax problem below.
For $m\ge1$ and $p\in[1,\infty]$, define
\[
\mathcal R_{m,p}\coloneqq
\inf_{\widehat q}
\sup_{P_{Y\mid X}}
\E_{(\DC[X]\otimes P_{Y\mid X})^{\otimes m}}
\!\left[
\Lp[{\QC[X]}]{
\operatorname{Cov}^{P}_{\{\Sstar\le\widehat q\}}(\cdot)-(1-\alpha)}
\right].
\]
Set
\[
m_\star\coloneqq
\left\lceil
\frac{4(\log 2)(1+\kappa)\,
\alpha(1-\alpha)}{\min\{\alpha,1-\alpha\}^2}
\right\rceil.
\]
Then for every $m\ge m_\star$ and $p\in[1,\infty]$,
\[
\mathcal R_{m,p}
\ge c_\alpha\sqrt{\frac{1+\kappa}{m}},
\qquad
c_\alpha\coloneqq\frac18\sqrt{(\log 2)\alpha(1-\alpha)}.
\]
\end{reptheorem}

\begin{proof}
\emph{Fixed experiment and density realization.}
We construct two kernels $P^0_{Y\mid X},P^1_{Y\mid X}$ satisfying
\Cref{assum:mass_regular} with the common marginal $\DC[X]$, and write
$\DC[XY]^j=\DC[X]\otimes P^j_{Y\mid X}$. At threshold zero, their target
coverages will be $1-\alpha$ and $1-\alpha-\xi$, respectively. The total
variation distance between their $m$-fold source laws will be at most $1/2$.
Write $\E_j$ for expectation under $\DC[XY]^{j\,\otimes m}$.

Fix an arbitrary scalar threshold rule $\widehat q$.
Choose distinct points $x_0,x_D\in\Xset$ and set
\[
\QC[X]=\delta_{x_0},
\qquad
\DC[X]=\frac1{1+\kappa}\delta_{x_0}
+\frac{\kappa}{1+\kappa}\delta_{x_D}.
\]
Then $\csdiv=1+\kappa$ and
\[
\begin{aligned}
\QC[X](\{x_0\})=1,
&\qquad \QC[X](\{x_D\})=0,\\
\DC[X](\{x_0\})=\frac1{\csdiv},
&\qquad \DC[X](\{x_D\})=1-\frac1{\csdiv}.
\end{aligned}
\]
Then $\wsh=\rmd\QC[X]/\rmd\DC[X]$ equals $\csdiv$ at $x_0$ and zero at
$x_D$, and
\[
\begin{aligned}
\E_{\DC[X]}[\wsh]&=1,
&
\E_{\DC[X]}[\wsh^2]&=\csdiv,\\
\chi^2(\QC[X]\Vert\DC[X])&=\csdiv-1.
\end{aligned}
\]
A source draw reaches $x_0$ with probability $1/\csdiv$. Since
$\QC[X]=\delta_{x_0}$, for every candidate kernel $P_{Y\mid X}$,
every realized threshold $\widehat q$, and every $p\in[1,\infty]$,
\[
\Lp[{\QC[X]}]{
\operatorname{Cov}^{P}_{\{\Sstar\leq\widehat q\}}(\cdot)-(1-\alpha)}
=\bigl|
\operatorname{Cov}^{P}_{\{\Sstar\leq\widehat q\}}(x_0)-(1-\alpha)
\bigr|\eqsp.
\]

Take $P^0_{Y\mid X}$ uniform on $[-1,1]$ for every $x$, and fix
\[
\Sstar(x,y)=|y|-(1-\alpha).
\]
This is the equal-tailed oracle CQR score for $P^0_{Y\mid X}$, whose oracle
endpoints are $-1+\alpha$ and $1-\alpha$; the same score is retained under
$P^1_{Y\mid X}$. Define
\[
B_0=[-(1-\alpha),0],
\qquad
B_1=(0,\alpha].
\]
Thus the score-band lengths and reference masses are $|B_0|=1-\alpha$ and
$|B_1|=\alpha$. At $x_0$ define
\begin{equation}
\label{eq:cs_lecam_alt_density}
p^1_{Y\mid X}(y\mid x_0)
=\frac{1-\alpha-\xi}{2(1-\alpha)}\indiacc{|y|\le1-\alpha}
+\frac{\alpha+\xi}{2\alpha}\indiacc{1-\alpha<|y|\le1},
\end{equation}
and set $P^1_{Y\mid X}=P^0_{Y\mid X}$ off $x_0$. The fixed score then has
masses $(1-\alpha-\xi,\alpha+\xi)$ on $(B_0,B_1)$ under the alternative.

If $0<\xi\le\frac12\min\{\alpha,1-\alpha\}$, both densities lie between
$1/4$ and $3/4$ on $[-1,1]$. Each side of either equal-tail quantile has mass
at least $\alpha/2$, so the upper density bound places both quantiles at
distance at least $2\alpha/3$ from both support boundaries. Hence all four
one-sided neighborhoods of radius
$r_0=\alpha/2$ lie in $[-1,1]$. The density bounds give the upper
interval-mass inequality with $\up=3/4$ and the four lower inequalities with
$\low=1/4$. Thus both kernels satisfy \Cref{assum:mass_regular}.

\emph{Threshold and fraction-loss geometry.}
The action class consists only of rays $\{\Sstar\le q\}$. Clipping $q$ to
$[0,\alpha]$ cannot increase either absolute coverage error, so it is enough to
write
\[
\widetilde q\coloneqq\min\{\alpha,\max\{0,\widehat q\}\},
\qquad
a\coloneqq\frac{\widetilde q}{\alpha}\in[0,1]
\]
for the fraction of $B_1$ admitted by the ray. Under $\DC[XY]^{0}$ the split is
$(1-\alpha,\alpha)$, whereas under $\DC[XY]^1$ it is
$(1-\alpha-\xi,\alpha+\xi)$. Hence, with
$a^{\star}\coloneqq\xi/(\alpha+\xi)$,
\[
\begin{aligned}
D_0(a)\coloneqq
\operatorname{Cov}^{P^0}_{\{\Sstar\leq\alpha a\}}(x_0)-(1-\alpha)
&=\alpha a,\\
D_1(a)\coloneqq
\operatorname{Cov}^{P^1}_{\{\Sstar\leq\alpha a\}}(x_0)-(1-\alpha)
&=(\alpha+\xi)(a-a^{\star}).
\end{aligned}
\]
The zeros are $a_0=0$ and $a_1=a^{\star}$. By the point-mass target collapse,
the theorem's loss under hypothesis $j$ equals $\E_j|D_j(a)|$ for every
$p\in[1,\infty]$.

\emph{Two-point information bound.}
Write
$\chi^{2}_{\mathrm{obs}}\coloneqq
\chi^{2}(\DC[XY]^{1}\Vert\DC[XY]^{0})$.
The hypotheses differ only in the $(B_0,B_1)$ split at $x_0$.
The reference joint masses of these cells are
$\DC[X](\{x_0\})(1-\alpha)$ and $\DC[X](\{x_0\})\alpha$, and their
changes are $-\DC[X](\{x_0\})\xi$ and
$\DC[X](\{x_0\})\xi$. Since the densities are constant within each cell,
the widths cancel in the $\chi^2$ integral. Thus
\[
\chi^{2}_{\mathrm{obs}}
=\DC[X](\{x_0\})\,\xi^{2}\Bigl[\frac{1}{1-\alpha}+\frac{1}{\alpha}\Bigr]
=\frac{\xi^{2}}{\csdiv\,\alpha(1-\alpha)}\eqsp,
\]
where $1/\csdiv$ is the mass of the informative atom.

To control product total variation, use the tensorizing $\chi^2$ divergence.
For $P\ll Q$, Cauchy--Schwarz gives
\[
\mathrm{TV}(P,Q)=\frac12\int|p-q|
=\frac12\int\frac{|p-q|}{\sqrt q}\,\sqrt q
\le\frac12\sqrt{\int\frac{(p-q)^{2}}{q}}\,\sqrt{\int q}
=\frac12\sqrt{\chi^{2}(P\Vert Q)}\eqsp.
\]
Expanding the square gives
\[
\int\frac{(p-q)^{2}}{q}=\int\frac{p^{2}}{q}-2\int p+\int q\eqsp,
\qquad\text{i.e.}\qquad
1+\chi^{2}(P\Vert Q)=\int\frac{p^{2}}{q}\eqsp,
\]
and this integral factorizes on products
\cite[Section~2.4]{tsybakov2009introduction}:
\begin{equation}
\label{eq:cs_chi2_tensorization}
1+\chi^{2}\bigl(P^{\otimes m}\Vert Q^{\otimes m}\bigr)
=\int\prod_{s=1}^{m}\frac{p(z_{s})^{2}}{q(z_{s})}\,\prod_{s=1}^{m}\rmd z_{s}
=\prod_{s=1}^{m}\int\frac{p^{2}}{q}
=\bigl(1+\chi^{2}(P\Vert Q)\bigr)^{m}\eqsp.
\end{equation}
The two displays chain to
\[
\mathrm{TV}\bigl(\DC[XY]^{1\,\otimes m},\DC[XY]^{0\,\otimes m}\bigr)
\le\tfrac12\sqrt{\chi^{2}\bigl(\DC[XY]^{1\,\otimes m}\Vert\DC[XY]^{0\,\otimes m}\bigr)}
=\tfrac12\sqrt{(1+\chi^{2}_{\mathrm{obs}})^{m}-1}\eqsp,
\]
so $\mathrm{TV}\le\tfrac12$ holds once $(1+\chi^{2}_{\mathrm{obs}})^{m}\le2$. By $1+x\le e^{x}$, the budget
\begin{equation}\label{eq:cs_budget}
\chi^{2}_{\mathrm{obs}}\le\frac{\log 2}{m}
\end{equation}
suffices:
\[
(1+\chi^{2}_{\mathrm{obs}})^{m}
\le e^{m\,\chi^{2}_{\mathrm{obs}}}\le e^{\log 2}=2\eqsp,
\qquad\text{whence}\qquad
\mathrm{TV}\le\tfrac12\sqrt{2-1}=\tfrac12\eqsp.
\]

Because the per-draw divergence increases with $\xi$, the largest gap allowed
by~\eqref{eq:cs_budget} saturates it:
\[
\frac{\xi^{2}}{\csdiv\,\alpha(1-\alpha)}=\frac{\log 2}{m}
\quad\Longrightarrow\quad
\xi=\xi^{\star}\coloneqq
\sqrt{\frac{(\log 2)\csdiv\,\alpha(1-\alpha)}{m}}\eqsp,
\]
whence
$\mathrm{TV}(\DC[XY]^{0\,\otimes m},\DC[XY]^{1\,\otimes m})
\le\tfrac12$. The
theorem assumes $m\ge m_\star$, which gives
$\xi^\star\le\frac12\min\{\alpha,1-\alpha\}$; hence the density check
following \eqref{eq:cs_lecam_alt_density} applies.

\emph{Le Cam assembly.}
Write $p_0,p_1$ for the densities of
$\DC[XY]^{0\,\otimes m},\DC[XY]^{1\,\otimes m}$ against a common dominating
measure. For any statistic $a$ of the sample, the triangle inequality
$|a-a_0|+|a-a_1|\ge|a_1-a_0|=a^{\star}$ against the minimum density gives
\[
\begin{aligned}
\E_0|a-a_0|+\E_1|a-a_1|
&\ \ge\ \int\min(p_0,p_1)\,\bigl(|a-a_0|+|a-a_1|\bigr)\\
&\ \ge\ a^{\star}\int\min(p_0,p_1)
=\ a^{\star}\,(1-\mathrm{TV})\eqsp.
\end{aligned}
\]
Here the last equality is the affinity identity
$\int\min(p_0,p_1)=1-\mathrm{TV}$
\cite[Theorem~2.2]{tsybakov2009introduction}. With
$\max_j\ge\tfrac12\sum_j$ and $\mathrm{TV}\le\tfrac12$,
\[
\max_{j}\E_j|a-a_j|
\ \ge\ \frac{a^{\star}}{2}\,(1-\mathrm{TV})
\ \ge\ \frac{a^{\star}}{4}\eqsp.
\]

Both slopes are at least $\alpha$: $|D_0(a)|=\alpha\,|a-a_0|$ and
$|D_1(a)|=(\alpha+\xi)\,|a-a_1|\ge\alpha\,|a-a_1|$. Thus the two-point
inequality transfers to the deviations at the price of the smaller slope:
\[
\max_{j}\E_j\bigl|D_j(a)\bigr|
\ \ge\ \alpha\,\max_{j}\E_j|a-a_j|
\ \ge\ \frac{\alpha\,a^{\star}}{4}
\ =\ \frac{\alpha\,\xi}{4(\alpha+\xi)}
\ \ge\ \frac{\alpha\,\xi}{4\cdot2\alpha}
\ =\ \frac{\xi}{8}\eqsp,
\]
the last inequality by $\alpha+\xi\le2\alpha$, i.e.\ $\xi\le\alpha$. At the saturated $\xi=\xi^{\star}$,
\[
\frac{\xi^{\star}}{8}
=\underbrace{\tfrac1{8}\sqrt{(\log 2)\alpha(1-\alpha)}}_{=\,c_{\alpha}}
\sqrt{\frac{\csdiv}{m}}\eqsp.
\]
Both hypotheses lie in the supremum class. By the point-mass target collapse, for
every $p\in[1,\infty]$,
\[
\begin{aligned}
\sup_{P_{Y\mid X}}
&\E_{(\DC[X]\otimes P_{Y\mid X})^{\otimes m}}
\!\left[
\Lp[{\QC[X]}]{
\operatorname{Cov}^{P}_{\{\Sstar\leq\widehat q\}}(\cdot)-(1-\alpha)}
\right]\\
&\ge\max_{j\in\{0,1\}}\E_j|D_j(a)|
\ge\frac{\xi^\star}{8}\\
&=c_\alpha\sqrt{\frac{\csdiv}{m}}\eqsp.
\end{aligned}
\]
Taking the infimum over $\widehat q$ yields
\[
\mathcal R_{m,p}
\ge c_\alpha\sqrt{
\frac{\csdiv}{m}},
\qquad
c_\alpha=\frac18\sqrt{(\log 2)\alpha(1-\alpha)}.
\]
Every displayed $L^p(\QC[X])$ norm equals
$|\E_{\QC[X]}[
\operatorname{Cov}^{P}_{\{\Sstar\leq\widehat q\}}(X)]-(1-\alpha)|$, so the
same bound holds for the absolute target marginal-coverage deviation. This
proves the theorem.
\end{proof}

\begin{repproposition}{prop:cs_carrier_upper}
Under the assumptions of \Cref{prop:cs_lecam_lower}, use the source and target
marginals and fixed reference score used in its two-point construction, and let
$x_0$ denote the unique target atom. Thus $\QC[X]=\delta_{x_0}$ and
$\DC[X](\{x_0\})=\csdiv^{-1}$.
For $m\ge1$, let $N=\sum_{i=1}^m\indiacc{X_i=x_0}$, order the corresponding scores as $T_{(1)}\le\cdots\le T_{(N)}$, set $T_{(N+1)}=+\infty$, and define
\[
\widehat q^{\mathrm{car}}\coloneqq T_{(k_N)},
\qquad
k_n=\left\lceil(n+1)(1-\alpha)\right\rceil.
\]
For each admissible kernel, let $F_P(t)=P_{Y\mid X=x_0}\{\Sstar(x_0,Y)\le t\}$, with expectation taken over $\Dcal\sim(\DC[X]\otimes P_{Y\mid X})^{\otimes m}$.
\[
\begin{aligned}
&
\sup_{P_{Y\mid X}}
\E_{\Dcal}\left|F_P(\widehat q^{\mathrm{car}})-(1-\alpha)\right|\\
&\qquad\le
\frac32
\sqrt{
\frac{\csdiv}{m+1}
\left[1-\left(1-\csdiv^{-1}\right)^{m+1}\right]}\\
&\qquad\le \frac32\sqrt{\frac{\csdiv}{m+1}}.
\end{aligned}
\]
The supremum risk on the left is asymptotic to
$\sqrt{2\alpha(1-\alpha)\csdiv/(\pi m)}$ as $m/\csdiv\to\infty$.
\end{repproposition}

\begin{proof}[Proof of Proposition~\ref{prop:cs_carrier_upper}]
Put $\tau=1-\alpha$ and
\[
N=\sum_{i=1}^m\indiacc{X_i=x_0}
\sim\operatorname{Bin}(m,\csdiv^{-1}).
\]
On this construction, every calibration observation at $x_0$ and the test atom at $x_0$ have weight $\csdiv$, while every observation at $x_D$ has weight zero.
After cancelling the common positive factor, the exact weighted conformal
measure \eqref{eq:cs_exact_weighted_measure_main} becomes
\[
\frac1{N+1}\sum_{j=1}^{N}\delta_{T_{(j)}}
+\frac1{N+1}\delta_{+\infty}.
\]
Its $(1-\alpha)$-quantile is therefore the carrier rule:
\[
\Qhatw(x_0)=\widehat q^{\mathrm{car}}=T_{(k_N)},
\qquad
k_n=\lceil(n+1)\tau\rceil,
\qquad
T_{(n+1)}=+\infty.
\]
For $N=0$, the measure is $\delta_{+\infty}$; if $k_N=N+1$, its quantile is
again $+\infty$. Thus the identity covers every calibration realization.

The all-interval upper mass bound in \Cref{assum:mass_regular} makes the
conditional law atomless, while $\Sstar(x_0,\cdot)$ has level sets of
cardinality at most two; hence $F_P$ is continuous. Conditional on $N=n$, the
probability integral transform makes the transformed carrier scores
i.i.d.\ uniform. Consequently,
\[
F_P(T_{(k)})\sim\operatorname{Beta}(k,n+1-k),
\qquad 1\le k\le n;
\]
see also \cite[Proposition~2]{ramos2026transportedbeta}. Let
$B_{n,k}\sim\operatorname{Beta}(k,n+1-k)$ for $k\le n$, and set
$B_{n,n+1}=1$. Since $\QC[X](\{x_0\})=1$, the $L^p(\QC[X])$ profile loss is
$|F_P(\widehat q^{\mathrm{car}})-\tau|$ for every $p\in[1,\infty]$. Averaging over $N$ gives
\begin{equation}
\label{eq:cs_carrier_rank_exact}
\begin{aligned}
&\sup_{P_{Y\mid X}}
\E_{\Dcal}
\!\left[\left|F_P(\widehat q^{\mathrm{car}})-\tau\right|\right]\\
&\qquad=\sum_{n=0}^m \binom{m}{n}\csdiv^{-n}
\left(1-\csdiv^{-1}\right)^{m-n}
\E|B_{n,k_n}-\tau|.
\end{aligned}
\end{equation}
The right side does not depend on the kernel and equals the supremum risk of
the carrier rule.

Write $r_n=\E|B_{n,k_n}-\tau|$. If $k_n\le n$, the beta mean
$\mu_n=k_n/(n+1)$ and variance satisfy
\[
0\le\mu_n-\tau<\frac1{n+1},
\qquad
\operatorname{Var}(B_{n,k_n})
=\frac{k_n(n+1-k_n)}{(n+1)^2(n+2)}
\le\frac1{4(n+2)}.
\]
Therefore
\[
r_n
\le\sqrt{\operatorname{Var}(B_{n,k_n})}+|\mu_n-\tau|
\le\frac1{2\sqrt{n+2}}+\frac1{n+1}.
\]
If $k_n=n+1$, the ceiling identity implies
$1-\tau<1/(n+1)$, and the same bound holds because $B_{n,n+1}=1$. Hence,
for every $n\ge0$,
\begin{equation}
\label{eq:cs_conditional_carrier_finite_bound}
r_n
\le\frac1{2\sqrt{n+2}}+\frac1{n+1}
\le\frac3{2\sqrt{n+1}}.
\end{equation}
By Cauchy--Schwarz, the left side of
\eqref{eq:cs_carrier_rank_exact} is at most
\[
\frac32\E\!\left[\frac1{\sqrt{N+1}}\right]
\le\frac32\sqrt{\E\!\left[\frac1{N+1}\right]}.
\]
The binomial identity
\begin{equation}
\label{eq:cs_binomial_reciprocal_exact}
\E\!\left[\frac1{N+1}\right]
=\frac{\csdiv}{m+1}
\left[1-\left(1-\csdiv^{-1}\right)^{m+1}\right]
\end{equation}
follows from
$\binom mn/(n+1)=\binom{m+1}{n+1}/(m+1)$. This proves both finite upper
bounds.

It remains to identify the leading constant. Since
$k_n/n=\tau+O(n^{-1})$, Corollary~21.5 and Lemma~21.7 of van der Vaart
\cite{vandervaart1998} give
\[
\sqrt n\,(B_{n,k_n}-\tau)
\ \rightsquigarrow\ \mathcal N\bigl(0,\tau(1-\tau)\bigr).
\]
The moment bound above gives uniform integrability; the convention
$B_{n,n+1}=1$ applies only to finitely many $n$. Thus
\[
\sqrt n\,r_n\longrightarrow
c_\tau\coloneqq\sqrt{\frac{2\tau(1-\tau)}{\pi}}.
\]
Set $\lambda_m=m/\csdiv$. If $\lambda_m\to\infty$, then
$N/\lambda_m\to1$ in probability and
$\sqrt{\lambda_m}\,r_N\to c_\tau$ in probability. Moreover,
\eqref{eq:cs_conditional_carrier_finite_bound} and
\eqref{eq:cs_binomial_reciprocal_exact} give
\[
\E\!\left[\bigl(\sqrt{\lambda_m}\,r_N\bigr)^2\right]
\le\frac94\,\frac{m}{m+1}
\left[1-\left(1-\csdiv^{-1}\right)^{m+1}\right]
\le\frac94.
\]
Hence $\{\sqrt{\lambda_m}\,r_N\}$ is uniformly integrable.
Therefore, the left-hand side of \eqref{eq:cs_carrier_rank_exact} is asymptotic to $\sqrt{2\alpha(1-\alpha)\csdiv/(\pi m)}$ as $m/\csdiv\to\infty$.
At $\alpha=0.1$, the leading constant is $0.2394$.
\end{proof}

\subsection{A high-probability fixed-score lower bound for an atomic target}
\label{subsec:cs_lev}

Spread the perturbation over $K$ target atoms and assign one threshold value
to each atom. A Varshamov--Gilbert packing
\cite{gilbert1952comparison,varshamov1957estimate}, followed by the
$\chi^2$ form of Fano's inequality, raises the failure probability from the
two-point constant to $1-2e^{-K/32}$ while preserving the realized-coverage
loss and the factor $\csdiv$.

\begin{lemma}
\label{lem:cs_vg}
For every integer $K\ge0$ there exists $A\subseteq\{0,1\}^{K}$ with
\[
|A|\ \ge\ e^{K/8}
\qquad\text{and}\qquad
d_{\mathrm{H}}(v,v')\ \ge\ \frac K4\quad\text{for all distinct }v,v'\in A.
\]
\end{lemma}

\begin{proof}
For $K=0$ take the single point $A=\{0,1\}^{0}$. For $K\ge1$,
\cite[Lemma~4.7]{massart2007concentration} states that, for every
$\beta\in(0,1)$, there exists $A\subseteq\{0,1\}^{K}$ with
\[
d_{\mathrm{H}}(v,v')\ >\ (1-\beta)\,\frac K2\quad\text{for all distinct }v,v'\in A,
\qquad\text{and}\qquad
\ln|A|\ \ge\ \frac{\rho_{\beta}\,K}{2},
\]
where $\rho_{\beta}=(1+\beta)\ln(1+\beta)+(1-\beta)\ln(1-\beta)$; we use
$\beta$ for the parameter there called $\alpha$. Take $\beta=\tfrac12$. Then
$d_{\mathrm{H}}(v,v')>K/4$, hence $d_{\mathrm{H}}(v,v')\ge K/4$. Moreover,
$\rho_{1/2}=\tfrac32\ln\tfrac32+\tfrac12\ln\tfrac12>\tfrac14$, so
$\ln|A|\ge\rho_{1/2}K/2>K/8$ and $|A|\ge e^{K/8}$.
\end{proof}

\begin{reptheorem}{thm:cs_fano_lower}
Fix $\alpha\in(0,1)$, an integer $K\ge23$, and $\kappa>0$, and assume that
$\Xset$ contains at least $2K$ distinct points whose singleton sets are
measurable. Then there exist source and target marginals $\DC[X],\QC[X]$,
distinct points $x_1,\ldots,x_K$, a reference conditional kernel
$P^0_{Y\mid X}$, and a reference score $\Sstar$, all independent of $m$,
with the following properties. The marginals satisfy
\Cref{assum:cov_shift},
\[
\chi^2(\QC[X]\Vert\DC[X])=\kappa,
\qquad
\QC[X]=\frac1K\sum_{k=1}^K\delta_{x_k}.
\]
$P^0_{Y\mid X}$ satisfies \Cref{assum:mass_regular} on a set of full
$\DC[X]$-measure. The score $\Sstar$ is the equal-tailed oracle CQR score
under $P^0_{Y\mid X}$ and is held fixed throughout the minimax problem below.
For a threshold vector $\widehat q_{1:K}$, define its canonical measurable
extension by
\[
\widehat q(x)
=
\begin{cases}
\widehat q_k,&x=x_k\text{ for some }k\in\{1,\ldots,K\},\\
0,&x\notin\{x_1,\ldots,x_K\}.
\end{cases}
\]
The value chosen off the target support is immaterial to the loss below.
Set
\begin{equation}
\label{eq:cs_lev_admiss}
m_\star^{(K)}\coloneqq
\frac{(1+\kappa)\,\alpha(1-\alpha)K}
{4\min\{\alpha,1-\alpha\}^2}.
\end{equation}
Then for every integer $m>m_\star^{(K)}$ and every $p\in[1,\infty]$,
\[
\begin{aligned}
&\inf_{\widehat q_{1:K}}
\sup_{P_{Y\mid X}}
\PP_{(\DC[X]\otimes P_{Y\mid X})^{\otimes m}}\!\Biggl(
\Lp[{\QC[X]}]{\operatorname{Cov}^{P}_{\{\Sstar\le\widehat q\}}(\cdot)-(1-\alpha)}\\
&\hspace{11em}\ge\frac1{64}
\sqrt{\frac{(1+\kappa)\,
\alpha(1-\alpha)K}{m}}
\Biggr)
\ge1-2e^{-K/32}.
\end{aligned}
\]
\end{reptheorem}

\begin{proof}
Set
\begin{equation}
\label{eq:cs_lev_xi}
\xi\coloneqq\frac14
\sqrt{\frac{(1+\kappa)\,\alpha(1-\alpha)K}{m}}.
\end{equation}
Since $m>m_\star^{(K)}$, \eqref{eq:cs_lev_admiss} gives
$\xi<\frac12\min\{\alpha,1-\alpha\}$.

\emph{Step 1: geometry and density admissibility.}
The $2^K$ hypotheses assign one reference or alternative bit to each target
atom. The extra $K$ points required by the theorem serve only as source-only
dump atoms.
Fix distinct points $x_1,\ldots,x_K,x_{D,1},\ldots,x_{D,K}$, support both
marginals on these $2K$ points, and define
\[
\begin{aligned}
\QC[X](\{x_k\})&=\frac1K,
&\QC[X](\{x_{D,k}\})&=0,\\
\DC[X](\{x_k\})&=\frac1{(1+\kappa)K},
&\DC[X](\{x_{D,k}\})&=\frac{\kappa}{(1+\kappa)K},
\qquad k=1,\ldots,K.
\end{aligned}
\]
Thus $\csdiv=1+\kappa$, the target law is uniform on the atoms $x_k$, each
$x_{D,k}$ is source-only, and
\[
\begin{aligned}
\wsh(x_k)&=\csdiv,
&\qquad
\wsh(x_{D,k})&=0,
\\
\E_{\DC[X]}[\wsh]&=1,
&\qquad
\E_{\DC[X]}[\wsh^2]&=\csdiv.
\end{aligned}
\]
Consequently, $\chi^2(\QC[X]\Vert\DC[X])=\csdiv-1$.

Let $P^0_{Y\mid X=x}$ be uniform on $[-1,1]$ for every $x\in\Xset$, and set
$\Sstar(x,y)=|y|-(1-\alpha)$. This is the equal-tailed oracle CQR score under
$P^0_{Y\mid X}$. Define its two score bands by
\[
B_0=[-(1-\alpha),0],
\qquad
B_1=(0,\alpha].
\]
They have masses $1-\alpha$ and $\alpha$, respectively. Coarsen an
observation to a label
$T\in\{1,\ldots,K\}\times\{B_0,B_1,D\}$: the label is $(k,B_b)$ when
$X=x_k$ and $\Sstar(x_k,Y)\in B_b$, and it is $(k,D)$ when $X=x_{D,k}$.
Thus $D$ records a dump atom, not a score region at $x_k$.

Let
$\DC[XY]^{\circ}=\DC[X]\otimes P^0_{Y\mid X}$ and
$\QC[XY]^{\circ}=\QC[X]\otimes P^0_{Y\mid X}$ be the all-reference source
and target laws, and let $\overline{\DC}^{\,\circ}$ and
$\overline{\QC}^{\,\circ}$ denote their coarsenings under $T$. For each index
$k$,
\[
\begin{array}{c@{\qquad\qquad}c}
\begin{aligned}
\overline{\DC}^{\,\circ}(k,B_0)&=\frac{1-\alpha}{\csdiv\, K},\\
\overline{\DC}^{\,\circ}(k,B_1)&=\frac{\alpha}{\csdiv\, K},\\
\overline{\DC}^{\,\circ}(k,D)&=\frac{\csdiv-1}{\csdiv\, K}
\end{aligned}
&
\begin{aligned}
\overline{\QC}^{\,\circ}(k,B_0)&=\frac{1-\alpha}{K},\\
\overline{\QC}^{\,\circ}(k,B_1)&=\frac{\alpha}{K},\\
\overline{\QC}^{\,\circ}(k,D)&=0
\end{aligned}
\\[-0.25em]
\text{(source)} & \text{(target)}
\end{array}
\]

For $\tau\in\{0,1\}^K$, let $P^\tau_{Y\mid X}$ use the alternative density
\eqref{eq:cs_lecam_alt_density} at $x_k$ exactly when $\tau_k=1$ and agree
with $P^0_{Y\mid X}$ otherwise. Thus it moves conditional mass $\xi$ from
$B_0$ to $B_1$ at precisely the perturbed target atoms. Set
$\DC[XY]^\tau=\DC[X]\otimes P^\tau_{Y\mid X}$ and let
$\overline{\DC}^{\,\tau}$ be its coarsening under $T$. Then
\begin{equation}
\label{eq:cs_lev_vertexlaw}
\begin{aligned}
\overline{\DC}^{\,\tau}(k,B_0)
&=\frac{(1-\alpha)-\xi\,\indiacc{\tau_{k}=1}}{\csdiv\, K},\\
\overline{\DC}^{\,\tau}(k,B_1)
&=\frac{\alpha+\xi\,\indiacc{\tau_{k}=1}}{\csdiv\, K},\\
\overline{\DC}^{\,\tau}(k,D)
&=\frac{\csdiv-1}{\csdiv\, K}.
\end{aligned}
\end{equation}
The dump mass $(\csdiv-1)/(\csdiv\, K)$ is constant across vertices, and
$\overline{\DC}^{\,\circ}=\overline{\DC}^{\,\tau\equiv0}$.
Since
$\xi<\frac12\min\{\alpha,1-\alpha\}$, the density calculation following
\eqref{eq:cs_lecam_alt_density} applies to every vertex. Thus all vertex
kernels satisfy \Cref{assum:mass_regular} with common constants
$\low=1/4$, $\up=3/4$, and $r_0=\alpha/2$, while the marginals, target atoms,
reference conditional, and $\Sstar$ do not depend on $m$. Their coarsened
source masses are exactly \eqref{eq:cs_lev_vertexlaw}.

\emph{Step 2: action and leverage.}
For $a=(a_1,\ldots,a_K)\in[0,1]^K$, the threshold
$\alpha a_k$ at $x_k$ admits all of $B_0$ and a fraction $a_k$ of $B_1$.
Its conditional coverage is
\begin{equation}
\label{eq:cs_lev_cov}
G_k^\tau(a_k)
\coloneqq
\operatorname{Cov}^{P^\tau}_{\{\Sstar\leq\alpha a_k\}}(x_k)
=\underbrace{\bigl[(1-\alpha)-\xi\,\indiacc{\tau_{k}=1}\bigr]}_{B_0\text{-mass}}
+a_{k}\,\underbrace{\bigl[\alpha+\xi\,\indiacc{\tau_{k}=1}\bigr]}_{B_1\text{-mass}}\eqsp.
\end{equation}
Write
$G^\tau(a)\coloneqq
(G_1^\tau(a_1),\ldots,G_K^\tau(a_K))$.
Set $a^{\star}\coloneqq\xi/(\alpha+\xi)$. Then
\[
G_k^\tau(a_k)-(1-\alpha)
=
\begin{cases}
\alpha a_k, & \tau_k=0,\\
(\alpha+\xi)(a_k-a^{\star}), & \tau_k=1.
\end{cases}
\]
Thus the zero is $0$ at an unperturbed target atom and $a^{\star}$ at a
perturbed target atom.

\emph{Step 3: per-draw and product $\chi^{2}$ budget.}
Let $\PP_\tau=(\DC[XY]^\tau)^{\otimes m}$ and
$\PP_\circ=(\DC[XY]^\circ)^{\otimes m}$. Conditional on the label $T$, the
full observation $(X,Y)$ has the same distribution at every vertex: within
each score band the density is uniform, and the dump conditional is
vertex-independent. Consequently, the one-draw likelihood ratio is a
function of $T$ alone, and
\[
\chi^2(\DC[XY]^\tau\Vert\DC[XY]^\circ)
=
\chi^2(\overline{\DC}^{\,\tau}\Vert
\overline{\DC}^{\,\circ}).
\]
We therefore compute the divergence on the finite label space and then bound
every
$\chi^2(\PP_\tau\Vert\PP_\circ)$ against this common reference.

The per-draw divergence is
\[
\chi^{2}\bigl(\overline{\DC}^{\,\tau}\,\Vert\,
\overline{\DC}^{\,\circ}\bigr)
\ =\ \sum_{k=1}^{K}\ \sum_{\ell\in\{B_0,B_1,D\}}
\frac{\bigl(\overline{\DC}^{\,\tau}(k,\ell)
-\overline{\DC}^{\,\circ}(k,\ell)\bigr)^{2}}
{\overline{\DC}^{\,\circ}(k,\ell)}\eqsp.
\]
By~\eqref{eq:cs_lev_vertexlaw} the gaps in coordinate $k$ are
\[
\begin{aligned}
\overline{\DC}^{\,\tau}(k,B_0)-\overline{\DC}^{\,\circ}(k,B_0)
&= -\frac{\xi\,\indiacc{\tau_{k}=1}}{\csdiv\, K}\eqsp,\\
\overline{\DC}^{\,\tau}(k,B_1)-\overline{\DC}^{\,\circ}(k,B_1)
&= \phantom{-}\frac{\xi\,\indiacc{\tau_{k}=1}}{\csdiv\, K}\eqsp,\\
\overline{\DC}^{\,\tau}(k,D)-\overline{\DC}^{\,\circ}(k,D)
&=0\eqsp,
\end{aligned}
\]
so an unperturbed coordinate contributes zero. A perturbed coordinate contributes
\[
\begin{aligned}
&\frac{\bigl(\xi/(\csdiv\,K)\bigr)^2}
{(1-\alpha)/(\csdiv\,K)}
+\frac{\bigl(\xi/(\csdiv\,K)\bigr)^2}
{\alpha/(\csdiv\,K)}\\
&\qquad=\frac{\xi^2}{\csdiv\,K}
\left(\frac1{1-\alpha}+\frac1\alpha\right)\\
&\qquad=\frac{\xi^2}{\csdiv\,\alpha(1-\alpha)K}\eqsp;
\end{aligned}
\]
Write $h(\tau)\coloneqq\sum_{k=1}^K\tau_k$ for the Hamming weight. Therefore,
for every vertex $\tau\in\{0,1\}^{K}$,
\begin{equation}
\label{eq:cs_lev_radius}
\begin{aligned}
\chi^{2}\bigl(\DC[XY]^\tau\Vert\DC[XY]^\circ\bigr)
&=\chi^{2}\bigl(\overline{\DC}^{\,\tau}\Vert
\overline{\DC}^{\,\circ}\bigr)\\
&=\frac{h(\tau)}{K}\,
\frac{\xi^{2}}{\csdiv\,\alpha(1-\alpha)}\\
&\le\frac{\xi^{2}}{\csdiv\,\alpha(1-\alpha)}\eqsp.
\end{aligned}
\end{equation}
with equality at the all-alternative vertex $\tau\equiv1$.

By~\eqref{eq:cs_chi2_tensorization},
$\chi^{2}(\PP_{\tau}\Vert\PP_{\circ})
=\bigl(1+\chi^{2}(\DC[XY]^\tau\Vert\DC[XY]^\circ)\bigr)^{m}-1$.
Moreover, \eqref{eq:cs_lev_xi} gives
\[
\frac{m\,\xi^{2}}{\csdiv\,\alpha(1-\alpha)}
=\frac{K}{16}\eqsp,
\]
and $1+x\le e^{x}$ together with the radius~\eqref{eq:cs_lev_radius} give, for every $\tau$,
\begin{equation}
\label{eq:cs_lev_budget}
\chi^{2}\bigl(\PP_{\tau}\,\Vert\,\PP_{\circ}\bigr)
\ \le\ e^{m\,\chi^{2}(\DC[XY]^\tau\Vert\DC[XY]^\circ)}-1
\ \le\ e^{K/16}-1\eqsp.
\end{equation}

\emph{Step 4: packing, decoding, and action-class transfer.}
For $z=(z_1,\ldots,z_K)\in\R^K$, define the normalized discrete norm
\[
\lVert z\rVert_{1,K}\coloneqq\frac1K\sum_{k=1}^K|z_k|.
\]
Because $\QC[X]=K^{-1}\sum_{k=1}^K\delta_{x_k}$, any function $f$ satisfies
$\Lone{f}[{\QC[X]}]=\lVert(f(x_1),\ldots,f(x_K))\rVert_{1,K}$.

Fix an atom-specific calibration vector $a\in[0,1]^K$ measurable in the
sample. Decoding occurs in fraction space: collect the zero-error fractions,
$0$ at an unperturbed atom and $a^{\star}$ at a perturbed atom, into
\[
a^{\star}(\tau)\coloneqq\bigl(a^{\star}_{1}(\tau),\dots,a^{\star}_{K}(\tau)\bigr),
\qquad
a^{\star}_{k}(\tau)\coloneqq a^{\star}\,\indiacc{\tau_{k}=1}\ \in\ \{0,\ a^{\star}\}\eqsp.
\]
For any two vertices,
\begin{equation}
\label{eq:cs_lev_sep}
\lVert a^{\star}(\tau)-a^{\star}(\tau')\rVert_{1,K}
=\frac1K\sum_{k=1}^{K}a^{\star}\,\indiacc{\tau_{k}\ne\tau'_{k}}
=\frac{a^{\star}}{K}\,d_{\mathrm{H}}(\tau,\tau')\eqsp.
\end{equation}
For the packing $A$ of \Cref{lem:cs_vg}, distinct codewords are therefore
separated by at least $a^{\star}/4$.

The \emph{rounding test} rounds $a$ to the nearest of the $a^{\star}(\sigma)$:
\[
\mathrm{dev}(\sigma)\coloneqq
\lVert a-a^{\star}(\sigma)\rVert_{1,K}\ \ (\sigma\in A),
\qquad
\psi\coloneqq\argmin_{\sigma\in A}\mathrm{dev}(\sigma)\eqsp.
\]
Break ties by a fixed order. If $\mathrm{dev}(\tau)<a^{\star}/8$, then for
every $\sigma\ne\tau$, the triangle inequality and
\eqref{eq:cs_lev_sep} give
\[
\begin{aligned}
\mathrm{dev}(\sigma)
&\ \ge\ \lVert a^{\star}(\tau)-a^{\star}(\sigma)\rVert_{1,K}
-\mathrm{dev}(\tau)\\
&\ \ge\ \frac{a^{\star}}{4}-\mathrm{dev}(\tau)
\ >\ \frac{a^{\star}}{4}-\frac{a^{\star}}{8}\\
&\ =\ \frac{a^{\star}}{8}
\ >\ \mathrm{dev}(\tau)\eqsp.
\end{aligned}
\]
Thus $\tau$ is the strict minimizer. Consequently,
$\psi\ne\tau$ implies $\mathrm{dev}(\tau)\ge a^{\star}/8$.

The fraction and coverage losses satisfy, for every $a\in[0,1]^K$ and $\tau$,
\begin{equation}
\label{eq:cs_lev_transfer}
\bigl\lVert G^\tau(a)-(1-\alpha)\mathbf 1_K\bigr\rVert_{1,K}
\ \ge\ \alpha\,\lVert a-a^{\star}(\tau)\rVert_{1,K}\eqsp.
\end{equation}
Indeed, the two slopes in Step~2 are $\alpha$ and $\alpha+\xi$; averaging
their coordinatewise bounds gives~\eqref{eq:cs_lev_transfer}. Since
$m>m_\star^{(K)}$, \eqref{eq:cs_lev_admiss} gives
$\xi<\frac12\min\{\alpha,1-\alpha\}\le\alpha$, so
$\alpha+\xi\le2\alpha$ and
\[
\alpha\,a^{\star}=\frac{\alpha\,\xi}{\alpha+\xi}\ \ge\ \frac{\alpha\,\xi}{2\alpha}=\frac{\xi}{2}\eqsp.
\]
Chaining rounding and transfer on a fixed realization, $\psi\ne\tau$ forces
$\mathrm{dev}(\tau)\ge a^{\star}/8$, hence
\[
\bigl\lVert G^\tau(a)-(1-\alpha)\mathbf 1_K\bigr\rVert_{1,K}
\ \ge\ \alpha\,\mathrm{dev}(\tau)
\ \ge\ \frac{\alpha\,a^{\star}}{8}
\ \ge\ \frac{\xi/2}{8}
\ =\ \frac{\xi}{16}\eqsp.
\]
Thus, for every $\tau\in A$,
$\{\psi\ne\tau\}\subseteq
\{\lVert G^\tau(a)-(1-\alpha)\mathbf 1_K\rVert_{1,K}\ge\xi/16\}$, and under
$\PP_{\tau}$
\begin{equation}
\label{eq:cs_lev_contrapose}
\PP_{\tau}\Bigl(
\bigl\lVert G^\tau(a)-(1-\alpha)\mathbf 1_K\bigr\rVert_{1,K}
\ge\frac{\xi}{16}\Bigr)
\ \ge\ \PP_{\tau}(\psi\ne\tau)
\ =\ 1-\PP_{\tau}(\psi=\tau)\eqsp.
\end{equation}
The $\chi^{2}$ form of Fano's inequality is the $f(t)=(t-1)^{2}$ case of
the $f$-divergence bound of
\cite[Example~II.5, equation~(12), p.~2389]{guntuboyina2011lower}; see also
\cite{polyanskiy2025information}. For probability measures
$P_{1},\dots,P_{N}$ dominated by a probability measure $Q$ on a common
measurable space, and a measurable test $\psi$ with values in
$\{1,\dots,N\}$ and fibers $E_j\coloneqq\{z:\psi(z)=j\}$, it states
\begin{equation}
\label{eq:cs_fano_engine}
\frac1N\sum_{j=1}^{N}P_{j}(E_{j})
\ \le\ \frac1N+\sqrt{\frac{1}{N^{2}}\sum_{j=1}^{N}\chi^{2}(P_{j}\Vert Q)}\eqsp,
\end{equation}
whose left side is the average probability of correct identification. Indeed,
Cauchy--Schwarz gives
\[
P_j(E_j)-Q(E_j)
\le \sqrt{Q(E_j)\,\chi^2(P_j\Vert Q)},
\]
and summing over the disjoint fibers, using $\sum_jQ(E_j)=1$, proves
\eqref{eq:cs_fano_engine}. The radius
\eqref{eq:cs_lev_radius} and budget~\eqref{eq:cs_lev_budget} control all $N$
divergences simultaneously.

Apply~\eqref{eq:cs_fano_engine} directly to the full calibration sample, with
reference $Q=\PP_{\circ}$, hypotheses $\{\PP_{\tau}:\tau\in A\}$,
$N\coloneqq|A|\ge e^{K/8}$ (\Cref{lem:cs_vg}), and the rounding test $\psi$:
\[
\frac1N\sum_{\tau\in A}\PP_{\tau}(\psi=\tau)
\ \le\ \frac1N+\sqrt{\frac{1}{N^{2}}\sum_{\tau\in A}
\chi^{2}\bigl(\PP_{\tau}\,\Vert\,\PP_{\circ}\bigr)}\eqsp.
\]
By the budget~\eqref{eq:cs_lev_budget} each summand is $\le e^{K/16}-1\le e^{K/16}$, so
\[
\sqrt{\frac{1}{N^{2}}\sum_{\tau\in A}
\chi^{2}\bigl(\PP_{\tau}\,\Vert\,\PP_{\circ}\bigr)}
\ \le\ \sqrt{\frac{N\,e^{K/16}}{N^{2}}}
\ =\ \sqrt{\frac{e^{K/16}}{N}}
\ \le\ \sqrt{\frac{e^{K/16}}{e^{K/8}}}
\ =\ e^{-K/32},
\]
while $1/N\le e^{-K/8}\le e^{-K/32}$. Therefore
\begin{equation}
\label{eq:cs_lev_avgsuccess}
\frac1N\sum_{\tau\in A}\PP_{\tau}(\psi=\tau)\ \le\ 2\,e^{-K/32}.
\end{equation}
Averaging~\eqref{eq:cs_lev_contrapose} over $A$ and using
\eqref{eq:cs_lev_avgsuccess} gives
\[
\sup_{\tau\in A}\PP_{\tau}\Bigl(
\bigl\lVert G^\tau(a)-(1-\alpha)\mathbf 1_K\bigr\rVert_{1,K}
\ge\frac{\xi}{16}\Bigr)
\ \ge\ 1-\frac1N\sum_{\tau\in A}\PP_{\tau}(\psi=\tau)
\ \ge\ 1-2e^{-K/32}.
\]
Since $A\subseteq\{0,1\}^K$ and $a$ was arbitrary, this establishes the
Fano bound for atom-specific fractions under the normalized $\ell^1$ loss.

For a realized threshold vector $\widehat q_{1:K}$, set, for each $k$,
\[
a_k=
\begin{cases}
0,&\widehat q_k\le0,\\
\widehat q_k/\alpha,&0<\widehat q_k<\alpha,\\
1,&\widehat q_k\ge\alpha.
\end{cases}
\]
When $0\le\widehat q_k\le\alpha$, the threshold ray has
coverage~\eqref{eq:cs_lev_cov}; outside this window its absolute coverage
error dominates that of the clipped fraction. Indeed, for
$\widehat q_k\le0$ the coverage is at most $1-\alpha$ and is nondecreasing in
the threshold, whereas for $\widehat q_k\ge\alpha$ the coverage equals one.
Therefore, under every vertex
$\tau$, on every realization, and for every $p\in[1,\infty]$,
\[
\Lp[{\QC[X]}]{
\operatorname{Cov}^{P^\tau}_{\{\Sstar\leq\widehat q\}}(\cdot)-(1-\alpha)}
\ge
\bigl\lVert G^\tau(a)-(1-\alpha)\mathbf 1_K\bigr\rVert_{1,K}.
\]
Here the right side is an $L^1$ loss; the inequality follows from the
coordinatewise clipping comparison and monotonicity of $L^p(\QC[X])$ over
the probability measure $\QC[X]$. The cube conditionals satisfy
\Cref{assum:mass_regular}, while the marginals, target atoms, reference
conditional, and $\Sstar$ remain fixed.
The fraction-space bound therefore transfers to the theorem's action and
supremum classes:
\[
\begin{aligned}
&\inf_{\widehat q_{1:K}}
\sup_{P_{Y\mid X}}
\PP_{(\DC[X]\otimes P_{Y\mid X})^{\otimes m}}\!\Biggl(
\Lp[{\QC[X]}]{
\operatorname{Cov}^{P}_{\{\Sstar\leq\widehat q\}}(\cdot)-(1-\alpha)}\\
&\hspace{11em}\ge\frac1{64}
\sqrt{\frac{\csdiv\,
\alpha(1-\alpha)K}{m}}
\Biggr)
\ge1-2e^{-K/32}.
\end{aligned}
\]
This completes the proof of \Cref{thm:cs_fano_lower}.
\end{proof}

\subsection{Split conformal calibration at each target atom}
\label{subsec:cs_cellwise_upper}

\begin{repproposition}{prop:cs_cellwise_upper}
Under the assumptions of \Cref{thm:cs_fano_lower}, consider the source and
target marginals, target atoms $x_1,\ldots,x_K$, and fixed reference score
used in its $K$-atomic construction. In particular,
\[
\QC[X]=\frac1K\sum_{k=1}^K\delta_{x_k},
\qquad
\DC[X](\{x_k\})=\frac1{\csdiv K},
\quad k=1,\ldots,K.
\]
For every $p\in[1,\infty)$, there is a
constant $C_p<\infty$, depending only on $p$, such that the rule
\eqref{eq:cs_cellwise_rule}, with $\widehat q$ defined by the canonical
extension in \Cref{thm:cs_fano_lower}, satisfies, for every $m\ge1$,
\[
\sup_{P_{Y\mid X}}
\left\{
\E_{\Dcal}\!\left[
\left(
\Lp[{\QC[X]}]{
\operatorname{Cov}^{P}_{\{\Sstar\le\widehat q\}}(\cdot)-(1-\alpha)}
\right)^p
\right]
\right\}^{1/p}
\le C_p\sqrt{\frac{\csdiv K}{m}}.
\]
\end{repproposition}

\begin{proof}
Fix an admissible $P_{Y\mid X}$ and put $\tau=1-\alpha$. Throughout the proof,
$C_p$ denotes a finite constant depending only on $p$ and may increase from
line to line. For each target atom, let
\[
F_k(t)=P_{Y\mid X=x_k}\{\Sstar(x_k,Y)\le t\}.
\]
The upper interval-mass bound in \Cref{assum:mass_regular} and the form of
$\Sstar$ make $F_k$ continuous. Conditional on $N_k=n$, the probability
integral transform therefore gives
\[
F_k(\widehat q_k)\ \sim\ B_{n,j_n},
\qquad
j_n=\lceil(n+1)\tau\rceil,
\]
where $B_{n,j}$ has the $\operatorname{Beta}(j,n+1-j)$ distribution for
$j\le n$, and $B_{n,n+1}=1$.

We first record a uniform moment bound. If $j_n=n+1$, then
$\alpha<1/(n+1)$ and
$|B_{n,j_n}-\tau|^p\le(n+1)^{-p}$. If $j_n\le n$, set
$\mu_n=j_n/(n+1)$. Then $0\le\mu_n-\tau<1/(n+1)$. For admissible values of
$t>0$, the binomial representation of the $j_n$th uniform order statistic
and Hoeffding's inequality give
\[
\begin{aligned}
\PP(B_{n,j_n}\ge\mu_n+t)
&=\PP\bigl(\operatorname{Bin}(n,\mu_n+t)\le j_n-1\bigr)
\le e^{-2nt^2},\\
\PP(B_{n,j_n}\le\mu_n-t)
&=\PP\bigl(\operatorname{Bin}(n,\mu_n-t)\ge j_n\bigr)
\le e^{-2nt^2}.
\end{aligned}
\]
Indeed,
$n(\mu_n+t)-(j_n-1)=nt+1-\mu_n\ge nt$ and
$j_n-n(\mu_n-t)=nt+\mu_n\ge nt$; outside the admissible range the
corresponding tail probability is zero. Hence
\[
\E|B_{n,j_n}-\mu_n|^p
\le 2p\int_0^\infty t^{p-1}e^{-2nt^2}\,\rmd t
\le C_p^p n^{-p/2}.
\]
Absorbing the bias $|\mu_n-\tau|<(n+1)^{-1}$ and treating the case
$j_n=n+1$ above yields
\begin{equation}
\label{eq:cs_cellwise_beta_moment}
\E|B_{n,j_n}-\tau|^p
\le C_p^p(n+1)^{-p/2},
\qquad n\ge0.
\end{equation}

Here $N_k\sim\operatorname{Bin}(m,1/(\csdiv K))$. Write
$\lambda=m/(\csdiv K)$. If $\lambda\le1$, then
$\E[(N_k+1)^{-p/2}]\le1\le\lambda^{-p/2}$. If $\lambda>1$, the binomial
lower-tail bound gives $\PP(N_k\le\lambda/2)\le e^{-\lambda/8}$, whence
\begin{equation}
\label{eq:cs_cellwise_count_moment}
\E[(N_k+1)^{-p/2}]
\le (2/\lambda)^{p/2}+e^{-\lambda/8}
\le C_p^p\lambda^{-p/2}.
\end{equation}
Since $\QC[X]$ is uniform on the target atoms and the values of $\widehat q$
away from them do not contribute to its $L^p(\QC[X])$ loss,
\[
\begin{aligned}
&\E_{\Dcal}\!\left[
\left(
\Lp[{\QC[X]}]{
\operatorname{Cov}^{P}_{\{\Sstar\le\widehat q\}}-(1-\alpha)}
\right)^p
\right]\\
&\qquad=
\frac1K\sum_{k=1}^K
\E_{\Dcal}|F_k(\widehat q_k)-\tau|^p
\le C_p^p\left(\frac{\csdiv K}{m}\right)^{p/2}.
\end{aligned}
\]
The bound is independent of the admissible kernel. Taking the supremum and
the $p$th root proves the proposition.
\end{proof}

\section{Additional Details for the Numerical Experiments}
\label{sec:numerical_carrier_details}
Experiments E1 and E2 evaluate the fixed-score calibration benchmarks,
whereas E3 evaluates learned CQR under a continuous shift. They therefore
refer to distinct statistical experiments.

Experiments E1 and E2 use $\tau=1-\alpha$ and the fixed reference score
$\Sstar(x,y)=|y|-(1-\alpha)$. The probability integral transform reduces the
coverage error to a beta order statistic. Put
$k_n=\lceil(n+1)\tau\rceil$, let
$B_{n,k}\sim\operatorname{Beta}(k,n+1-k)$ for $k\le n$, and set
$B_{n,n+1}=1$. The resulting risks are independent of the admissible
conditional kernel.

\subsection{Scalar carrier experiment}

For the two-point construction
\[
 \QC[X]=\delta_{x_0},\qquad
 \DC[X]=\frac{1}{1+\kappa}\delta_{x_0}
 +\frac{\kappa}{1+\kappa}\delta_{x_D},
\]
we have $\csdiv=1+\kappa$ and
$N\sim\operatorname{Bin}(m,(1+\kappa)^{-1})$ carrier observations. Writing
$r_n=\E|B_{n,k_n}-\tau|$, with $r_n=\alpha$ when $k_n=n+1$, the exact
worst-kernel risk of the weighted rule is
\[
 R(m,\kappa)=\sum_{n=0}^m\binom mnq^n(1-q)^{m-n}r_n,
 \qquad q=(1+\kappa)^{-1}.
\]
For $B\sim\operatorname{Beta}(a,b)$, $\mu=a/(a+b)$, and regularized
incomplete beta function $I_x(a,b)$, we evaluate
\[
 \E|B-\tau|
 =2\tau I_\tau(a,b)-2\mu I_\tau(a+1,b)+\mu-\tau.
\]

E1 takes $\alpha=0.1$ and $\kappa\in\{1,3,9,27\}$. The blue curves in
\Cref{fig:e1_scalar} are exact Binomial--Beta mixtures. When $k_N=N+1$, the
threshold equals $+\infty$ and the conditional error equals $\alpha$, causing
the initial plateau. In effective-size coordinates, the Le Cam lower bound
starts at
\[
 \frac{m}{1+\kappa}
 \ge\frac{4(\log 2)\alpha(1-\alpha)}{\min\{\alpha,1-\alpha\}^2}
 \approx25.0,
\]
up to the integer ceiling in \eqref{eq:cs_lecam_threshold}. Moreover,
the leading-term calculation in the proof of \Cref{prop:cs_carrier_upper}
gives, as $m/(1+\kappa)\to\infty$,
\[
 \sqrt{\frac{m}{1+\kappa}}R(m,\kappa)
 \longrightarrow\sqrt{\frac{2\alpha(1-\alpha)}{\pi}}.
\]
Panel~(c) varies $\alpha\in\{0.05,0.1,0.2\}$ at effective size $10^4$.
Constants in the finite and minimax bounds are not optimized.

\begin{figure*}[t]
  \centering
  \includegraphics[width=\textwidth]{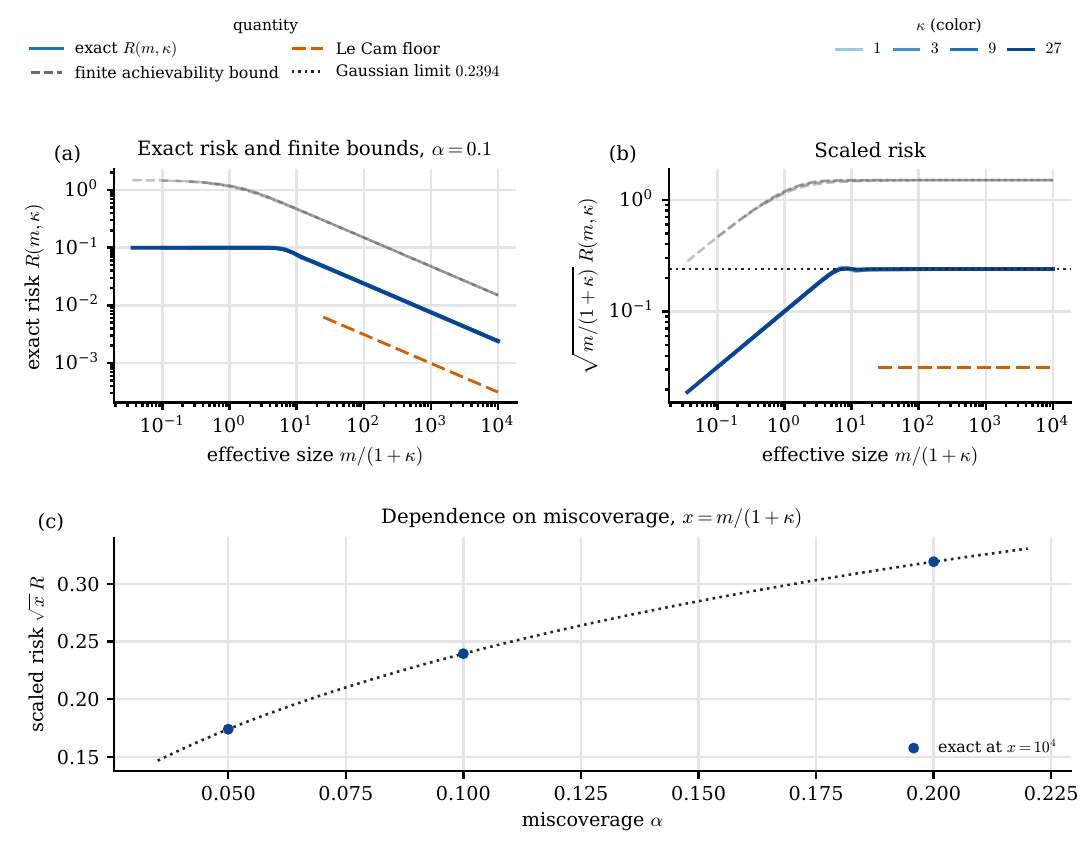}
  \caption{Scalar carrier experiment (E1) at $\alpha=0.1$. All curves are
  exact evaluations of the Binomial--Beta identity for
  $\kappa\in\{1,3,9,27\}$; no simulation is used.}
  \label{fig:e1_scalar}
\end{figure*}

\FloatBarrier

\subsection{Atomwise carrier experiment}

For the $K$-atomic construction,
\[
 \QC[X](\{x_k\})=\frac1K,\qquad
 \DC[X](\{x_k\})=\frac1{(1+\kappa)K},
 \qquad
 \DC[X](\{x_{D,k}\})=\frac{\kappa}{(1+\kappa)K}.
\]
The atomwise rule \eqref{eq:cs_cellwise_rule} has rank
$j_n=\lceil(n+1)\tau\rceil$ at count $n$. With
$q=((1+\kappa)K)^{-1}$ and
$N_D=m-\sum_{k=1}^KN_k$,
\[
 (N_1,\ldots,N_K,N_D)
 \sim\operatorname{Multinomial}
 \left(m;q,\ldots,q,\frac{\kappa}{1+\kappa}\right).
\]
Conditionally on the counts, the atom-specific errors are independent and
\[
 e_k(\widehat q_k)\mid\{N_k=n\}
 \stackrel{\mathrm d}{=}B_{n,j_n}-\tau.
\]
For $1\le p<\infty$, set
\[
 L_p=\left(\frac1K\sum_{k=1}^K|e_k(\widehat q_k)|^p\right)^{1/p},
 \qquad
 M_p=\{\E L_p^p\}^{1/p}.
\]
Writing $r_n^{(p)}=\E|B_{n,j_n}-\tau|^p$ gives the exact identity
\[
 M_p=\left\{
 \sum_{n=0}^m\binom mnq^n(1-q)^{m-n}r_n^{(p)}
 \right\}^{1/p}.
\]
The Monte Carlo curves estimate $\E L_p$ by sampling multinomial counts and
conditional beta variables; Jensen's inequality gives $\E L_p\le M_p$.

Let $\lambda=m/(\csdiv K)$, $\sigma_\alpha^2=\alpha(1-\alpha)$, and let
$Z_1,\ldots,Z_K$ be independent standard normal variables. For fixed $K$, as
$\lambda\to\infty$, conditional independence, the joint convergence
$N_k/\lambda\to1$ in probability for $k=1,\ldots,K$, and the beta quantile
central limit theorem give
\[
 \sqrt\lambda\bigl(e_1(\widehat q_1),\ldots,e_K(\widehat q_K)\bigr)
 \rightsquigarrow
 \sigma_\alpha(Z_1,\ldots,Z_K).
\]
The moment bounds \eqref{eq:cs_cellwise_beta_moment} and
\eqref{eq:cs_cellwise_count_moment}, applied at arbitrarily large exponents,
give uniform integrability of the normalized losses. Hence, for
$1\le p<\infty$,
\[
 \sqrt\lambda M_p\longrightarrow
 \sigma_\alpha\{\E|Z_1|^p\}^{1/p},
 \qquad
 \sqrt\lambda\E L_p\longrightarrow
 \sigma_\alpha\E\left(\frac1K\sum_{k=1}^K|Z_k|^p\right)^{1/p}.
\]
For $L_\infty=\max_k|e_k(\widehat q_k)|$, the same argument gives
\[
 \sqrt\lambda\E L_\infty\longrightarrow
 \sigma_\alpha\E\max_{k\le K}|Z_k|.
\]
Finally, for every fixed $\varepsilon\in(0,1)$, Mills' ratio gives the lower
bound below, while the Gaussian moment generating function gives the upper
bound:
\[
 (1-\varepsilon)\sqrt{2\log K}\{1-o(1)\}
 \le \E\max_{k\le K}|Z_k|
 \le \sqrt{2\log(2K)}.
\]
Letting $\varepsilon\downarrow0$ yields
$\E\max_{k\le K}|Z_k|\sim\sqrt{2\log K}$.

E2 takes $\alpha=0.1$, $\kappa\in\{1,3,9\}$, and
$K\in\{23,64,256\}$. Gray lines in \Cref{fig:e2_cellwise} are exact $M_p$;
hollow markers estimate $\E L_p$ from $10^4$ replications per point with seed
$20260826$. Only counts and beta variables are simulated. Panels~(b1)--(b3)
show $p\in\{1,2,8\}$ and the finite-$K$ Jensen gap. Panel~(c) uses
$K\in\{23,64,256,1024\}$, $\kappa=3$, and $\lambda=10^4$ to illustrate the
$\sqrt{\log K}$ behavior. Its conclusion is empirical.

\begin{figure*}[t]
  \centering
  \includegraphics[width=\textwidth]{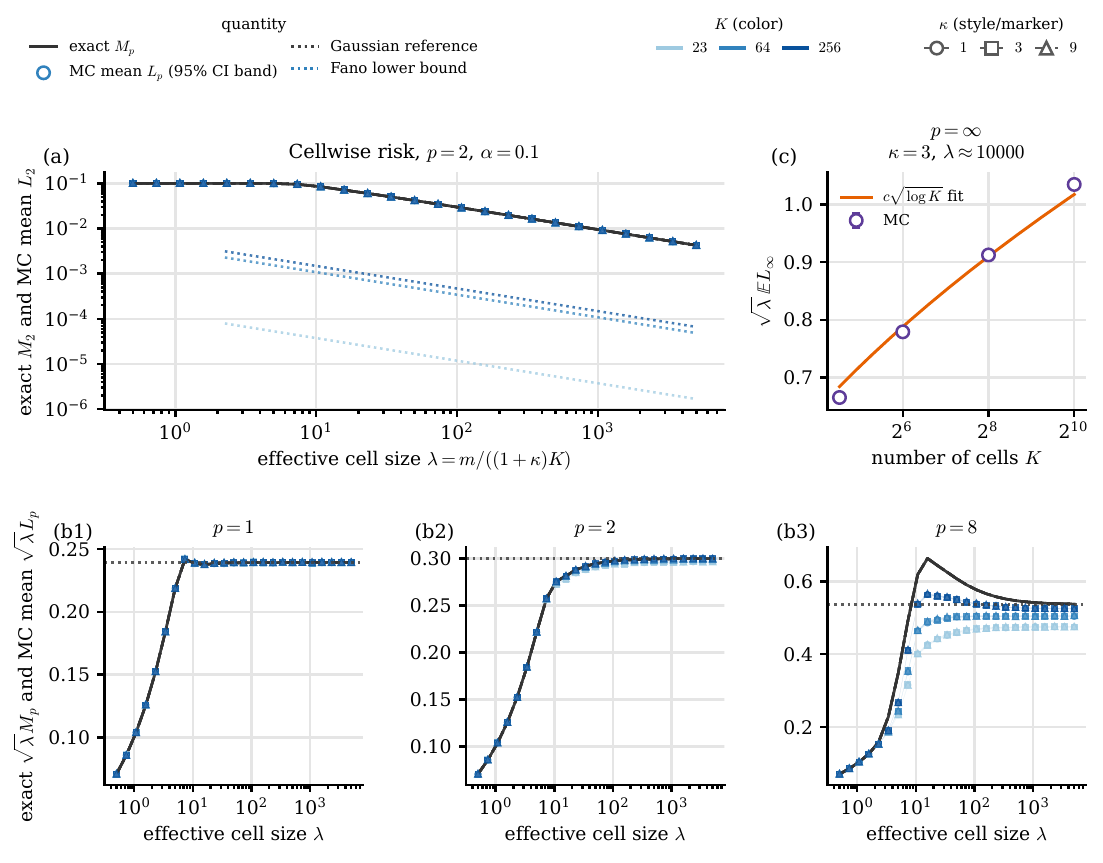}
  \caption{Atomwise carrier experiment (E2) at $\alpha=0.1$. Gray lines are
  exact moments $M_p$; hollow markers are Monte Carlo estimates of $\E L_p$
  from $10^4$ replications with pointwise $95\%$ bands.}
  \label{fig:e2_cellwise}
\end{figure*}

\FloatBarrier

\subsection{Learned CQR computational details}

For E3 in \Cref{subsec:numerical_cqr}, write
$\mu(x)=\sin(2\pi x)$ and
$\sigma(x)=1/2+\cos(2\pi x)/4$. For a realized interval
$C(x)=[\ell_C(x),u_C(x)]$, exact Gaussian conditional coverage is
\[
 \operatorname{Cov}_{C}(x)
 =
 \Phi\!\left(\frac{u_C(x)-\mu(x)}{\sigma(x)}\right)
 -\Phi\!\left(\frac{\ell_C(x)-\mu(x)}{\sigma(x)}\right),
\]
with the natural values for empty sets and infinite endpoints. Composite
trapezoidal quadrature on $x_j=j/1024$, $j=0,\ldots,1024$, evaluates all target
integrals; Monte Carlo randomness is confined to the source training and
calibration folds, and the optimizer stream is fixed across replications.

For $R=200$ replications and seed $20260826$, panel~(a) of
\Cref{fig:e3_cqr} reports means with pointwise
$\overline T\pm1.96\,\widehat{\operatorname{se}}(\overline T)$ intervals,
truncated to $[0,1]$. Panels~(b)--(c) report medians with pointwise empirical
$[0.025,0.975]$-quantile bands. All realized exact weighted thresholds entering
the length and conditional-profile panels are finite.

\section{Neural Network Approximation, Metric Entropy, and Estimation}
\label{sec:nn_definitions}
\subsection{Neural network classes and notation}\label{sec:nn-classes}
This subsection fixes the pinball-risk, H\"older, and sparse-network notation
used in \Cref{theo:requ_rates} and its proof.

\paragraph{Pinball risk and H\"older class}
For $\tau\in(0,1)$, define
\[
\pinball(u)\coloneqq u\bigl(\tau-\indiacc{u<0}\bigr),
\qquad
\Risk{f}\coloneqq
\E_{(X,Y)\sim\DC[XY]}\bigl[\pinball(Y-f(X))\bigr].
\]
The empirical risk $\EmpRisk{f}$ is the sample average of the same loss over
$\Dtrain$.
For $\beta>0$, set $\kbeta\coloneqq\ceil{\beta}-1$, the largest integer
strictly smaller than $\beta$.  The H\"older ball of radius $H>0$ on
$\Omega\subset\R^d$ is
\begin{equation}
\label{eq:holder-ball}
\HC^\beta(\Omega,H)
\coloneqq
\Biggl\{f:\Omega\to\R:
\sum_{|\bgamma|\le\kbeta}
\supnorm{\partial^{\bgamma}f}
+\sum_{|\bgamma|=\kbeta}
\sup_{\substack{\mathbf x,\mathbf y\in\Omega\\\mathbf x\ne\mathbf y}}
\frac{\abs{\partial^{\bgamma}f(\mathbf x)-\partial^{\bgamma}f(\mathbf y)}}
{\supnorm{\mathbf x-\mathbf y}^{\beta-\kbeta}}
\le H\Biggr\}.
\end{equation}

For a function $f$ on $\Xset$, $\supnorm{f}$ denotes the uniform norm
$\sup_{x\in\Xset}|f(x)|$.

\paragraph{Sparse-ReLU class}
Let $\relu(t)=\max(t,0)$ and
$\relu_{\mathbf v}(\mathbf y)
=(\relu(y_1-v_1),\ldots,\relu(y_r-v_r))^\top$.
For an architecture $(L,\bp)$ with
$\bp=(p_0,\ldots,p_{L+1})\in\N^{L+2}$, define the network realization
recursively by
\begin{equation}
\label{eq:nn-def}
\begin{aligned}
\mathbf h_0(\mathbf x)&\coloneqq\mathbf x,\\
\mathbf h_\ell(\mathbf x)
&\coloneqq\relu_{\mathbf v_\ell}
\bigl(W_{\ell-1}\mathbf h_{\ell-1}(\mathbf x)\bigr),
\qquad 1\le\ell\le L,\\
f_\theta(\mathbf x)&\coloneqq W_L\mathbf h_L(\mathbf x).
\end{aligned}
\end{equation}
Here $W_\ell\in\R^{p_{\ell+1}\times p_\ell}$,
$\mathbf v_\ell\in\R^{p_\ell}$, $\mathbf v_0\coloneqq\mathbf 0$,
$p_0=d$, and $p_{L+1}=1$.

\paragraph{Output truncation}
Define the output truncation by
\begin{equation}\label{eq:truncation-def}
\TM(x)\coloneqq\min\{\max(x,-M),M\}.
\end{equation}
As the Euclidean projection onto $[-M,M]$, $\TM$ is $1$-Lipschitz.
Consequently, output truncation preserves Lipschitz continuity and does not
increase the metric entropy of the class.
It also cannot increase pointwise error relative to any function taking
values in $[-M,M]$.

Writing $\supnorm{W}$ and $\lzeronorm{W}$ for the max-entry norm and number
of nonzero entries, define the clipped sparse class by
\begin{equation}
\label{eq:sh-class}
\NNclassSH
\coloneqq
\left\{
\TM\circ f_\theta:f_\theta\text{ is of the form \eqref{eq:nn-def}},\
\begin{aligned}
&\max_{0\le j\le L}
\bigl\{\supnorm{W_j}\vee\abs{\mathbf v_j}_\infty\bigr\}\le1,\\
&\sum_{j=0}^L
\bigl(\lzeronorm{W_j}+\abs{\mathbf v_j}_0\bigr)\le s
\end{aligned}
\right\}.
\end{equation}
Every $f\in\NNclassSH$ satisfies $\supnorm{f}\le M$.

\paragraph{Schmidt-Hieber architecture for nonparametric rates}
For input dimension $d$, smoothness $\beta > 0$, and sample size $n$,
we define the Schmidt-Hieber width parameter
\begin{equation}\label{eq:sh-scaling}
N_n \;:=\; \bigl\lceil n^{d/(2\beta + d)} \bigr\rceil,
\end{equation}
and the Schmidt-Hieber architectural parameters
\begin{subequations}\label{eq:sh-arch}
\begin{align}
L_{\mathrm{SH}}(n)
\;&:=\; 8 + \bigl(\lceil \log_2 n \rceil + 5\bigr) \bigl(1 + \lceil \log_2(d \vee \beta) \rceil \bigr),
\label{eq:sh-L}\\
W_{\mathrm{SH}}(n)
\;&:=\; 6(d + \lceil \beta \rceil) \, N_n,
\label{eq:sh-W}\\
\bp_{\mathrm{SH}}(n)
\;&:=\; \bigl(d, \underbrace{W_{\mathrm{SH}}(n), \ldots, W_{\mathrm{SH}}(n)}_
       {L_{\mathrm{SH}}(n) \text{ terms}}, 1 \bigr),
\label{eq:sh-p}\\
S_{\mathrm{SH}}(n)
\;&:=\; \left\lceil141 (d + \beta + 1)^{3+d} \, N_n \,
\bigl(\lceil \log_2 n \rceil + 6\bigr)\right\rceil.
\label{eq:sh-S}
\end{align}
\end{subequations}
These are the depth, width, and sparsity parameters supplied by
\cite[Theorem~5]{schmidthieber_2020} with
width parameter $N_n$ and
discretization parameter $\lceil\log_2 n\rceil$. They define the class used
in \Cref{theo:requ_rates}.

The resulting clipped class is
\begin{equation}
\label{eq:NNSHn}
\NNclassSHn{L_{\mathrm{SH}}}{\bp_{\mathrm{SH}}}{S_{\mathrm{SH}}}(n)
\coloneqq
\NNclassSHn{L_{\mathrm{SH}}(n)}{\bp_{\mathrm{SH}}(n)}{S_{\mathrm{SH}}(n)}.
\end{equation}
Its depth is $O(\log n)$, its maximal width is
$O(n^{d/(2\beta+d)})$, and its sparsity is
$O(n^{d/(2\beta+d)}\log n)$.
Define the training threshold by
\begin{equation}
\label{eq:n_sh}
n_{\mathrm{SH}}
\coloneqq\left\lceil
(\beta+1)^{2\beta+d}
\vee(\lipconst_\beta+1)^{(2\beta+d)/d}\rme^{2\beta+d}
\vee\rme^{3(2\beta+d)/(2\beta)}
\right\rceil.
\end{equation}

\subsection{Lipschitz stability of network parameters}

\begin{proposition}
\label{lem:inf_norm_bound}
Fix $L\in\N$ and an architecture
$\bp=(p_0,\ldots,p_{L+1})\in\N^{L+2}$.  Let
$\theta^{(m)}=(W_\ell^{(m)},v_\ell^{(m)})_\ell$, $m\in\{1,2\}$, have all
parameters in $[-1,1]$.  The corresponding realizations
$\func{\param^{(1)}}$ and $\func{\param^{(2)}}$ satisfy
\begin{equation}
\label{eq:inf_norm_bound}
\begin{aligned}
\Delta_W&\coloneqq\max_{0 \leq \ell \leq L}
\|W_\ell^{(1)} - W_\ell^{(2)}\|_\infty,\\
\Delta_v&\coloneqq\max_{1 \leq \ell \leq L}
\|v_\ell^{(1)} - v_\ell^{(2)}\|_\infty,\\
\sup_{x \in [0, 1]^{p_0}}
|\func{\param^{(1)}}(x) - \func{\param^{(2)}}(x)|
&\leq (\Delta_W\vee\Delta_v)(L+1)\prod_{j=0}^{L}(p_j+1).
\end{aligned}
\end{equation}
\end{proposition}

The proof needs a bound on the intermediate layers.
\begin{lemma}
\label{lem:layer_output_bound}
Fix $L\in\N$ and $(p_0,\ldots,p_L)\in\N^{L+1}$.
For $x\in[0,1]^{p_0}$, let $h_0(x)=x$ and
\[
h_k(x)=\relu_{v_k}\bigl(W_{k-1}h_{k-1}(x)\bigr),
\qquad k=1,\ldots,L,
\]
where $W_{k-1}\in[-1,1]^{p_k\times p_{k-1}}$ and
$v_k\in[-1,1]^{p_k}$.  Then, for $1\le k\le L$,
\[
\sup_{x\in[0,1]^{p_0}}\norm{h_k(x)}[\infty]
\le\prod_{j=0}^{k-1}(p_j+1).
\]
\end{lemma}
\begin{proof}
For $k=0$, $\norm{h_0(x)}[\infty]=\norm{x}[\infty]\le1$, matching the
empty product.  Suppose for some $k\ge1$ that
\[
\sup_{x\in[0,1]^{p_0}}\norm{h_{k-1}(x)}[\infty]
\le\Pi_{k-1},
\qquad
\Pi_{k-1}\coloneqq\prod_{j=0}^{k-2}(p_j+1).
\]
For the pre-activation vector
$z_k\coloneqq W_{k-1}h_{k-1}(x)\in\R^{p_k}$ and
$1\le j\le p_k$,
\[
|(z_k)_j|
=\left|\sum_{l=1}^{p_{k-1}}(W_{k-1})_{jl}(h_{k-1}(x))_l\right|.
\]
The weight bound and induction hypothesis imply
\begin{align}
|(z_k)_j|
&\le\sum_{l=1}^{p_{k-1}}
\underbrace{|(W_{k-1})_{jl}|}_{\le1}
\underbrace{|(h_{k-1}(x))_l|}_{\le\Pi_{k-1}}
\le p_{k-1}\Pi_{k-1}.
\end{align}
Since $h_k(x)=\relu_{v_k}(z_k)$ and $|\relu(u)|\le|u|$,
\begin{equation}
|(h_k(x))_j|
\le|(z_k)_j-(v_k)_j|
\le|(z_k)_j|+|(v_k)_j|.
\end{equation}
Thus $\norm{v_k}[\infty]\le1$ and $\Pi_{k-1}\ge1$ give
\[
\norm{h_k(x)}[\infty]
\le p_{k-1}\Pi_{k-1}+1
\le(p_{k-1}+1)\Pi_{k-1}.
\]
Iterating from the base case yields
\[
\norm{h_k(x)}[\infty]\le\prod_{j=0}^{k-1}(p_j+1).
\]
\end{proof}

\begin{proof}[Proof of Proposition~\ref{lem:inf_norm_bound}]
Let $h_0^{(1)}(x)=h_0^{(2)}(x)=x$.  For $1\le k\le L$, define
$h_k^{(1)}$ and $h_k^{(2)}$ recursively as in
\Cref{lem:layer_output_bound}, using $(\Wone,\vone)$ and
$(\Wtwo,\vtwo)$, respectively.  Set
$\eps\coloneqq\Delta_W\vee\Delta_v$ and
$\Pi_k\coloneqq\prod_{j=0}^{k-1}(p_j+1)$.  Also put
\[
\Delta_k(x)\coloneqq
\norm{h_k^{(1)}(x)-h_k^{(2)}(x)}[\infty].
\]
At the input layer, $\Delta_0(x)=\norm{x-x}[\infty]=0$.  For
$1\le k\le L$,
\[
h_k^{(m)}(x)
=\relu_{v_k^{(m)}}\bigl(W_{k-1}^{(m)}h_{k-1}^{(m)}(x)\bigr),
\qquad m\in\{1,2\}.
\]
The $1$-Lipschitz property of ReLU gives
    \begin{align}
        \Delta_k(x) &= \norm{\relu_{\vone_k}\left(\Wone_{k-1} h_{k-1}^{(1)}(x)\right) - \relu_{\vtwo_k}\left(\Wtwo_{k-1} h_{k-1}^{(2)}(x)\right)}[\infty] \\
        &\leq \norm{\left(\Wone_{k-1} h_{k-1}^{(1)}(x) - \vone_k\right) - \left(\Wtwo_{k-1} h_{k-1}^{(2)}(x) - \vtwo_k\right)}[\infty] \\
        &\leq \norm{\Wone_{k-1} h_{k-1}^{(1)}(x) - \Wtwo_{k-1} h_{k-1}^{(2)}(x)}[\infty] + \norm{\vone_k - \vtwo_k}[\infty].
\end{align}
Adding and subtracting $\Wone_{k-1}h_{k-1}^{(2)}(x)$ gives
\[
\Wone_{k-1}h_{k-1}^{(1)}-\Wtwo_{k-1}h_{k-1}^{(2)}
=\Wone_{k-1}(h_{k-1}^{(1)}-h_{k-1}^{(2)})
+(\Wone_{k-1}-\Wtwo_{k-1})h_{k-1}^{(2)}.
\]
For a $p_{\mathrm{out}}\times p_{\mathrm{in}}$ matrix,
$\norm{Ay}[\infty]\le p_{\mathrm{in}}\norm{A}[\infty]
\norm{y}[\infty]$.  The parameter bounds therefore imply
    \begin{align}
        \Delta_k(x) &\leq p_{k-1} \underbrace{\norm{\Wone_{k-1}}[\infty]}_{\leq 1} \Delta_{k-1}(x) + p_{k-1} \underbrace{\norm{\Wone_{k-1} - \Wtwo_{k-1}}[\infty]}_{\leq \eps} \norm{h_{k-1}^{(2)}(x)}[\infty] + \eps \\
        &\leq p_{k-1} \Delta_{k-1}(x) + \eps \left( p_{k-1} \norm{h_{k-1}^{(2)}(x)}[\infty] + 1 \right).
\end{align}
By \Cref{lem:layer_output_bound},
$\norm{h_{k-1}^{(2)}(x)}[\infty]\le\Pi_{k-1}$ and
$p_{k-1}\Pi_{k-1}+1\le(p_{k-1}+1)\Pi_{k-1}=\Pi_k$.  Hence
\[
\Delta_k(x)\le p_{k-1}\Delta_{k-1}(x)+\eps\Pi_k.
\]
Solving this recurrence from $\Delta_0(x)=0$ yields
\[
\Delta_L(x)
\le\eps\sum_{k=1}^L\Pi_k\left(\prod_{j=k}^{L-1}p_j\right).
\]
For the final layer, $\fone(x)=\Wone_Lh_L^{(1)}(x)$ and
$\ftwo(x)=\Wtwo_Lh_L^{(2)}(x)$, so
\begin{align}
        \norm{\fone(x) - \ftwo(x)}[\infty] &= \norm{\Wone_L h_L^{(1)}(x) - \Wtwo_L h_L^{(2)}(x)}[\infty] \\
        &\leq p_L \norm{\Wone_L}[\infty] \Delta_L(x) + p_L \norm{\Wone_L - \Wtwo_L}[\infty] \norm{h_L^{(2)}(x)}[\infty] \\
        &\leq p_L \Delta_L(x) + \eps p_L \Pi_L.
\end{align}
Since $p_L\Pi_L\le\Pi_{L+1}$, summing the layer contributions gives
\[
\norm{\fone(x)-\ftwo(x)}[\infty]
\le\eps(L+1)\prod_{j=0}^{L}(p_j+1).
\]
Taking the supremum over $x$ proves~\eqref{eq:inf_norm_bound}.
\end{proof}

\subsection{Approximation by ReLU networks}

The integer $n_{\mathrm{SH}}$ is defined in~\eqref{eq:n_sh}.  Its first two
terms enforce the width condition of
\cite[Theorem~5]{schmidthieber_2020};
the third makes $n\mapsto n^{-\beta/(2\beta+d)}(\log n)^{3/2}$ non-increasing
on $[n_{\mathrm{SH}},\infty)$ and guarantees $\log n\ge1$ there.

\begin{lemma}
\label{lem:relu-approx}
Fix $\alpha\in(0,1)$, $\tau\in\{\alo,\ahi\}$, $d\in\N$, $\beta>0$, and
$\lipconst_\beta>0$, and assume
\Cref{assum:smoothness,assum:boundedness}.  For $n\ge n_{\mathrm{SH}}$, let
$\mathcal F_n^{\mathrm{SH}}$ be the sparse ReLU class in
\eqref{eq:NNSHn}.
Then some $\tilde f\in\mathcal F_n^{\mathrm{SH}}$ satisfies
\begin{equation}
\label{eq:relu-approx}
\supnorm{\tilde{f} - \fstar{\tau}} \leq \Capp\, n^{-\beta/(2\beta+d)},
\end{equation}
where
\begin{equation}
\Capp\coloneqq
2(2\lipconst_\beta+1)\lr{1+d^2+\beta^2}\,6^d
+\lipconst_\beta\,3^\beta.
\end{equation}
\end{lemma}

\begin{proof}
By \Cref{assum:smoothness}, $\fstar{\tau} \in \HC^\beta(\X, \lipconst_\beta)$. For $n\ge n_{\mathrm{SH}}$, the first two terms of~\eqref{eq:n_sh} give
\[
N_n=\lceil n^{d/(2\beta+d)}\rceil\ \ge\ n_{\mathrm{SH}}^{d/(2\beta+d)}\ \ge\ (\beta+1)^d \vee (\lipconst_\beta+1)\rme^d.
\]
Applying
\cite[Theorem~5]{schmidthieber_2020} with
width parameter $N_n$ and depth
parameter $\lceil\log_2 n\rceil$, the condition $N_n \geq (\beta+1)^d \vee (\lipconst_\beta+1)\rme^d$ guarantees
the existence of a ReLU network $\tilde{f}^{\mathrm{SH}}$ of depth
$L_{\mathrm{SH}}(n)$, hidden width $W_{\mathrm{SH}}(n)$, sparsity at most
$S_{\mathrm{SH}}(n)$, and all parameters in $[-1,1]$, satisfying
\begin{equation}\label{eq:sh-raw}
\supnorm{\tilde{f}^{\mathrm{SH}} - \fstar{\tau}}
\leq c_1\, N_n\, 2^{-\lceil\log_2 n\rceil} + c_2\, N_n^{-\beta/d},
\end{equation}
with $c_1 = (2\lipconst_\beta+1)\lr{1 + d^2 + \beta^2}\, 6^d$ and $c_2 = \lipconst_\beta \cdot 3^\beta$.

Set $\tilde f\coloneqq\TM\circ\tilde f^{\mathrm{SH}}$.  Membership in
$\NNclassSHn{L_{\mathrm{SH}}}{\bp_{\mathrm{SH}}}{S_{\mathrm{SH}}}(n)$ follows
from~\eqref{eq:sh-class} and the preceding parameter and sparsity bounds.
Projection onto $[-M,M]$ fixes $\fstar{\tau}$ by
\Cref{assum:boundedness} and is nonexpansive; hence, for every $\mathbf x\in\X$,
\begin{equation}
\abs{\tilde f(\mathbf x)-\fstar{\tau}(\mathbf x)}
=\abs{\TM(\tilde f^{\mathrm{SH}}(\mathbf x))
-\TM(\fstar{\tau}(\mathbf x))}
\le\abs{\tilde f^{\mathrm{SH}}(\mathbf x)-\fstar{\tau}(\mathbf x)}.
\end{equation}

Taking suprema and using $2^{-\lceil\log_2 n\rceil}\le n^{-1}$ and
$N_n\le n^{d/(2\beta+d)}+1$, the first term in~\eqref{eq:sh-raw} satisfies
\begin{equation}
N_n\, 2^{-\lceil\log_2 n\rceil}
\leq \lr{n^{d/(2\beta+d)} + 1}\, n^{-1}
\leq 2\, n^{-2\beta/(2\beta+d)}.
\end{equation}
Similarly, for the second term, $N_n^{-\beta/d} \leq n^{-\beta/(2\beta+d)}$. Since
$2\beta/(2\beta+d) > \beta/(2\beta+d)$, the first term in~\eqref{eq:sh-raw} is
dominated by the second for $n \geq 1$. Therefore,
\begin{equation}
\supnorm{\tilde{f} - \fstar{\tau}}
\leq (2c_1 + c_2)\, n^{-\beta/(2\beta+d)}
= \Capp\, n^{-\beta/(2\beta+d)}. \qedhere
\end{equation}

\end{proof}

\subsection{Estimation error}
\label{subsec:excess_risk}
Fix $\alpha\in(0,1)$ and $\tau\in\{\alo,\ahi\}$.
Define the pinball Lipschitz and Bernstein constants by
\begin{equation}
\label{eq:definition-V}
V =\frac{2K_\tau^2}{\low} \quad  \text{and} \quad K_\tau = \max\{\tau,1-\tau\}.
\end{equation}
To separate approximation from estimation, define the best-in-class predictor
over $\Param$ by
\begin{equation}
\label{eq:definition-best_f_inclass}
\bestfinclass= \func{\bestparam} \quad \text{where} \quad
\bestparam\in \argmin_{\param \in \Param} \Risk{\func{\param}},
\end{equation}
and define its approximation error by
\begin{equation}
\label{eq:approximation-error}
\approxerror= \Risk{\bestfinclass}- \Risk{\fstar{\tau}}.
\end{equation}
\begin{proposition}
\label{prop:risk-control}
Fix $\alpha\in(0,1)$ and $\tau\in\{\alo,\ahi\}$.
Assume \Cref{assum:boundedness}, \Cref{assum:density}, and that
$\supnorm{\func{\param}} \le M$ for all $\param \in \Param$.
Assume that the population minimizer in
\eqref{eq:definition-best_f_inclass} exists and that
$\EmpRisk{\func{\param}}$ has a measurable minimizer $\hparam$ over
$\Param$, and set $\fhat{\tau}=\func{\hparam}$.
Assume that, for every $\varepsilon>0$, there is a finite
$\varepsilon$-net $\mathcal{N}_\varepsilon \subset \Param$ such that
\begin{equation}
\label{eq:ass_lemma1}
 \sup_{\param \in \Param} \min_{\vartheta \in \mathcal{N}_\varepsilon}
 \supnorm{\func{\vartheta} - \func{\param}} \le \varepsilon.
\end{equation}
For every $\delta\in(0,1)$ and $\varepsilon>0$, with probability at least
$1-\delta$,
\begin{equation}
\label{eq:risk-control}
    \Risk{\fhat{\tau}}-\Risk{\bestfinclass} \le  \frac{(4V + (8/3)K_\tau M)u_\varepsilon(\delta)}{n} + 4\sqrt{\frac{V(2\approxerror + K_\tau \varepsilon) u_\varepsilon(\delta)}{n}} + 4 K_\tau\varepsilon,
\end{equation}
where $\bestfinclass$ is defined in \eqref{eq:definition-best_f_inclass} and where we have set
\begin{equation}
\label{eq:u-eps}
u_\varepsilon(\delta) = \log\Big(\frac{|\mathcal N_\varepsilon|}{\delta}\Big).
\end{equation}
\end{proposition}

Following \cite[Theorem~1]{puchkin2024rates}, the proof has three steps: a
pointwise Bernstein deviation for a fixed $\func{\param}$, Lipschitz stability
of the empirical and population risks under the sup-norm approximation in
\eqref{eq:ass_lemma1}, and an $\eps$-net union bound.

\begin{lemma}
\label{lemma:quadratic_bounds}
Fix $\alpha\in(0,1)$ and $\tau\in\{\alo,\ahi\}$, and assume
\Cref{assum:boundedness,assum:density}.
For every measurable function $f:\X\to\Y$, the excess pinball risk satisfies
\begin{equation}
\label{eq:global_sc}
    \frac{\low}{2} \Ltwo{f - \fstar{\tau}}^2
    \;\le\;
    \Risk{f} - \Risk{\fstar{\tau}}
    \;\le\;
    \frac{\up}{2} \Ltwo{f - \fstar{\tau}}^2,
\end{equation}
where $\fstar{\tau}$ is the conditional $\tau$-quantile function.
\end{lemma}

In this paper, the lemma is used only in the proof of \Cref{theo:requ_rates},
where \eqref{eq:sh-class} clips every competitor to $\Y=[-M,M]$ and
\Cref{assum:boundedness} places $\fstar{\tau}$ in the same interval, so
\Cref{assum:density} controls the entire segment between them.

\begin{proof}
Fix $x\in\X$ outside a $\DC[X]$-null set on which the conclusions of \Cref{assum:density} hold, and define the conditional risk
\[
  r_x(t) \coloneqq \E\bigl[\pinball(Y-t)\mid X=x\bigr], \qquad t\in\R.
\]
Set
\[
  u \coloneqq f(x), \qquad v \coloneqq \fstar{\tau}(x).
\]
Let
\[
  F_x(t) \coloneqq F_{Y\mid X}(t\mid x), \qquad t\in\R.
\]
By \Cref{assum:density}, the conditional distribution of $Y\mid X=x$ is absolutely continuous on $[-M,M]$ with density $p_{Y\mid X}(\cdot\mid x)$ satisfying
\[
  \low \le p_{Y\mid X}(y\mid x) \le \up, \qquad y\in[-M,M].
\]
Hence $F_x$ is continuous and strictly increasing on $[-M,M]$. Since $v=\fstar{\tau}(x)$ is the conditional $\tau$-quantile, it follows that
\begin{equation}
\label{eq:quadratic_Fxv_tau}
  F_x(v)=\tau.
\end{equation}

By Knight's identity \cite{knight1998limiting}, for every $y,u,v\in\R$,
\begin{equation}
\label{eq:quadratic_knight}
  \pinball(y-u)-\pinball(y-v)
  = (u-v)\bigl(\indiacc{y\le v}-\tau\bigr)
    + \int_v^u \bigl(\indiacc{y\le z}-\indiacc{y\le v}\bigr)\,\rmd z.
\end{equation}
Taking the conditional expectation in \eqref{eq:quadratic_knight} given $X=x$ and using \eqref{eq:quadratic_Fxv_tau}, we obtain
\begin{align}
  r_x(u)-r_x(v)
  &= (u-v)\bigl(F_x(v)-\tau\bigr) + \int_v^u \bigl(F_x(z)-F_x(v)\bigr)\,\rmd z \nonumber\\
  &= \int_v^u \bigl(F_x(z)-F_x(v)\bigr)\,\rmd z.
  \label{eq:quadratic_conditional_excess}
\end{align}

\medskip
\noindent\textit{Case 1: $u\ge v$.}
Since both $u$ and $v$ belong to $[-M,M]$, for every $z\in[v,u]$ we have
\[
  F_x(z)-F_x(v)=\int_v^z p_{Y\mid X}(t\mid x)\,\rmd t.
\]
Using the density bounds from \Cref{assum:density},
\[
  \low (z-v) \le F_x(z)-F_x(v) \le \up (z-v), \qquad z\in[v,u].
\]
Integrating this inequality with respect to $z$ over $[v,u]$ and using \eqref{eq:quadratic_conditional_excess} gives
\[
  \frac{\low}{2}(u-v)^2
  \le r_x(u)-r_x(v)
  \le \frac{\up}{2}(u-v)^2.
\]

\medskip
\noindent\textit{Case 2: $u<v$.}
The opposite deviation is bounded identically: from
\eqref{eq:quadratic_conditional_excess}, $r_x(u)-r_x(v)=\int_u^v
\bigl(F_x(v)-F_x(z)\bigr)\,\rmd z$, and the same density bounds give
\[
  \frac{\low}{2}(v-u)^2
  \le r_x(u)-r_x(v)
  \le \frac{\up}{2}(v-u)^2.
\]

Combining the two cases, for $\DC[X]$-almost every $x$,
\[
  \frac{\low}{2}\bigl(f(x)-\fstar{\tau}(x)\bigr)^2
  \le r_x\bigl(f(x)\bigr)-r_x\bigl(\fstar{\tau}(x)\bigr)
  \le \frac{\up}{2}\bigl(f(x)-\fstar{\tau}(x)\bigr)^2.
\]
Finally, taking expectation with respect to $\DC[X]$ and using the tower property,
\[
  \E\bigl[r_X(f(X))\bigr]=\Risk{f},
  \qquad
  \E\bigl[r_X(\fstar{\tau}(X))\bigr]=\Risk{\fstar{\tau}},
\]
we conclude that
\[
  \frac{\low}{2}\,\Ltwo{f-\fstar{\tau}}^2
  \le \Risk{f}-\Risk{\fstar{\tau}}
  \le \frac{\up}{2}\,\Ltwo{f-\fstar{\tau}}^2.
\]
\end{proof}

For the Bernstein step in \Cref{prop:risk-control}, define the excess loss by
\begin{equation}
\label{eq:ell_excess_loss_def} 
\ell_f(X,Y)\;\coloneqq\;\pinball\bigl(Y-f(X)\bigr)-\pinball\bigl(Y - \fstar{\tau}(X)\bigr)\eqsp.
\end{equation}

On the clipped class, it has envelope $2K_\tau M$; the next lemma gives the
companion second-moment bound $\E[\ell_f(X,Y)^2]\leq
V\E[\ell_f(X,Y)]$ used in \Cref{prop:risk-control}.

\begin{lemma}
\label{lem:bernstein_pinball}
Fix $\alpha\in(0,1)$ and $\tau\in\{\alo,\ahi\}$, and assume
\Cref{assum:boundedness,assum:density}.
For every measurable $f:\X\to\Y$, the following bound holds
\begin{equation}
\label{eq:bernstein_pinball}
\E\bigl[\ell_f(X,Y)^2\bigr] \leq V \E\bigl[\ell_f(X,Y)\bigr] = V (\Risk{f}-\Risk{\fstar{\tau}})\eqsp,
\end{equation}
where the constant $V$ is defined in \eqref{eq:definition-V}.
\end{lemma}
\begin{proof}
The pinball loss is $K_\tau$-Lipschitz, where $K_\tau=\max\{\tau,1-\tau\}$. Therefore, almost surely,
\[
  |\ell_f(X,Y)|
  = \bigl|\pinball\bigl(Y-f(X)\bigr)-\pinball\bigl(Y-\fstar{\tau}(X)\bigr)\bigr|
  \le K_\tau\,|f(X)-\fstar{\tau}(X)|.
\]
Squaring and taking expectation gives
\begin{equation}
\label{eq:pinball_sq_bound}
  \E\bigl[\ell_f(X,Y)^2\bigr]
  \le K_\tau^2\,\Ltwo{f-\fstar{\tau}}^2.
\end{equation}
By the lower bound in \Cref{lemma:quadratic_bounds},
\[
  \Ltwo{f-\fstar{\tau}}^2
  \le \frac{2}{\low}\bigl(\Risk{f}-\Risk{\fstar{\tau}}\bigr).
\]
Substituting this into \eqref{eq:pinball_sq_bound} and using $V=2K_\tau^2/\low$ from \eqref{eq:definition-V},
\[
  \E\bigl[\ell_f(X,Y)^2\bigr]
  \le V\bigl(\Risk{f}-\Risk{\fstar{\tau}}\bigr).
\]
By~\eqref{eq:ell_excess_loss_def},
$\E[\ell_f(X,Y)]=\Risk{f}-\Risk{\fstar{\tau}}$, which gives the equality
in~\eqref{eq:bernstein_pinball}.
\end{proof}

The pointwise Bernstein bound used in \Cref{prop:risk-control} has centered
envelope $4K_\tau M$ and variance bound
$2V\bigl(\Risk{\func{\param}}-\Risk{\bestfinclass}+2\approxerror\bigr)$, with
$\bestfinclass$ as in \eqref{eq:definition-best_f_inclass}.

\begin{lemma}
\label{lem:One_Delta_Bernstein_correct}
Fix $\alpha\in(0,1)$ and $\tau\in\{\alo,\ahi\}$.
Assume \Cref{assum:boundedness}, \Cref{assum:density}, and that
$\supnorm{\func{\param}} \le M$ for all $\param \in \Param$.
Assume that the population minimizer in
\eqref{eq:definition-best_f_inclass} exists.
For every $\param\in\Param$ and $\delta\in(0,1)$, with probability at least
$1-\delta$,
\begin{multline}
\label{eq:One_Delta_Bernstein_correct}
\Risk{\func{\param}}-\Risk{\bestfinclass}
\le
\EmpRisk{\func{\param}}-\EmpRisk{\bestfinclass}
\\+
\sqrt{\frac{4V\bigl(\Risk{\func{\param}}-\Risk{\bestfinclass}+2 \approxerror \bigr)\,\log(1/\delta)}{n}}
+\frac{4K_\tau M\log(1/\delta)}{3n}\eqsp,
\end{multline}
where $\approxerror$ is defined in \eqref{eq:approximation-error}.
\end{lemma}

\begin{proof}
Fix $\param\in\Param$ and define
\[
  Z_i \coloneqq \pinball\bigl(Y_i-\func{\param}(X_i)\bigr)-\pinball\bigl(Y_i-\bestfinclass(X_i)\bigr),
  \qquad i=1,\dots,n.
\]
Then $Z_1,\dots,Z_n$ are i.i.d., and
\begin{equation}
\label{eq:bernstein_emp_mean}
  \frac1n\sum_{i=1}^n Z_i
  = \EmpRisk{\func{\param}}-\EmpRisk{\bestfinclass},
  \qquad
  \E[Z_1]=\Risk{\func{\param}}-\Risk{\bestfinclass}.
\end{equation}

Because $\pinball$ is $K_\tau$-Lipschitz and both $\func{\param}$ and $\bestfinclass$ take values in $[-M,M]$, we have almost surely
\[
  |Z_i|
  \le K_\tau\,\bigl|\func{\param}(X_i)-\bestfinclass(X_i)\bigr|
  \le 2K_\tau M.
\]
Hence
\begin{equation}
\label{eq:bernstein_envelope}
  |Z_i-\E[Z_i]| \le 4K_\tau M \qquad \text{a.s.}
\end{equation}

Next we bound the variance. Write
\[
  A_i \coloneqq \pinball\bigl(Y_i-\func{\param}(X_i)\bigr)-\pinball\bigl(Y_i-\fstar{\tau}(X_i)\bigr),
\]
\[
  B_i \coloneqq \pinball\bigl(Y_i-\bestfinclass(X_i)\bigr)-\pinball\bigl(Y_i-\fstar{\tau}(X_i)\bigr).
\]
Then $Z_i=A_i-B_i$, so by $(a-b)^2\le 2a^2+2b^2$,
\[
  Z_i^2 \le 2A_i^2+2B_i^2.
\]
Taking expectations and applying \Cref{lem:bernstein_pinball} first with $f=\func{\param}$ and then with $f=\bestfinclass$, we get
\begin{align}
  \E[Z_i^2]
  &\le 2\E[A_i^2]+2\E[B_i^2] \\
  &\le 2V\bigl(\Risk{\func{\param}}-\Risk{\fstar{\tau}}\bigr)
      +2V\bigl(\Risk{\bestfinclass}-\Risk{\fstar{\tau}}\bigr) \\
  &= 2V\bigl(\Risk{\func{\param}}-\Risk{\fstar{\tau}}+\approxerror\bigr).
\end{align}
Since
\[
  \Risk{\func{\param}}-\Risk{\fstar{\tau}}
  = \bigl(\Risk{\func{\param}}-\Risk{\bestfinclass}\bigr)
    + \bigl(\Risk{\bestfinclass}-\Risk{\fstar{\tau}}\bigr)
  = \bigl(\Risk{\func{\param}}-\Risk{\bestfinclass}\bigr)+\approxerror,
\]
it follows that
\begin{equation}
\label{eq:bernstein_variance}
  \PVar(Z_i)
  \le \E[Z_i^2]
  \le 2V\bigl(\Risk{\func{\param}}-\Risk{\bestfinclass}+2\approxerror\bigr).
\end{equation}

Apply Bernstein's inequality
\cite[Theorem~2.9 and (2.10), pp.~35--36]{boucheronlugosimassart2013} to the
centered variables
\[
  W_i \coloneqq \E[Z_i]-Z_i.
\]
They are i.i.d., centered, satisfy $|W_i|\le 4K_\tau M$ by \eqref{eq:bernstein_envelope}, and have variance
\[
  \PVar(W_i)=\PVar(Z_i)\le 2V\bigl(\Risk{\func{\param}}-\Risk{\bestfinclass}+2\approxerror\bigr)
\]
by \eqref{eq:bernstein_variance}. Therefore, with probability at least $1-\delta$,
\begin{align}
  \E[Z_1]-\frac1n\sum_{i=1}^n Z_i
  &\le \sqrt{\frac{2\,\PVar(Z_1)\log(1/\delta)}{n}} + \frac{4K_\tau M\log(1/\delta)}{3n} \\
  &\le \sqrt{\frac{4V\bigl(\Risk{\func{\param}}-\Risk{\bestfinclass}+2\approxerror\bigr)\log(1/\delta)}{n}}
      + \frac{4K_\tau M\log(1/\delta)}{3n}.
\end{align}
Substituting the identities in \eqref{eq:bernstein_emp_mean} yields \eqref{eq:One_Delta_Bernstein_correct}.
\end{proof}

\begin{proof}[Proof of Proposition~\ref{prop:risk-control}]
Apply \Cref{lem:One_Delta_Bernstein_correct} at confidence level
$\delta/|\mathcal N_\varepsilon|$. A union bound gives an event of probability
at least $1-\delta$ on which, for every
$\theta'\in\mathcal N_\varepsilon$,
\begin{equation}
\begin{aligned}
\Risk{\func{\theta'}}-\Risk{\bestfinclass}
&\le\EmpRisk{\func{\theta'}}-\EmpRisk{\bestfinclass}\\
&\quad+\sqrt{\frac{4V\bigl(\Risk{\func{\theta'}}-\Risk{\bestfinclass}
+2\approxerror\bigr)u_\varepsilon(\delta)}{n}}\\
&\quad+\frac{4K_\tau M u_\varepsilon(\delta)}{3n}.
\end{aligned}
\end{equation}

Fix $\param\in\Param$. Choose
$\pi(\param)\in\mathcal N_\varepsilon$ such that
$\supnorm{\func{\pi(\param)}-\func{\param}}\le\varepsilon$. The preceding
Bernstein bound holds at $\theta'=\pi(\param)$.

The pinball loss is $K_\tau$-Lipschitz. Hence
\begin{equation}
|\Risk{\func{\param}}-\Risk{\func{\pi(\param)}}|
\le K_\tau\varepsilon,
\qquad
|\EmpRisk{\func{\pi(\param)}}-\EmpRisk{\func{\param}}|
\le K_\tau\varepsilon.
\end{equation}
It follows that
\begin{equation}
\begin{split}
\Risk{\func{\param}}-\Risk{\bestfinclass}
&\le\Risk{\func{\pi(\param)}}-\Risk{\bestfinclass}
+K_\tau\varepsilon,\\
\EmpRisk{\func{\pi(\param)}}-\EmpRisk{\bestfinclass}
&\le\EmpRisk{\func{\param}}-\EmpRisk{\bestfinclass}
+K_\tau\varepsilon.
\end{split}
\end{equation}
Substitute these inequalities into the Bernstein bound. Inside the square
root, use
$\Risk{\func{\pi(\param)}}\le\Risk{\func{\param}}+K_\tau\varepsilon$.
This gives the uniform bound
\begin{multline}
\label{eq:post-net}
\Risk{\func{\param}}-\Risk{\bestfinclass}
\le\EmpRisk{\func{\param}}-\EmpRisk{\bestfinclass}\\
+\sqrt{\frac{4V\bigl(\Risk{\func{\param}}-\Risk{\bestfinclass}
+2\approxerror+K_\tau\varepsilon\bigr)u_\varepsilon(\delta)}{n}}
+\frac{4K_\tau M u_\varepsilon(\delta)}{3n}+2K_\tau\varepsilon.
\end{multline}

Since $\hparam$ minimizes the empirical risk,
$\EmpRisk{\func{\hparam}}-\EmpRisk{\bestfinclass}\le0$. Set
$x=\Risk{\func{\hparam}}-\Risk{\bestfinclass}$. Equation
\eqref{eq:post-net} gives
\begin{equation}
x\le
\sqrt{\frac{4V(x+2\approxerror+K_\tau\varepsilon)
 u_\varepsilon(\delta)}{n}}
+\frac{4K_\tau M u_\varepsilon(\delta)}{3n}+2K_\tau\varepsilon.
\end{equation}

Set
$A=2\sqrt{Vu_\varepsilon(\delta)/n}$,
$B=2\approxerror+K_\tau\varepsilon$, and
$C=4K_\tau M u_\varepsilon(\delta)/(3n)+2K_\tau\varepsilon$. Then
$x\le A\sqrt{x+B}+C$. Use
$\sqrt{x+B}\le\sqrt{x}+\sqrt B$ and
$A\sqrt x\le x/2+A^2/2$. We obtain
\begin{equation}
x\le A^2+2A\sqrt B+2C.
\end{equation}
Substitution gives
\begin{equation}
x\le
\frac{4V u_\varepsilon(\delta)}{n}
+4\sqrt{\frac{V(2\approxerror+K_\tau\varepsilon)
 u_\varepsilon(\delta)}{n}}
+\frac{8K_\tau M u_\varepsilon(\delta)}{3n}
+4K_\tau\varepsilon.
\end{equation}
This proves the proposition.
\end{proof}

\begin{corollary}
\label{cor:risk_coro}
Fix $\alpha\in(0,1)$ and $\tau\in\{\alo,\ahi\}$.
Assume \Cref{assum:boundedness}, \Cref{assum:density}, and that
$\supnorm{\func{\param}}\le M$ for all $\param\in\Param$.
Assume that the population minimizer in
\eqref{eq:definition-best_f_inclass} exists and that
$\EmpRisk{\func{\param}}$ has a measurable minimizer $\hparam$ over
$\Param$, and set $\fhat{\tau}=\func{\hparam}$.
Assume that, for every $\varepsilon>0$, there is a finite
$\varepsilon$-net satisfying \eqref{eq:ass_lemma1}.
Let
$\mathcal N_{1/n}$ be a $(1/n)$-net satisfying \eqref{eq:ass_lemma1} with
$\varepsilon=1/n$, and set
$\mathcal H_n=\log|\mathcal N_{1/n}|$. Then, for every
$\delta\in(0,1)$ and $n\ge1$, with probability at least $1-\delta$,
\begin{equation}
\label{eq:risk-net-coro}
\Risk{\fhat{\tau}}-\Risk{\bestfinclass}
\le K_1\sqrt{\frac{\approxerror(\mathcal H_n+\log(1/\delta))}{n}}
+K_2\frac{\mathcal H_n+\log(1/\delta)}{n}
+K_3\frac1n,
\end{equation}
where $K_1=4\sqrt{2V}$,
$K_2=4V+(8/3)K_\tau M+2$, and
$K_3=4K_\tau+2VK_\tau$.
\end{corollary}

\begin{proof}
Substitute
$u_{1/n}(\delta)=\mathcal H_n+\log(1/\delta)$ and
$\varepsilon=1/n$ in \Cref{prop:risk-control}. With
$A\coloneqq\mathcal H_n+\log(1/\delta)$, this gives
\begin{equation}
\label{eq:risk-control-n}
\Risk{\fhat{\tau}}-\Risk{\bestfinclass}
\le4\sqrt{\frac{V\left(2\approxerror+K_\tau/n\right)A}{n}}
+\left(4V+\frac83K_\tau M\right)\frac An
+\frac{4K_\tau}{n}.
\end{equation}
The square-root term satisfies
\begin{align}
4\sqrt{\frac{V\left(2\approxerror+K_\tau/n\right)A}{n}}
&\le4\sqrt{\frac{2V\approxerror A}{n}}
+4\sqrt{\frac{VK_\tau}{n}\frac An}\\
&\le4\sqrt{2V}\sqrt{\frac{\approxerror A}{n}}
+2\left(\frac{VK_\tau}{n}+\frac An\right)\\
&=4\sqrt{2V}\sqrt{\frac{\approxerror A}{n}}
+\frac{2VK_\tau}{n}+\frac{2A}{n}\eqsp.
\end{align}
Substitution in \eqref{eq:risk-control-n} gives
\eqref{eq:risk-net-coro} with the stated constants.
\end{proof}
\begin{lemma}
\label{lem:entropy_sparse_relu_archi}
Fix $d\in\N$ and $M>0$.
Let $\NNclassSH$ be the sparse-ReLU network class defined
in~\eqref{eq:sh-class}, with depth $L \geq 1$, width vector
$\bp = (d, \WC, \ldots, \WC, 1)$ of constant hidden width
$\WC\geq d$,
integer sparsity budget $s\in\N$, and truncation level $M$. Then, for any
$n \geq 1$,
\begin{equation}\label{eq:entropy-archi}
\begin{aligned}
  \logcover{\tfrac{1}{n}}{\NNclassSH}
  &\;\leq\;
  s \log\!\lr{(2L+1)\,\WC^{2} + 1}
  \\
  &\quad+
  s \log\!\lr{1 + 2n\,(L+1)(\WC+1)^{L+1}}.
\end{aligned}
\end{equation}
\end{lemma}

\begin{proof}
\emph{Parameter support count.}
Every element of $\NNclassSH$ is $\TM \circ f_{\param}$ with $f_{\param}$
of the form~\eqref{eq:nn-def}; embed every such raw network into the
uniform-width ambient architecture
\[
  \bar{\bp} \;=\; \lr{d,\,\underbrace{\WC,\dots,\WC}_{L\text{ hidden layers}},\,1},
\]
which contains $\bp$. Any such raw network $f_{\param}$ is realized by a
parameter vector $\param \in [-1,1]^P$ with $\lzeronorm{\param} \leq s$,
where
\[
  P \;=\; \WC d + (L-1)\WC^{2} + \WC + L\WC \;\leq\; (2L+1)\,\WC^{2},
\]
using $d \leq \WC$ and $\WC \geq 1$. Denote by $\widetilde{\Param}
\subset [-1,1]^P$ the set of ambient parameter vectors realizing the raw networks
underlying $\NNclassSH$. For each $T \subset \lrcb{1,\dots,P}$ with
$\abs{T} \leq s$, let $\widetilde{\Param}_T \coloneqq \lrcb{\param \in
\widetilde{\Param} : \supp(\param) \subset T}$. Since
$\widetilde{\Param} = \bigcup_{T:\abs{T}\leq s} \widetilde{\Param}_T$,
the number of admissible supports is at most $(P+1)^s$: if $s\le P$, then
$\sum_{k=0}^{s}\binom{P}{k}\le\sum_{k=0}^{s}\binom{s}{k}P^k=(P+1)^s$;
if $s>P$, it is $2^P\le(P+1)^s$.

\emph{Quantization.}
Fix $\eps > 0$, set $C_{\mathrm{Lip}} \coloneqq (L+1)(\WC+1)^{L+1}$ and
$\eta \coloneqq \eps/C_{\mathrm{Lip}}$, and fix an admissible support $T$.
Partitioning $[-1,1]^T$ into axis-parallel cells of side $\eta$ yields at most
\[(1 + 2/\eta)^{\abs{T}} = (1 + 2C_{\mathrm{Lip}}/\eps)^{\abs{T}}
\]
cells. For each cell $Q$ with $Q \cap \widetilde{\Param}_T \neq \emptyset$,
pick one representative $\param_Q$; every $\param \in Q \cap
\widetilde{\Param}_T$ satisfies $\supnorm{\param-\param_Q}\leq\eta$.

\emph{Network-output stability.}
For $\param, \param' \in \widetilde{\Param}$, \Cref{lem:inf_norm_bound}
applied to $\bar{\bp}$ gives
\[
  \supnorm{f_{\param} - f_{\param'}}
  \;\leq\;
  \supnorm{\param - \param'}\,(L+1)\prod_{j=0}^{L}(\WC + 1)
  \;\leq\;
  C_{\mathrm{Lip}}\,\supnorm{\param - \param'},
\]
where the last inequality uses $d \leq \WC$. Thus every $\param \in Q \cap
\widetilde{\Param}_T$ satisfies $\supnorm{f_{\param} - f_{\param_Q}} \leq
C_{\mathrm{Lip}}\eta = \eps$, hence $\supnorm{\TM \circ f_{\param} -
\TM \circ f_{\param_Q}} \leq \eps$ because $\TM$ is $1$-Lipschitz. Therefore,
for each fixed support~$T$,
\[
  \cover{\eps}{\lrcb{\TM \circ f_{\param}: \param \in \widetilde{\Param}_T}}
  \;\leq\;
  \lr{1 + \tfrac{2C_{\mathrm{Lip}}}{\eps}}^{s}.
\]

\emph{Covering-number multiplication.}
Multiplying the number of supports by the preceding per-support bound gives
\[
  \cover{\eps}{\NNclassSH}
  \;\leq\;
  (P+1)^{s}\lr{1 + \tfrac{2C_{\mathrm{Lip}}}{\eps}}^{s}.
\]
Taking logs, substituting the bounds on $P$ and $C_{\mathrm{Lip}}$, and
setting $\eps = 1/n$ yields~\eqref{eq:entropy-archi}.
\end{proof}

\begin{corollary}
\label{cor:entropy_SH}
Fix $d\in\N$, $\beta>0$, and $M>0$.
Let $N_n$ and the Schmidt--Hieber architectural parameters be defined
by~\eqref{eq:sh-scaling}--\eqref{eq:sh-arch}.
Consider the class
\[
\NNclassSHn{L_{\mathrm{SH}}}{\bp_{\mathrm{SH}}}{S_{\mathrm{SH}}}(n).
\]
Define
\[
  C_L^\star := 15\lr{1 + \lceil \log_2(d \vee \beta)\rceil},
  \quad
  C_W^\star := 12(d + \lceil \beta \rceil),
  \quad
  C_S^\star := 4512\,(d + \beta + 1)^{3+d},
\]
so that, for every $n \geq 2$,
\begin{equation}\label{eq:SH-bounds}
  L_{\mathrm{SH}}(n) \leq C_L^\star \log_2 n,
  \quad
  W_{\mathrm{SH}}(n) \leq C_W^\star\, n^{d/(2\beta+d)},
  \quad
  S_{\mathrm{SH}}(n) \leq C_S^\star\, n^{d/(2\beta+d)} \log_2 n.
\end{equation}
Then there exists a constant $C_{\mathrm{ent}} > 0$, depending only on $d$
and $\beta$, such that for every $n \geq 2$,
\begin{equation}\label{eq:entropy-SH}
  \logcover{\tfrac{1}{n}}{\NNclassSHn{L_{\mathrm{SH}}}{\bp_{\mathrm{SH}}}{S_{\mathrm{SH}}}(n)}
  \;\leq\;
  C_{\mathrm{ent}}\, n^{d/(2\beta+d)}\,(\log_2 n)^{3}.
\end{equation}
\end{corollary}

\begin{proof}
Set $\gamma := d/(2\beta+d) \in (0,1)$ and fix $n \geq 2$, so that
$\log_2 n \geq 1$ and $\log n \leq \log_2 n$. The
bounds~\eqref{eq:SH-bounds} follow from $N_n \leq 2 n^{\gamma}$,
$\lceil\log_2 n\rceil + 5 \leq 7 \log_2 n$,
$\lceil\log_2 n\rceil + 6 \leq 8 \log_2 n$, and
$\lceil z\rceil\le2z$ for $z\ge1$, substituted
into~\eqref{eq:sh-L}--\eqref{eq:sh-S}; the constant $C_L^\star$ absorbs
the additive $8$ in~\eqref{eq:sh-L} via $\log_2 n \geq 1$.

Applying \Cref{lem:entropy_sparse_relu_archi} with $\WC = W_{\mathrm{SH}}(n)$,
\begin{equation}\label{eq:cor-proof-start}
\begin{aligned}
  \logcover{\tfrac{1}{n}}{\NNclassSHn{L_{\mathrm{SH}}}{\bp_{\mathrm{SH}}}{S_{\mathrm{SH}}}(n)}
  &\;\leq\;
  s \log\!\lr{(2L+1)\WC^{2} + 1}
  \\
  &\quad+
  s \log\!\lr{1 + 2n\,(L+1)(\WC+1)^{L+1}}.
\end{aligned}
\end{equation}

 Using $(2L+1)\WC^2 + 1 \leq 4 L\WC^2$
and~\eqref{eq:SH-bounds},
\[
  \log\!\lr{(2L+1)\WC^2 + 1}
  \leq \log\!\lr{4 C_L^\star (C_W^\star)^{2}}
      + \log\log_2 n + 2\gamma \log n
  \leq c_{1} \log_2 n,
\]
with $c_{1} := \log\!\lr{4 C_L^\star (C_W^\star)^{2}} + 3$.

From $1 + 2n(L+1)(\WC+1)^{L+1} \leq 4n(L+1)(\WC+1)^{L+1}$,
$L+1 \leq (C_L^\star + 1)\log_2 n$, and
$\log(\WC+1) \leq \log(C_W^\star+1) + \log_2 n$,
\begin{align}
  \log\!\lr{1 + 2n(L+1)(\WC+1)^{L+1}}
  &\leq \log 4 + \log n + \log(L+1) + (L+1)\log(\WC+1) \\
  &\leq c_{2}\,(\log_2 n)^{2},
\end{align}
with $c_{2} := \log 4 + 2 + \log(C_L^\star + 1)
+ (C_L^\star + 1)\lr{\log(C_W^\star + 1) + 1}$.

Multiplying both terms of~\eqref{eq:cor-proof-start}
by $s \leq C_S^\star\, n^{\gamma} \log_2 n$ and using
$\log_2 n \leq (\log_2 n)^{2}$,
\[
\begin{aligned}
  \logcover{\tfrac{1}{n}}{\NNclassSHn{L_{\mathrm{SH}}}{\bp_{\mathrm{SH}}}{S_{\mathrm{SH}}}(n)}
  &\;\leq\;
  C_S^\star(c_{1} + c_{2})\, n^{\gamma}\,(\log_2 n)^{3}
  \\
  &\;=:\;
  C_{\mathrm{ent}}\, n^{d/(2\beta+d)}\,(\log_2 n)^{3},
\end{aligned}
\]
with
\begin{equation}
\begin{aligned}
C_{\mathrm{ent}}
&:= C_S^\star\Bigl(
\log\!\lr{4 C_L^\star (C_W^\star)^{2}} + 5 + \log(C_L^\star + 1)\\
&\qquad+ (C_L^\star + 1)\lr{\log(C_W^\star + 1) + 1}
+ \log 4\Bigr).
\end{aligned}
\end{equation}

Explicitly,
\begin{equation}
\label{eq:Cent}
\begin{aligned}
    C_{\mathrm{ent}} &= 4512\,(d+\beta+1)^{3+d} \bigg( \log\!\lr{8640\lr{1 + \lceil \log_2(d \vee \beta)\rceil}(d + \lceil \beta \rceil)^{2}} \\
    &\qquad + 5 + \log\!\lr{16 + 15\lceil \log_2(d \vee \beta)\rceil} \\
    &\qquad + \lr{16 + 15\lceil \log_2(d \vee \beta)\rceil}\lr{\log\!\lr{12(d + \lceil \beta \rceil) + 1} + 1} + \log 4 \bigg).
\end{aligned}
\end{equation}
\end{proof}

\subsection{Proof of~Theorem~\ref{theo:requ_rates}}
\label{sec:theo_requ_rates}

\begin{proof}
The proof combines the approximation bound \Cref{lem:relu-approx}, the entropy
bound \Cref{cor:entropy_SH}, and the oracle inequality \Cref{cor:risk_coro}.
We use $\widetilde f_{n,\tau}$ for the raw empirical minimizer throughout.

Fix $\tau\in\{\alo,\ahi\}$ and let
\[
    \NNclass \coloneqq \NNclassSHn{L_{\mathrm{SH}}(n)}{\bp_{\mathrm{SH}}(n)}{S_{\mathrm{SH}}(n)}
\]
denote the Schmidt--Hieber network class with architecture parameters chosen
as functions of $n$.  Let $\widetilde f_{n,\tau}$ be the raw empirical
minimizer of \Cref{theo:requ_rates}.

\paragraph{Step 1: approximation error}
By \Cref{assum:smoothness,lem:relu-approx}, some
$g_{\tau,n}\in\NNclass$ satisfies

\begin{equation}
    \label{eq:relu_approx}
    \supnorm{g_{\tau,n}-\fstar{\tau}} \leq \Capp\, n^{-\beta/(2\beta+d)}.
\end{equation}

The parameter set of $\NNclass$ is a finite union of closed subsets of a
finite-dimensional cube and is therefore compact.  The parameter-stability
bound of \Cref{lem:inf_norm_bound}, output clipping, and Lipschitz continuity
of the pinball loss make the population risk continuous on this set.  Hence a
population minimizer $\bestfinclass\in\arg\min_{f\in\NNclass}\Risk{f}$ exists
(see~\eqref{eq:definition-best_f_inclass}).  Both $\bestfinclass$ and
$g_{\tau,n}$ depend on $n$ through $\NNclass$, but we suppress this dependence.
For every training sample, the empirical pinball risk is continuous in the
parameter on the same compact set, so its argmin is nonempty and compact. The
empirical risk is jointly measurable in the sample and the parameter. Since
$([0,1]^d\times[-M,M])^n$ is a standard Borel space, the measurable maximum
theorem~\cite[Theorem~18.19]{aliprantisborder2006} yields a measurable exact
empirical-risk minimizer. This justifies the measurable ERM used in
\Cref{theo:requ_rates}.
By minimality, $\Risk{\bestfinclass}\le\Risk{g_{\tau,n}}$. Therefore,

\begin{equation}
    \label{eq:requ_bias_risk}
    \begin{aligned}
        \approxerror &\coloneqq \Risk{\bestfinclass} - \Risk{\fstar{\tau}}
        \le \Risk{g_{\tau,n}} - \Risk{\fstar{\tau}} \\
        &\qquad\le \frac{\up}{2}\Ltwo{g_{\tau,n}-\fstar{\tau}}^2
        \le \frac{\up}{2}\supnorm{g_{\tau,n}-\fstar{\tau}}^2
        \le \frac{\up\Capp^2}{2}\, n^{-\frac{2\beta}{2\beta+d}}.
    \end{aligned}
\end{equation}

The second inequality uses the upper bound in~\eqref{eq:global_sc}; the last
uses~\eqref{eq:relu_approx} and $\Ltwo{h}\le\supnorm{h}$.

\paragraph{Step 2: metric entropy}
Choose a finite $1/n$-net $\mathcal N_{1/n}$ of $\NNclass$ of minimum
cardinality in the uniform norm and fix one parameter representative for each
network. We use the same symbol for the resulting subset of the parameter set.
By \Cref{cor:entropy_SH} and the change-of-base formula
$\log_2 n=(\log n)/(\log 2)$,
\begin{equation}
    \label{eq:entropy_bound}
    \mathcal{H}_n \coloneqq \log|\mathcal N_{1/n}|
    =\logcover{\frac{1}{n}}{\NNclass}
    \le \bar C_{\mathrm{ent}}\, n^{\frac{d}{2\beta + d}} (\log n)^3 \eqsp,
\end{equation}
where $\bar C_{\mathrm{ent}} \coloneqq C_{\mathrm{ent}}/(\log 2)^3$ and $C_{\mathrm{ent}}$ is the constant of \eqref{eq:Cent}.

\paragraph{Step 3: rate assembly}
The boundedness hypothesis of \Cref{cor:risk_coro} holds because
\eqref{eq:sh-class} clips every output to $[-M,M]$.  Define
\[
    \mathcal{E}_n(\widetilde f_{n,\tau})
    \coloneqq \Risk{\widetilde f_{n,\tau}}-\Risk{\fstar{\tau}}.
\]
Decomposing this excess risk as
$\mathcal{E}_n(\widetilde f_{n,\tau})
=\bigl(\Risk{\widetilde f_{n,\tau}}-\Risk{\bestfinclass}\bigr)
+\approxerror$, \Cref{cor:risk_coro} yields, with probability at least
$1-\delta$,
\begin{equation}
\label{eq:excess_risk_oracle}
    \mathcal{E}_n(\widetilde f_{n,\tau}) \le \approxerror + K_1 \sqrt{\frac{\approxerror \mathcal{H}_n}{n}} + K_1 \sqrt{\frac{\approxerror \log(1/\delta)}{n}} + K_2 \frac{\mathcal{H}_n}{n} + K_2 \frac{\log(1/\delta)}{n} + K_3 \frac{1}{n} \eqsp,
\end{equation}
where $K_1, K_2, K_3$ are defined in \Cref{cor:risk_coro}.

The bound
$K_1\sqrt{\approxerror\log(1/\delta)/n}
\le\approxerror+(K_1^2/4)\log(1/\delta)/n$
turns~\eqref{eq:excess_risk_oracle} into
\begin{equation}
\label{eq:excess_risk_decoupled}
    \mathcal{E}_n(\widetilde f_{n,\tau}) \le 2\approxerror + K_1 \sqrt{\frac{\approxerror \mathcal{H}_n}{n}} + K_2 \frac{\mathcal{H}_n}{n} + \left(K_2 + \frac{K_1^2}{4}\right) \frac{\log(1/\delta)}{n} + \frac{K_3}{n}\eqsp.
\end{equation}

Set $C_{\mathrm{bias}}\coloneqq\up\Capp^2/2$.  Equations
\eqref{eq:requ_bias_risk} and~\eqref{eq:entropy_bound} give
\[
\approxerror\le C_{\mathrm{bias}}n^{-2\beta/(2\beta+d)},
\qquad
\frac{\mathcal{H}_n}{n}\le\bar C_{\mathrm{ent}}
n^{-2\beta/(2\beta+d)}(\log n)^3.
\]
Consequently, the cross-term satisfies
\[
\sqrt{\frac{\approxerror\mathcal{H}_n}{n}}
\le\sqrt{C_{\mathrm{bias}}\bar C_{\mathrm{ent}}}
n^{-2\beta/(2\beta+d)}(\log n)^{3/2}.
\]

For $n\ge3$, $(\log n)^{3/2}\le(\log n)^3$ and
$n^{-1}\le n^{-2\beta/(2\beta+d)}(\log n)^3$.  Absorbing these terms gives
\begin{equation}
\label{eq:relu_excess_risk_rate}
\begin{aligned}
\mathcal{E}_n(\widetilde f_{n,\tau})
&\le
\underbrace{\left( 2 C_{\mathrm{bias}}
+ K_1 \sqrt{C_{\mathrm{bias}} \bar C_{\mathrm{ent}}}
+ K_2 \bar C_{\mathrm{ent}} + K_3 \right)}_{\eqqcolon\, C_{\mathrm{risk}}}\\
&\qquad\times n^{-\frac{2\beta}{2\beta+d}} (\log n)^3
+ \underbrace{\left( K_2 + \frac{K_1^2}{4} \right)}_{\eqqcolon\, C_{\mathrm{conf}}'}
\frac{\log(1/\delta)}{n} \eqsp.
\end{aligned}
\end{equation}

Finally, the lower bound in~\eqref{eq:global_sc} and
\eqref{eq:relu_excess_risk_rate} give
\[
\Ltwo{\widetilde f_{n,\tau}-\fstar{\tau}}^2
\le\frac{2}{\low}\mathcal{E}_n(\widetilde f_{n,\tau}).
\]
Taking square roots and using subadditivity yields
\begin{equation}
    \Ltwo{\widetilde f_{n,\tau}-\fstar{\tau}}
    \le \underbrace{\sqrt{\frac{2\,C_{\mathrm{risk}}}{\low}}}_{=\, C_{\mathrm{rate}}}\; n^{-\frac{\beta}{2\beta+d}} (\log n)^{3/2}
    \;+\; \underbrace{\sqrt{\frac{2\,C_{\mathrm{conf}}'}{\low}}}_{=\, C_{\mathrm{conf}}}\; \sqrt{\frac{\log(1/\delta)}{n}} \eqsp.
\end{equation}
\end{proof}

\subsection{Verification of the CQR assumptions}
\label{subsec:relu_endpoint_condition}

Under \Cref{assum:boundedness,assum:density}, the upper density bound gives
atomlessness and the all-interval mass bound in \Cref{assum:mass_regular}.
Each side of either oracle quantile has probability at least $\alpha/2$,
so both quantiles lie at distance at least $\alpha/(2\up)$ from both support
boundaries. The lower density bound then gives all four local inequalities
in \Cref{assum:mass_regular} with $r_0=\alpha/(2\up)$.

For $u\in(0,1)$, define
\begin{equation}
\label{eq:relu_epsilon_main}
\epsilon_2(n,u)
\coloneqq
\begin{cases}
2M, & n<n_{\mathrm{SH}},\\
\displaystyle
\min\!\left\{2M,
C_{\mathrm{rate}}n^{-\beta/(2\beta+d)}(\log n)^{3/2}
+C_{\mathrm{conf}}\sqrt{\frac{\log(1/u)}n}\right\},
& n\ge n_{\mathrm{SH}}.
\end{cases}
\end{equation}
For either $\tau\in\{\alo,\ahi\}$ and $n\ge n_{\mathrm{SH}}$,
\Cref{theo:requ_rates} at confidence $u$ gives the stochastic bound in the
second branch. The deterministic $2M$ bound holds for every $n$ because
the raw minimizer and the oracle quantile take values in $[-M,M]$.
Thus \Cref{assum:HPD-excess-risk} holds at $p=2$ with this radius.
Pointwise sorting as in \eqref{eq:raw_endpoint_sorting} preserves the output
range; \Cref{lem:sorting_contraction_mass} gives the contraction used by
the generic results.

\subsection{Proof of~Corollary~\ref{cor:relu_rates}}
\label{app:length_mismatch_l2}

\begin{proof}
\Cref{subsec:relu_endpoint_condition} verifies the CQR assumptions.
Under the localization conditions assumed in \Cref{cor:relu_rates}, apply
\Cref{prop:target,th:global_cov_bound} at $p=2$.  On the common event
$\mathcal A_{n,m}(\delta)$, which has probability at least $1-\delta$,
substitution of \eqref{eq:relu_epsilon_main} at $u=\delta/4$ in
\eqref{eq:length_mismatch_bound} and~\eqref{eq:cond_cov}
yields~\eqref{eq:relu_rate}.  Its hidden constant depends only on
$(\alpha,\low,\up,\beta,d,M,\lipconst_\beta)$.
\end{proof}

\subsection{Proof of~Corollary~\ref{cor:cs_relu}}
\label{subsec:cs_corollary}

The specialization uses the rearranged endpoints of
\eqref{eq:raw_endpoint_sorting} and the source $L^2$ rate of
\Cref{theo:requ_rates}.

\begin{proof}[Proof of Corollary~\ref{cor:cs_relu}]
The argument in \Cref{subsec:relu_endpoint_condition} verifies
the $p=2$ instance of \Cref{assum:HPD-excess-risk} for the raw endpoints with
$\epsilon_2$
of~\eqref{eq:relu_epsilon_main}; pointwise sorting gives the rearranged fitted
pair. At
confidence $u=\delta/4$ and for $n\ge n_{\mathrm{SH}}$, this error is at most
\[
C_{\mathrm{rate}}n^{-\beta/(2\beta+d)}(\log n)^{3/2}
+C_{\mathrm{conf}}\sqrt{\frac{\log(4/\delta)}n}.
\]
Under the localization conditions assumed in \Cref{cor:cs_relu}, apply
\Cref{prop:cs_length,th:cs_cond_cov} at $p=2$ on
$\mathcal A^w_{n,m}(\delta)$, the common event defined in
\Cref{lem:cs_calibration_bound}, and substitute this rate
in~\eqref{eq:cs_length_complete} and~\eqref{eq:cs_cond_cov_complete}.
Change of measure and $0\le\wsh\le\wmax$ give
\[
\Ltwo{\wsh}[{\QC[X]}]^{2}
=\E_{\DC[X]}[\wsh^{3}]
\le\wmax\E_{\DC[X]}[\wsh^{2}]
=\wmax\,\csdiv.
\]
Substitution into the exact shifted bounds gives the endpoint term, the
random-normalization term $(2-\alpha)\etaw{\delta}$, and the test-atom term
$(1-\alpha)m^{-1}\sqrt{\wmax\,\csdiv}$ in
\eqref{eq:cs_relu_rate}.
\end{proof}

\bibliography{refs}

\end{document}